\documentclass[11pt, a4]{article}
\usepackage[utf8]{inputenc}
\usepackage[full]{textcomp}
\usepackage[osf]{newtxtext}
\usepackage{comment}

\usepackage{amssymb}
\usepackage{mathtools}
\usepackage{hyperref}
\usepackage{breakurl}
\usepackage{mhenvs}
\usepackage{mhequ}
\usepackage{mhsymb}
\usepackage{booktabs}
\usepackage{tikz}
\usepackage{mathrsfs}
\usepackage{longtable}
\usepackage{times}

\newtheorem*{lemma**}{Lemma}
\newtheorem*{theorem**}{Theorem}

\usepackage{microtype}
\usepackage{wasysym}
\usepackage{centernot}
\usepackage{enumitem}

\numberwithin{equation}{section}

\makeatletter
\newcommand{\globalcolor}[1]{%
  \color{#1}\global\let\default@color\current@color
}
\makeatother

\newif\ifdark
\IfFileExists{dark}{\darktrue}{\darkfalse}

\ifdark

\definecolor{darkred}{rgb}{0.9,0.2,0.2}
\definecolor{darkblue}{rgb}{0.7,0.3,1}
\definecolor{darkgreen}{rgb}{0.1,0.9,0.1}
\definecolor{pagebackground}{rgb}{.15,.21,.18}
\definecolor{pageforeground}{rgb}{.84,.84,.85}
\pagecolor{pagebackground}
\AtBeginDocument{\globalcolor{pageforeground}}

\else

\definecolor{darkred}{rgb}{0.7,0.1,0.1}
\definecolor{darkblue}{rgb}{0.4,0.1,0.8}
\definecolor{darkgreen}{rgb}{0.1,0.7,0.1}
\definecolor{pagebackground}{rgb}{1,1,1}
\definecolor{pageforeground}{rgb}{0,0,0}

\fi

\theoremnumbering{Alph}
\renewtheorem{theorem*}{Theorem}

\DeclareMathAlphabet{\mathbbm}{U}{bbm}{m}{n}
\DeclareFontFamily{U}{BOONDOX-calo}{\skewchar\font=45 }
\DeclareFontShape{U}{BOONDOX-calo}{m}{n}{
  <-> s*[1.05] BOONDOX-r-calo}{}
\DeclareFontShape{U}{BOONDOX-calo}{b}{n}{
  <-> s*[1.05] BOONDOX-b-calo}{}
\DeclareMathAlphabet{\mcb}{U}{BOONDOX-calo}{m}{n}
\SetMathAlphabet{\mcb}{bold}{U}{BOONDOX-calo}{b}{n}

\let\epsilon\varepsilon
\def\E{{\symb E}}
\def\F{{\mathcal F}}
\def\H{{\mathscr H}}
\def\FC{\mathscr{C}}
\def\L{\mathbb L}

\def\X{\mathbb{X}}
\def\W{\mathbf{W}}
\def\WW{{ \mathbb W}}
\def\X{{\mathbf X}}
\def\XX{{\mathbb X}}

\def\Y{{\mathbf Y}}
\def\YY{{\mathbb Y}}

\def\ZZ{{\mathbb Z}}

\def\D{{\mathcal D}}
\def\err{\mathbf {Er}}

\def\C{\mathcal{C}}
\def\f{\frac}
\def\1{\mathbf{1}}
\def\lf{{\lfloor}}
\def\rf{{\rfloor}}

\def\Hess{{\mathrm Hess}}

\def\cov{{\mathrm{Cov}}}
\def\${|\!|\!|}

\def\<{\langle}
\def\>{\rangle}
\setlist{noitemsep,topsep=4pt}
\def\para_#1{/\!\!/_{\!#1}}
\def\var{\mathrm{var}}

\def\slash{\kern0.18em/\penalty\exhyphenpenalty\kern0.18em}
\def\dash{\kern0.18em--\penalty\exhyphenpenalty\kern0.18em}

\makeatletter % Stolen from the internet to make a fat \cdot which isn't as fat as a \bullet
\newcommand*{\fat}{}% Check if undefined
\DeclareRobustCommand*{\fat}{%
\mathbin{\mathpalette\bigcdot@{}}}
\newcommand*{\bigcdot@scalefactor}{.5}
\newcommand*{\bigcdot@widthfactor}{1.15}
\newcommand*{\bigcdot@}[2]{%
  \sbox0{$#1\vcenter{}$}% math axis
  \sbox2{$#1\cdot\m@th$}%
  \hbox to \bigcdot@widthfactor\wd2{%
    \hfil
    \raise\ht0\hbox{%
      \scalebox{\bigcdot@scalefactor}{%
        \lower\ht0\hbox{$#1\bullet\m@th$}%
      }%
    }%
    \hfil
  }%
}
\makeatother
{\theorembodyfont{\rmfamily}

\newtheorem{convention}[lemma]{Convention}
}

\newtheorem{assumption}[lemma]{Assumption}

\begin{document}
\title{Homogenization with fractional random fields}
\author{Johann Gehringer and Xue-Mei~Li\\ Imperial College London}
\date{27 November 2019}

\maketitle

\begin{abstract}
We consider a system of  differential equations in a fast long range dependent random environment
and  prove a homogenization theorem involving multiple scaling constants. The effective dynamics solves a rough differential equation,
which is `equivalent' to a stochastic equation driven by  mixed It\^o integrals and Young integrals with respect to
Wiener processes and Hermite processes. Lacking other tools we use the rough path theory for proving the convergence, 
our main technical endeavour  is on obtaining an enhanced  scaling limit theorem  for path integrals
(Functional CLT and non-CLT's) in a strong topology, the rough path topology, which is given by a  H\"older distance for 
stochastic processes and their lifts.
In  dimension one we also  include the negatively correlated case, for the second order / kinetic fractional BM model
we also bound the error.
  \end{abstract}

{\scriptsize \textit{Keywords:}  fractional Brownian motion, slow-fast systems, homogenization, Hermite processes}, 
{\scriptsize  passive tracer, random environment, functional limit theorem, rough path}
	
{\scriptsize \textit{MSC Subject classification:} 34F05, 60F05, 60F17, 60G18, 60G22, 60H05,  60H07, 60H10}

\setcounter{tocdepth}{2}
\tableofcontents
\section{Introduction}

In this article we prove a homogenization theorem to the following slow/fast system with  long range dependent random environment and multiple scaling constants, \begin{equation}\label{1}
\dot x_t^\epsilon= \sum_{k=1}^N \alpha_k(\epsilon) f_k( x_t^\epsilon) G_k( y_t^\epsilon),
\end{equation}
 showing that the solutions converge.
Here  $\epsilon$ is a small positive parameter,  $y_t$  models a stationary long range dependent fast random environment, $G_k\in L^p\cap \C^0$ are centred (not necessarily in a finite chaos), and  $\alpha_k(\epsilon)$ are the scaling constants to be identified.  
  The $x_t^\epsilon$ process models the position of a particle in a moving environment, $y^\epsilon$ is a stationary stochastic process moving at microscopic time scale.   When $f$ is divergence free, this is a popular model for passive tracers in a tubulent fluid.
  By homogenization we mean the following phenomenon:  during a finite macroscopic period,   the fast environment would have typically been everywhere, its effects can therefore be absorbed  into one effective vector field. This way one  obtains an autonomous equation whose solution approximate the position of the particle when the parameter $\epsilon$ is small.

 Noise with long span of  interdependence between their increments has  attracted the  attention of many mathematicians and physicists. 
In a study for loss in water storage, Hurst et al \cite{Hurst} observed long range time dependence in the time series data of water flows
and found that the time dependence varies proportionally to $t^H$ where $H\sim 0.73$.  Economical  data also exhibits cycles of varying lengths. 
By contrast, Brownian motions and stable processes have independent increments.   Benoit Mandelbrot and John Van Ness introduced the use
of  fractional Brownian motions (fBM) in \cite{Mandelbrot-VanNess} and found they are good models for the Hurst phenomenon, and the best among other models they compare with. Recall that a fBM
is a continuous Gaussian process with $\E(B_t-B_s)^2=|t-s|^{2H}$. (When $H=\f 12$, this is the BM.) They are self-similar with similarity exponent $H$.

Self-similarity attracted attention also from Sinai, Dobrushin, and Jona-Lasinio  for their relevance in  mathematically rigorous description of critical phenomena and in the renormalisation theory.  In \cite{Sinai}, for example, Sinai  constructed non-Gaussian self-similar fields; 
while Dobrushin \cite{Dobrushin} studied self-similar fields subordinated to self-similar Gaussian fields (multiple It\^o integrals).  
Those self-similar stochastic processes with stationary increments are a particular interesting class. When normalized  to begin at $0$, to have mean $0$ and variance $1$, at $t=1$, they necessarily have the covariance  $\f 12(t^{2H}+s^{2H}-|t-s|^{2H})$. Those of Gaussian variety are fBMs.  Hermite processes are non-Gaussian self-similar processes with  the above mentioned covariance and stationary increments.
They appear as scaling limits of functionals of long range dependent Gaussian processes.  The first of these appeared in \cite{Rosenblatt}, in which Rosenblatt
constructed an example of a non-strong mixing sequence of random variables. He proved that the afore-mentioned sequence (with slow decaying auto-correlation) is not strong mixing  by proving that the usual central limit theorem (CLT)  fails and obtained a non-Gaussian scaling limit which is in fact a rank $2$ Hermite process. Rank $2$ processes are called Rosenblatt process. Jona-Lasinio  was also concerned with the construction of a systematic theory of limit distributions for sums of `strongly dependent' random variables for which the classical central limit theorems does not hold, \cite{Jona-Lasinio}, see also the book \cite{Embrechts-Maejima}. These processes also appear in our effective dynamics, in a mixed manner.

Despite of the evidence pointing to long range depend noise,  the study of slow/fast systems has predominatedly focused on those with strongly mixing  or Markovian properties.  If the correlation of the random field decays sufficiently fast,  see \cite{Taylor,Green,Kubo}, for small parameters the particle is expected to behave diffusively and can be approximated by a Markov process with covariance given by the integral, if finite, of the correlation functions of the vector field.   If the correlation decays so slowly that  the  integral is infinite, the random field is said to have long range dependence.

We will take $y_t^\epsilon$ to be a fast fractional Ornstein-Uhlenbeck process (fOU). These are defined by the Langevin equation driven by fBM's, and also have long range dependence when the Hurst parameter $H>\f 12$, see \S\ref{OU-section}.  These are fascinating processes. On one hand  the solutions of the Langevin equation forgets its initial position exponentially fast. On the other hand its auto-correlation function, which measures how much the shifted process remembers, exhibits power law decay. The latter is  not  shared by all other functionals of fBMs.  For example  it was shown in \cite{Komorowski-Komorowski-Ryzhik}, that Donsker's Invariance Principle holds for  fBM  on the torus, in this case the correlation from the fBM is forgotten and lost in the wrapping.  
It is natural to expect the same loss of memory  in  fOU.  After all,  the linear contraction in the Langevin equation and the exponential convergence of the solutions would lead to the belief that it mixes as fast as the wrapped  fBM, this is not so.    Indeed, for $H>\f 12$  it is not strong mixing,  its auto correlation function is not integrable.   The long range memory survives also in the second order model (the Kinetic model for fBMs). 
  
   For the equation $\dot x_t^\epsilon= \epsilon^{H-1}  f(x_t^\epsilon) \, y^\epsilon_t $ on $\R$,  it is easy to show
  $$\left\|  |x^\epsilon -x|_{C^{\gamma'}([0,T])}  \; \right\|_{L^p}  \lesssim T^{\gamma}  \epsilon^{H-\gamma},$$
 where $0<\gamma'<\gamma < H$,  and $H>\f 13$ (this latter restriction is only needed for the error control).
 The limit solves the Young differential equation:
$\dot x_t=f(x_t) \;dX_t$,  where  $X_t$ is a fBM, so the effective limit  resembles, locally,  a fBM. 
For passive tracers in homogeneous incompressible fluid, there are some studies \cite{Fannjiang-Komorowski-2000,Komorowski-Novikov-Ryzhik-12}, in all of these eferences, the effective equation is  driven by either a BM or by a fBM, the method is also different.

However, the effective dynamics for (\ref{1}) will involve in general a mix of  BM's, fBM, and the non-Gaussian Hermite processes.
The appropriate scaling constants depend  on the  functions $G_i$.  
 For example for $N=1$ and $m$ the Hermite rank of  $G_1$,  $m=\f 1{2(1-H)} $  is the critical value 
for  the limit to be locally Gaussian. If $m$ is small,  the effective limit is locally the Hermite process of rank $m$.

We comment briefly on the effective equations. They are stochastic  equations  driven simultaneously by BM's, fBM's, Rosenblatt processes,
and higher rank Hermite processes.  For $H>\f 12$ the integrals  with respect to Hermite processes can be defined as Young integrals,  those 
with respect to the Wiener components are It\^o/ Stratonovich integrals.  We show that the Wiener process part of the driving limit is independent of
the Hermite processes, and the components of the Hermite processes can be written as multiple integrals with respect to the same Brownian motion $W_t$.
This means if we fix a sample path $W_t(\omega)$, we can consider the limit equation as a mixed Wierner-Young integral equation.
Mixed equations driven by BM's and  fBM have been studied for example in \cite{Guerra-Nualart,dasilva-Erraoui-ElHassan,Hairer-Li}.
Our convergence is actually in a strong topology, in $C^\gamma$ and in the rough path topology.  The limit equation is actually 
a rough differential equation, whose solution is in general not a semi-martingale and  is defined for all chance variables, the driver is a stochastic process of H\"older regularity class,  
   enhanced by iterated integrals.  
   
Lacking other tools, we will  use the solution theory  for rough path differential equations  to establish the required convergence,
see  \cite{Lyons94, Friz-Hairer},  in the first $p$-variation norms are used. 
Here it is convenient to use the H\"older path formulation and so we follow the notation in \cite{ Friz-Hairer}.  
However using the $p$-variation norms may help to  improve the integrability conditions in the Functional limit theorem. 
Due to the length of the article we do not study that aspect.
 To use the continuity theorem of the It\^o solution maps, we rewrite (\ref{1}) as rough differential equations driven by stochastic processes with a parameter $\epsilon$. It is then sufficient to prove the convergence of the drivers in the rough path topology, which is finer than the H\"older topology. 
We first prove the joint convergence of the drivers together with their iterated integrals in an appropriate H\"older space,  in the finite dimensional distributions.
Then we show that they converge also in the rough path topology. One of our key technical endeavours is therefore  a vector valued 
 functional `Central' Limit Theorem  in the rough path topology, this is accomplished in sections \ref{single} and  \ref{joint}. The statements of the main results will be presented in \S\ref{results}. The preliminary computations are presented together with a simple example in  \S\ref{preliminary}. In  \S\ref{section-single-scale}, we treat the 1-dimensional case, its proof does not use the extensive estimations obtained later, nor the CLT in rough path topology (those in \S\ref{single}  are sufficient). 
   The results for dimension $1$  is of course stronger, including all Hurst parameters, however to single this case out, we hope 
    to make transparent the proof for the multi-dimension and multi-scale case.
 The proof for the main theorem is  finalised in \S\ref{conclusion}. For reader's convenience,  interpretation of the rough differential equation, studied of fOU processes, kernel convergence for multiple Wiener integrals,   their asymptotic independence, and various preliminary estimates   are presented in the appendix.

\subsection*{Notation}

\begin{itemize}
\item  $(\Omega, \CF,  \P)$ is a  probability space and $\|\fat\|_{p}$ or $\Vert \fat \Vert_{L^p}$ denotes the norm in $L^p(\Omega)$, also when we refer to the sapce $L^p$ we mean $L^p(\Omega, \CF,  \P)$.
\item $W_t, t \in \R$ denotes a two-sided Wiener process.
\item  $\F_t$ denotes the filtration generated by the fractional Brownian Motion. 
\item $H$ is the Hurst parameter of the fBM.
\item $H^*(m) = m(H-1) +1 $.
\item $m_k$ is the Hermit rank of $G_k$. 
\item Convention :   $H^*(m_k)\le \f12$ for $k\le n$; otherwise  $H^*(m_k)>\f 12$,
\item $\mu=N(0,1)$ is the standard Gaussian distribution.
\item $L^p(\mu)$ denote the set of $L^p$ integrable  from $\R$ to $\R$ with respect to $\mu$.
\item $BC^r =\C_b^r$:  bounded continuous functions with bounded continuous derivatives up to order $r$. 
\item $\C_a^r=C^r\cap BC^a$.
\item $f\lesssim g$ means there exists a constant $c$ such that $f\le cg$.
 \item  $|x|_\alpha :=\sup_{s\neq t} \f{| x_t-x_s|}{|t-s|^\alpha}$ is  the homogeneous H\"older semi-norm, $0<\alpha <1$.
  \item For a process $x_t$,  set  $x_{s,t}:=x_t - x_s$.
\end{itemize}

\section{Formulation  of main results}\label{results}

We  let 
$\{H_m, m\ge 0\}$ denote the orthogonal Hermite polynomial of degree~$m$  on $L^2( \mu)$,
 normalised to have leading coefficient $1$ and $L^2(\mu)$ norm $\sqrt {m!} $.
For any $H\in (0,1)$, we define the non-increasing transformation $m\to H^*(m)$,
 \begin{equation}
	H^*(m) = m(H-1)+1.
\end{equation}

\begin{definition} 
   Let  $G:\R\to \R$ be an $L_2(\mu)$ function with chaos expansion 
\begin{equation}
G(x)=\sum_{k=m}^\infty c_k H_k(x), \qquad \qquad  c_k=\f 1 {k!}\<G, H_k\>_{L^2(\mu)}.
\end{equation}
( Observe that $\int G d\mu=0$   if and only if $c_0=0$.)
\begin{enumerate}
\item The smallest $m$ with $c_m\not =0$ is called the Hermite rank of $G$. 
\item  If $H^*(m)\le \f 12$ we say $G$ has  {\it high Hermite rank} (relative to $H$),  otherwise it is said to have {\it low Hermite rank}.
\end{enumerate}
\end{definition}

For  any $m\in \N$, we define a set of scaling constants as below, the
intuition leading to this will be clear when we present the relevant central  and non-central limit theroems.
\begin{equation}\label{beta}
\begin{aligned}
&\alpha\left(\epsilon,H^*(m)\right) = \left\{\begin{array}{cl}
\f 1 {\sqrt{\epsilon}}, \, \quad  &\text{ if } \, H^*(m)< \f 1 2,\\
\f  1 {\sqrt{  \epsilon \vert \ln\left( \epsilon \right) \vert}}, \, \quad  &\text{ if } \, H^*(m)= \f 1 2 \\
\epsilon^{H^*(m)-1}, \quad  \, &\text{ if } \,  H^*(m) > \f 1 2.
\end{array}\right.
\end{aligned}
\end{equation}

We fix a fractional Brownian motion $B_t^H$, with Hurst parameter $H \in (0,1)\setminus\{ \f 1 2\}$ (Homogenization for $H=\f 12$  is classic, 
the result  is independent of the Hermite rank and the scaling is given by $\alpha(\epsilon)=\f 1{\sqrt \epsilon}$).

 Let $y^\epsilon$ be a fast  stationary fractional Ornstein-Uhlenbeck process with standard Gaussian distribution.   We also write $y=y^1$ for simpicity.

\begin{convention}\label{convention}
Given a collection of functions $(G_k, k\le N)$, we will label the high rank ones first,  so the first $n$ functions
satisfy  $H^*(m_k) \le \f 12$, where $n \geq 0$,  and the rest  has $H^*(m_k) > \f 1 2$. \end{convention}

\subsection{Homogenization}
Let $G_k:\R\to \R$ be  centred function in $L_2(\mu)$ with Hermite rank $m_k$. Write
$$G_k= \sum_{l=m_k}^\infty  c_{k,l} H_l, \qquad \qquad    \alpha_k(\epsilon)=\alpha\left(\epsilon,H^*(m_k)\right).$$
We consider
\begin{equation}\label{multi-scale}
\left\{ \begin{aligned} \dot x_t ^\epsilon &=\sum_{k=1}^N\alpha_k(\epsilon) \, f_k(x_t^\epsilon) \,G_k(y_t^\epsilon),\\
 x_0^\epsilon&=x_0. \end{aligned}\right.  
 \end{equation}

For the main theorem below, we assume that  $G_k\in L_{p_k}$ for sufficiently high $p_k$, and $G_k$ satisfies a
 fast chaos decay condition.   {\it Both assumptions are automatically satisfied  if $G_k$ are polynomials,}
in which case we take $p_k=\infty$ in the statement below, then the only extra condition is that the Hermite rank of  $G_k$  is not in the interval $[\f 1{1-H}, \f 1  {2(1-H)}]$. The precise assumption will be detailed after the statement.

 \begin{theorem*} [\S  \ref{conclusion}]
 Given $H \in  \left( \f 1 2 ,1 \right)$, $f_k\in \C_b^3( \R^d ;\R^d)$, and   $G_k$'s  satisfying Assumption~\ref{assumption-multi-scale} below.
 Then, the solutions of (\ref{multi-scale}) 
converge weakly  in  $\C^\gamma$, on any finite time interval, for any $\gamma \in (\f 1 3 ,\f 1 2 - \f 1 {\min_{k \leq n } p_k})$,
to the solution of  the  following  stochastic differential equation
 $$ d{x}_t =\sum_{k=1}^n f_k(x_t) \circ d X^k_t+\sum_{k=n+1}^N f_k(x_t) d X^k_t, \quad x_0=x_0,$$
where $X^k_t$ is a Wiener process for $k \leq n$, a Hermite processes for $k>n$, and $\circ$ denotes  Stratonovich integral, otherwise a Young integral. 
\end{theorem*}

  \begin{definition}
A function $G\in L^2( \mu)$,   $G=\sum_{l=0}^\infty  c_l H_l$,
is said to satisfy the fast chaos decay condition with parameter $q\in \N$, if 
	$$\sum_{l=0}^\infty  {|c_l|}\; \sqrt{l!} \;(2q-1)^{\f l2}<\infty.$$ 
 \end{definition}

For Hermite polynomials and Gaussian measures we have the estimates:
 $$\|H_k\|_{2q}\le  (2q-1)^{\f k2 }\sqrt{ \E (H_k)^2} = (2q-1)^{\f k2 }\sqrt{k!}.$$
Consequently,  if $G$ satisfies the fast chaos decay condition with parameter $q$, then 
$$\|G\|_{2q} \le \sum_{l=0}^\infty  \vert c_l \vert \|H_l\|_q<\infty,$$
and $G\in L_{2q}$.  Observe that $\f 12 -\f 1{2q} >\f 13$, a condition needed for the convergence in $\C^\gamma$,   is equivalent to $q>3$. 
Also, if $G$ satisfies the decay condition with $q > 1$, then $G$ is continuous.
 Indeed, we have
$$\left|e^{-\f {x^2} 2} H_k(x) \right|_\infty \le 1.0865\sqrt{k!},$$
 from  \cite[pp787]{Abramowitz-Stegan}. (The polynomials  in \cite{Abramowitz-Stegan} are orthogonal with respect to $e^{x^2} dx$.)
 Thus the power series  $e^{-\f {x^2} 2}\sum_{l=0}^\infty   c_l H_l$ converges uniformly in $x$, the limit $G$ is
continuous. 
\begin{remark}
If $G$ satisfies the fast chaos decay condition with parameter $q >1$, then $G$ has a representation in $L^{2q} \cap \C$, with which we will work from here on.
\end{remark}

 \begin{assumption}
 [Functional limit rough $\CC^\gamma$- assumptions]
  \label{assumption-multi-scale}
 Each $G_k $ belongs to $L^{p_k}(\mu)$, where $p_k >2$, and has Hermite rank $m_k\ge 1$. Furthermore,
 \begin{enumerate}
 \item [(1)] Each $G_k$ satisfies the fast chaos decay condition with parameter $q \geq 4 $.
\item [(2)](Integrability condition) $p_k$ is sufficiently large so the following holds:
\begin{equation}\label{Hoelder-sum>1}
\min_{k\leq n} \left( \f 1 2 - \f 1 {p_k} \right) + \min_{n<k\leq N} \left( H^*(m_k) - \f 1  {p_k} \right)>1.
\end{equation}
\item[(3)] If $G_k$ has low Hermite rank, i.e. $H^*(m_k) > \f 1 2$,   assume $ H^*(m_k) - \f 1 {p_k} > \f 1 2$.
 \item[(4)] 
Either $H^*(m_k) <0$ or $H^*(m_k)> \f 1 2$.
\end{enumerate}
 \end{assumption}

\begin{remark}
\
\begin{enumerate}
\item 
The higher than usual moment assumptions  arise from the necessity to obtain  the convergence, not just in the space of continuous functions but also, in a rough path space $\FC^{\gamma}$ for some $\gamma > \f 1 3$ (which is naturally established by Kolmogorov type arguments) to be able to use the continuity of the solution maps in the rough path setting.

\item 

Condition (2) makes sure that the H\"older regularity of the terms which converge to a Wiener processes is at least $\eta$ and the ones for the terms which converge to Hermite processes is at least $\tau$ with $\eta + \tau >1$. This condition ensures that iterated integrals, in which one term converges to a Wiener and the other one to a Hermite process, can be interpreted as a Young integral.
\item 
In Condition (4) we have to assume   $H^*(m_k)<0$ instead of ${H^*(m_k)\le \f 12}$. This means we exclude functions with Hermite  rank in
$[\f 1{1-H}, \f 1  {2(1-H)}]$.
This restriction is due to Lemma \ref{lemma-integral}, where we deal only with high Hermite rank functions, we  only
obtain the required integrability estimates for $H^*(m_k)<0$. 
\end{enumerate}
\end{remark}

The homogenisation problem for a passive tracer in a turbulent flow has been studied in  \cite{Fannjiang-Komorowski-2000,Komorowski-Novikov-Ryzhik-12}. 
For a class of spatial homogeneous (incompressible) time stationary vector fields whose spectral density  satisfies suitable conditions,
 they showed  that the effective limit  is either a Brownian motion or a fractional Brownian motion. A class of homogenization theorems
 was shown in \cite{Kelly-Melbourne}, their integrability conditions were then lowered  in \cite{Chevyrev-Friz-Korepanov-Melbourne-Zhang} by using the  p-variation rough path formulation instead of the H\"older one, a related work is also to be found in  \cite{roughflows}.

\subsection{Lifted functional  limit theorem}

The static problem preluding the homogenization are functional  limit theorems. Once  appropriate limit theorems for the drivers are established,  we 
may use the continuity theorem for rough differential equations.   

 For continuous processes this concerns the scaling limit of $\int_0^t G( y_{\f s {\epsilon}})ds$ where  $G$ is a a centred function.
  Let functions $(G_1, \dots, G_N)$ be given. The pivot theorem  is concerned with the scaling limit , as  $\epsilon\to 0$, for
\begin{equation}
X^\epsilon:=\left( X^{1,\epsilon}, \dots, X^{N,\epsilon}\right), \qquad \qquad
  X^{k,\epsilon}= \alpha_k (\epsilon)\int_0^{t} G_k(y^{\epsilon}_s)ds.
\end{equation}
We further define the rough paths $\X^{\epsilon}=( X^\epsilon, \XX^{i,j,\epsilon})$, where
\begin{equation}
 \XX^{i,j,\epsilon}_{u,t} := \int_u^t  ( X^{i,\epsilon}_s - X^{i,\epsilon}_u ) d \,X^{j,\epsilon}_s
= \alpha_i(\epsilon) \alpha_j(\epsilon)\int_u^{t} \int_u^s   G_i(y^{\epsilon}_s) G_j(y^{\epsilon}_r) \,dr ds.
\end{equation}
We call $\X^{\epsilon}=(X^\epsilon, \XX^{i,j,\epsilon})$ the canonical lift of   $X^\epsilon$.

Such limit theorems are  closely related  to those for sums of sequence of  random variables.  For independent or strong mixing sequences, there is a central limit theorem (CLT), and the weak limit is always a Brownian motion. For interdependent  long range dependent stationary sequences, this was pioneered  by
Rosenblatt, who constructed a not strong mixing stationary sequence, with a non-CLT limit, the limit is later known as the Rosenblatt process.
    For stationary continuous time strong mixing processes, the CLT states that $ {\sqrt{\epsilon}} \int_0^t G( y_{\f s {\epsilon}})ds$
   converges weakly to a Markov process, this is  classical and well understood.
  For stochastic  processes whose auto-correlation functions do not decay sufficiently fast at  infinity,  there is no reason to have  the  $\sqrt \epsilon$ scaling or to have a diffusive limit, see   \cite{Bai-Taqqu, Taqqu, Rosenblatt, BenHariz, Dobrushin-Major, Breuer-Major,Buchmann-Ngai-bordercase,Hu-Nualart-Xu,Campese-Nourdin-Nualart}.

\bigskip

To state the main theorem clearly, we follow Convention \ref{convention} and label  the first  $n$ functions such that  ${H^*(m_k)\le \f 12}$ for $k \leq n$. Therefore, we will write 
\begin{equation}
X^{\epsilon} = (X^{W, \epsilon} , X^{Z, \epsilon}), \qquad  X^{W, \epsilon} \in \R^n, \; X^{Z, \epsilon} \in \R^{N-n}.
\end{equation}

\begin{assumption}[Functional Limit $\C^\gamma$ assumptions]
\label{assumption-single-scale-not-continuous}
Let $G_k\in L_{p_k}$ with Hermite rank $m_k\ge 1$.
\begin{itemize}
 \item[(1)]     If $H^*(m_k) \le  \f 1 2$  assume $\f 1 2 - \f 1 {p_k} > \f 1 3$, which is equivalent to $p_k >6$.
\item[(2)] If $H^*(m_k) > \f 1 2$ assume $ H^*(m_k) - \f 1 {p_k} > \f 1 2$.
 \end{itemize}
\end{assumption}
For the convergence in  finite dimensional distributions we only assume that each $G_k\in L_2$,  only for  convergence in $\C^\gamma$  we need the above assumption. The following is extracted from section \ref{joint}. 

\begin{theorem*}\label{theorem-CLT} [CLT]
Let $G_1, \dots, G_N \in L^2(\mu)$  be given and let $H \in (0,1)$. 
\begin{itemize}
\item [(a)] 
Then, there exists $X^W=( X^1, \dots, X^n)$  and $X^Z=(X^{n+1}, \dots, X^N)$, such that
$$(X^{W,\epsilon}, X^{Z,\epsilon}) \longrightarrow  (X^W, X^Z),$$
in the sense of finite dimensional distributions on every finite interval.
Furthermore, for any $t >0$
  $$\lim_{\epsilon \to 0} \|X^{Z,\epsilon}_t \to X^Z_t\|_{L^2}=0.$$
 \item [(b)]  If furthermore, each $G_k$ satisfies Assumption \ref{assumption-single-scale-not-continuous}, one obtains convergence in $\C^{\gamma}$ for $\gamma <\f 1 2 - \f {1} {\min_{k \leq n } p_k}$.

\item[(c)]  Suppose Assumption \ref{assumption-multi-scale} holds and further assume that $H \in (\f 1 2, 1)$. Then, on every finite interval
and for every $\gamma \in (\f 1 3 , \f 1 2 - \f {1} {\min_{k \leq n } p_k})$,
$$\left(X^\epsilon_t,
  \XX^{\epsilon}_{s,t}  \right) \to  \left(X_t, \XX_{s,t}+(t-s)A\right)$$ 
weakly in the rough topology $\FC^\gamma$, given in (\ref{rough-distance}).
  
    \end{itemize} 
   We now describe the limit. Set $X=( X^1, \dots, X^n, X^{n+1}, \dots, X^N)$. 
 \begin{enumerate}
 \item [(1)]   $X^W \in \R^n$  and $X^Z \in \R^{N-n}$ are independent. 
\item [(2)] For $i,j\le n$,   $\E\left( X^i_t X^j_s\right)= (t \wedge s) A^{i,j}$  where for  $\rho(r)=\E (y_ry_0)$,
$$A^{i,j}=\int_0^{\infty} \E\left( G_i(y_s) G_j(y_0) \right) ds = 
  \sum_{q=m_i\vee m_j}^{\infty}   c_{i,q}\; c_{j,q}  \; (k!) \, \int_0^\infty  \rho(r)^q\, dr.$$
In other words, $ X^W = U \hat W_t$ where $ \hat W_t$ is a standard Wiener process,  $U$ is a  square root of
$A$. 

\item [(3)] Let $Z_t^{H^*(m_k),m_k}$ be the Hermite processes, represented by (\ref{Hermite}).
Then,
$$ X^Z=(c_{n+1,m_{n+1}}  Z_t^{n+1} , \dots, c_{N,m_{N}}  Z_t^{N}).$$   where
  \begin{equation}\label{Hermite-2}
   Z_t^{k}= \f{m_k!}{K(H^*(m_k),m_k)} Z_t^{H^*(m_k),m_k}.
\end{equation}
We emphasize that the Wiener process $W_t$ defining the Hermite processes are the same, for every $k$,
which is in addition independent of $\hat W_t$. 
 \item [(4)] For $s<t$ the limiting second order process is given by
$$
\XX^{i,j}_{s,t}=\int_s^t  (X_r^i - X_s^i) d X_r^j, 
  \qquad 
 \begin{aligned}
& \hbox{an It\^o integral},  \qquad & \hbox{ for } i, j\le n,\\
&\hbox{a Young integral}, & \hbox{otherwise.}
\end{aligned}
$$

$$A^{i,j}=\left\{
 \begin{aligned}
 &\hbox{ as in part 2}, \qquad & \hbox{ if } i, j\le n,\\
&0, & \hbox{otherwise.}
\end{aligned}\right.
$$
 \end{enumerate}
 \end{theorem*}

For strong mixing processes, these are well known see \cite{Kipnis-Varadhan}, see also  the relatively recent book \cite{Komorowski-Landim-Olla}. A different type of limit  theorem for fractional Brownian motions is popular  under the topic rough volatility where the parameter $H$ is taken to be close to zero.

We now comment on the proof and give an overview on the results we used.
At the level of the convergence of  the base processes $ X^\epsilon$, there are a range of results.
See for example  \cite{Dobrushin-Major, Taqqu,  Breuer-Major, Buchmann-Ngai-bordercase,    Pipiras-Taqqu,  
 BenHariz,  Cheridito-Kawaguchi-Maejima, 
Bai-Taqqu,  Hu-Nualart-Xu, Nourdin-Nualart-Zintout,  Campese-Nourdin-Nualart}.
 Even at this level, the convergence has only been shown in finite dimensional distribution
  (occasionally in  the continuous topology). Even for the convergence in finite dimensional distribution,
  the known results are fragmented, they are not consistent in the assumptions.
  They are often proved for a subclass of situations, such as  finite chaos condition, 
    moment determinant condition etc., while  Hermite processes are in general not determined by their moments.
Some theorems are only proved for scalar processes, some are only at the level of sequences.

We first assemble  the convergence of scalar processes, extending them to  the same
 larger class of functions.  To extend the convergence to the H\"older topology,
    we  follow \cite{Campese-Nourdin-Nualart}  and use Malliavin calculus to obtain  moment bounds.
    This section is quite short, only about 3 pages. 
 
 For the joint convergence in H\"older norm, we use  \cite{ Bai-Taqqu, Taqqu, BenHariz}.
 The joint convergence of a vector valued process with components including a Wiener process and 
a non semi-martingale  is subtle.  We use a  reduction  theorem, a normal convergence theorem from \cite{Nourdin-Peccati}, 
and an extension of a limit theorem  from \cite{Jacod-Shiryaev}. We also use the fact that  the 
low Hermite components of $X^\epsilon$ converge in  $L_2$ (this is proven in the Appendix.)

For  the functional limit theorems in the rough path topology we first  show the joint convergence of the integrals and iterated integrals in finite dimensional
distribution. For this we establish a martingale approximation and use ergodic theorems. Due to the fact that we have a non-strong mixing process, proving the $L^2$ boundedness of the martingale approximations is rather involved, this is where we had to exclude functions with Hermite rank  satisfies
$H^*(m)\in [0, \f \12]$. For the $L^2$ boundedness we  follow  \cite{Hairer05, Hairer-Li} and  develop a locally independent  decomposition for the fOU process and use this decomposition  to compute the conditional moments. 
 The final hurdle is the relatively compactness of the iterated integrals in the rough path topology,   for which we  relied on the Diagram formula  and  an 
upper bound,  from \cite{Graphsnumber}, on the number of eligible complete graphs of pairings.

\subsection{Single scale model and examples}
The following is extracted from section 5, where further detail is given. Restricting ourselves to the one dimensional case we can see how the methodology works without technicalities, here we  extend to range of $H$ to $(\f 1 3 ,1)$ and drop the exclusilons on $m$ and we obtain the following theorem.
Given a centred function $G \in L^2(\mu) $, with chaos expansion $G=\sum_{k=m}^{\infty} c_k H_k$ we set $c>0$ by
\begin{equation}\label{c-square}
c^2=\left\{ \begin{array}{ll}  (\f {c_m m!}{K(H,m)})^2, \quad & H^*(m)>\f 12\\
2\sum_{k=m}^\infty (c_k)^2 k! \int_0^\infty \rho^k(s) ds, \quad &H^*(m)<\f 12\\
2m!(c_m)^2, \quad &H^*(m)=\f 12.
\end{array} \right.
\end{equation}

\begin{theorem*} 
Let $H \in (\f 1 3,1) \setminus \{ \f 1 2 \}$,  $f\in \C_b^3(\R; \R)$, and $G$ be a continuous function which satisfies Assumption \ref{assumption-single-scale-not-continuous}.
Consider
\begin{equation}\label{limit-eq}
\dot x_t ^\epsilon=  \alpha(\epsilon,H^*(m))\; f(x_t^\epsilon) \; G( y^\epsilon_{t}), \quad  \qquad  x_0^\epsilon=x_0.
 \end{equation} 
\begin{enumerate}
\item If
$H^*(m) > \f 1 2$, $x_t^\epsilon$ converges weakly in $\C^{\gamma}$  to the solution to the Young differential equation
$d\bar x_t = c f(\bar x_t) \,dZ_t^{H^*(m),m}$ with initial value $x_0$ for $\gamma \in (0, H^*(m)- \f 1 p)$.
\item If
$H^*(m) \leq \f 1 2$, 
 $x_t^\epsilon$ converges weakly  in  $\C^\gamma$ to the solution of the Stratonovich stochastic differential equation
$d\bar x_t = c  f(\bar x_t) \circ \,dW_t$ with $ \bar x_0=x_0$, where $\gamma \in(0, \f 1 2- \f 1 p)$.
\end{enumerate} 
\end{theorem*}

\begin{remark}
\
\begin{enumerate}

\item 
The constant $c$ could be $0$, for further details see Remark \ref{remark-variance}.

\item The condition $\f 1 2 - \f 1 p > \f 1 3$ is for the H\"older regularity of the solution paths to be at least $\f 13$, so we could define the integral
 by an enhanced Riemann sum. See \S\ref{how?} for the precise meaning.
\item The condition $f \in \C_b^3$ is not optimal, it  is only needed for applying  the conclusions of  Theorem \ref{cty-rough}, the continuity theorem which  states that the
 solutions of  a Young/ rough differential equations depends continuously on the driver.
  For part (1),  the Young case, it is sufficient to assume that  $f \in \mathcal{C}_b^2$. 
\item By using the $p$-variation norm instead of the H\"older norms, one can reduce the integrabilty condition to $p=1$ in (1) and $p>1$ in (2) and obtain convergence in the respective $p$-variation spaces, see \cite{Chevyrev-Friz-Korepanov-Melbourne-Zhang} for the use of such norms. The same is true for the other forthcoming theorems,
by observing that there is no loss of regularity in the Besov $p$-variation embeddings compared to the Besov-H\"older ones, see Appendix A from the book \cite{Friz-Victoir}.
 
\end{enumerate}
\end{remark}
We conclude this section with an example on $\R^d$ and a question. Take $H=\f 89$ and consider
$$\dot x_t^\epsilon=  \epsilon^{-\f {2} {9}} f_1( x_t^\epsilon) H_2(y_{\f t\epsilon}) +\epsilon^{-\f {4} {9}} f_2( x_t^\epsilon) H_4(y_{\f t\epsilon})  +  \f 1 {\sqrt \epsilon } f_3 ( x_t^\epsilon) H_{10}(y_{\f t\epsilon}).$$
 Their solutions  converge to that of an equation  driven simultaneously by a fBM, a Hermite process with similarity exponents $\f 5 9$, and a Wiener process.

\section{Preliminaries}\label{preliminary}

A  stochastic process $(X_n)$ is strongly mixing if its auto correlation $\rho(n)\to 0$ as $n\to \infty$
and  $$ \sup_{A,B} |P(A\cap B)-P(A)P(B)|\le \rho(n),$$  where the supremum is over $A\in \sigma\{ X_k, k\le m\}$, $B\in \sigma\{ X_{k}, k\ge n+ m\}$.

Let us make the convention that   $B^H_0=0$, $\E(B^H_1)^2=1$.  For simplicity we  often omit the Hurst parameter~$H$. 
A  fBM $B_t^H$, with hurst parameter $H > \f 1 2$ is not strong mixing,  otherwise the Central Limit Theorem  holds for $X_n=B^H_{n+1}-B^H_n$,  but  with suitable scaling, the limit yields a fractional Brownian motion.

The disjoint increments of $B_t^H$ are dependent unless $H=\f 12$:
\begin{equation*}
\E(B_t-B_s)(B_u-B_v)=\f 12 \left( |t-v|^{2H}+|s-u|^{2H}-|t-u|^{2H}-|s-v|^{2H} \right).
\end{equation*}
The correlation function   $$\begin{aligned}\tilde  \rho(n)&=\E (B_{n+1}-B_{n})(B_1-B_0)\\
 &\sim H(2H-1) n^{2H-2}, \qquad \hbox{ at infinity}. \end{aligned}$$
Thus, $\sum_{n=1}^\infty \tilde \rho(n)=\infty$ for $H>\f 12$, and so long range correlation persists.

Common to a Brownian motion, a fBM  has self-similarity with similarity exponent $H$ and stationary increments.  
Since $B_t^H$ has finite and non-trivial $\f 1 H$-variation over $[0,T]$ with variation of the order $\E \left(|B_1^H|^p \right)T$,  it has infinite  total variation. It has  zero  quadratic variation  for $H>\f 12$  and infinite  quadratic variation  for $H<\f 12$,   and therefore $B_t^H$ is not a semi-martingale unless $H=\f 12$. 
We refer to  \cite{Pipiras-Taqqu-book, Samorodnitsky, Cheridito-Kawaguchi-Maejima} for further detail.

\subsection{Hermite processes}
\label{section-Hermite}

We take the Hermite polynomials of degree $m$ to be $$H_m(x) = (-1)^m e^{\f {x^2} {2}} \f {d^m} {dx^m} e^{\f {x^2} {2}}.$$ Thus, $H_0(x)=1$,  $H_1(x)=x$.   
Let $\hat H$  be the inverse of $H^*(m)=m(H-1)+1$: $$ \hat H(m)=\f  1 m (H-1) + 1.$$

\begin{definition}\label{Hermite-processes}
Let $m\in \N$ with $\hat H(m)>\f 12$. The class of {\it Hermite processes} of rank $m$ are the following mean-zero process,
\begin{equation}\label{Hermite}
 Z_t^{H,m}= \f{K(H,m)}{m!}  \int_{0}^t H_m\left(\int_{\R} (s-u)_{+}^{  \hat H(m)-\f 3 2} d  W_u\right) \, ds,
\end{equation}
where the constant $K(H,m)$ is chosen so their variances are $1$ at $t=1$. The number $H$ is called its Hurst parameter.
\end{definition}
Since $\hat H(1)=H$, the rank $1$  Hermite processes $Z^{H, 1}$ are fractional BMs. Indeed (\ref{Hermite}) is  exactly the 
 Mandelbrot-Vanness representation for fBM. 
The Hermite processes have stationary increments and finite moments of all order with covariance 	\begin{equation}
	\E(  Z_t^{H,m} Z_s^{H,m}) =  \f 1 2 (   t^{2H} + s^{2H} - \vert t-s \vert^{2H}).
	\end{equation}
Therefore, using Kolmogorv's theorem one can show that the Hermite processes $Z_t^{H,m}$ have sample paths of H\"older regularity up to  $H$. 
They are also self similar with exponent $H$
$$ \lambda^H  Z^{H,m} _{\f \cdot\lambda} \sim  Z^{H,m}_..$$

We recall another  formulation, useful for proving convergence to Hermite processes.
According to It\^o \cite{Ito51} and  \cite[Thm1.1.2]{Nualart},  if $f$ is an $L^2$ function of norm~$1$,  the multiple It\^o-Wiener integral with kernel $\prod_i f(t_i)$
can be identified with the evaluation of $H_m$ on a single Wiener integral: 
$$\int_\R\dots \int_\R  f(t_1) \dots f(t_m) \, d W(t_1) \dots d  W(t_m) = H_m\left( \int_\R f(s)   \,d   W_s\right),$$
So the Hermite processes can be defined by the multiple It\^o-Wiener integrals:
\begin{equation}Z_t^{H,m}=\f  {K(H,m)} {m!} \int_{\R^m} \left(\int_0^t \prod_{j=1}^m (s-t_j)_+^{  -(\f 1 2 + \f {1-H} {m})} \, ds\right) \,  d  W({t_1}) \dots d  W({t_m}).
\end{equation}
In particular,  two  Hermite processes  $Z^{H, m}$ and $Z^{H',m'}$, defined by the same Wiener process, are uncorrelated if $m \not = m'$.

\begin{remark}
 In  some literature, e.g. \cite{Maejima-Ciprian},
	the Hermite processes are defined with a different exponent as below:  
	$$\tilde Z_t^{H,m}= \f{K(H^*(m),m)}{m!}  \int_{0}^t H_m\left(\int_{\R} (s-u)_{+}^{H-\f 3 2} d W_u\right) \, ds.$$
They are related by
 \begin{equation}Z_t^{H^*(m), m}=\tilde Z_t^{H,m}, \qquad Z_t^{H,m}=\tilde Z_t^{\hat H(m),m}.
\end{equation}
\end{remark} 
Further detail on Hermite processes can also be found in \cite{Maejima-Ciprian}.
The limit processes in Theorem A  are given by Hermite processes of the form $Z^{H^*(m),m}_t = \tilde{Z}^{H,m}_t$.

\subsection{ Fractional Ornstein-Uhlenbeck processes}\label{OU-section}
We define the stationary fractional Ornstein-Uhlenbeck processes to be
$$y_t =  \sigma \int_{-\infty}^t e^{-(t-s) } dB_s,$$
where $\sigma$ is chosen such that  $y_t$ is distributed as $\mu=N(0,1)$ and  $B_t$ is a  two-sided fractional BM.
It  is the  solution of the following Langevin equation:
$$dy_t = - y_t dt + \sigma d B_t, \qquad y_0 = \sigma \int_{-\infty}^0 e^{ s }  dB_s.$$

 We take $y_t^\epsilon$, the fast fOU, to be the  stationary solution of
$$dy_t^\epsilon = -\f 1 \epsilon \lambda y_t^\epsilon\, dt + \f { \sigma} {{\epsilon}^H}\, d B_t.$$
Observe that $ y_\cdot ^\epsilon$ and $ y_{\f \cdot \epsilon} $ have the same distributions, and
\begin{equs}\label{OU-process}
 y^\epsilon_t&=\f \sigma {\epsilon^H}\int_{-\infty}^t e^{-\f 1 \epsilon (t-s) } dB_s.
\end{equs}

Let  us denote their auto-correlation function by $\rho$ and $\rho^\epsilon$:
$$\rho(s,t):=\E(y_sy_t), \qquad \rho^{\epsilon}(s,t):= \E(y^{\epsilon}_sy^{\epsilon}_t)$$ 
Let $\rho(s)=\E (y_0y_s)$ for $s\ge 0$  and extended to $\R$ by symmetry, then $\rho(s,t)=\rho(t-s)$ and similarly for $\rho^{\epsilon}$.
For $H>\f 12$, the set of functions for which Wiener integrals are defined include $L^2$ functions and
so $\rho$ has a nicer expression.

 Indeed, since
$$\E (B_tB_s)=H(2H-1) \int_0^t\int_0^s |r_1-r_2|^{2H-2}dr_1dr_2,$$
we have
 $$\f{ \partial^2}{\partial t\partial s} \E(B_tB_s) =H(2H-1)|t-s|^{2H-2},$$ which is integrable, and therefore we may use the Wiener isometry
\begin{equs}
	\E(y_ty_s) =\sigma^2 H(2H-1) \int_{-\infty}^t\int_{-\infty}^s  e^{-(s+t-r_1-r_2) }  |r_1-r_2|^{2H-2} dr_1 dr_2.
\end{equs}
For $u>0$, 
$$\rho(u)=\sigma^2 H(2H-1) \int_{-\infty}^{u}\int_{-\infty}^0  e^{-(u-r_1-r_2) }  |r_1-r_2|^{2H-2} dr_1 dr_2.$$ 
With this we observe  the following correlation decay relation.
\begin{lemma}
	\label{correlation-lemma}
	Let  $H\in (0, \f 12)\cup (\f 12, 1)$.  For any $t\not =s$, \begin{equation}\label{cor1}
	|\rho(s,t)| \lesssim 1\wedge |t-s|^{2H-2}.
	\end{equation}
\end{lemma}
The proof for this is elementary, for reader's convenience it is given  in the appendix.

By Lemma \ref{correlation-lemma},  $  \int_0^\infty \rho^m(s) ds$ is finite if $ H^*(m)<\f 12$, or if $H=\f 12$ and $m\in \N$, as in the latter the usual OU process admits exponential decay of correlations.
\begin{lemma}\label{Integrals}
Let $H\in (0,1) \setminus \{ \f 1 2\}$, fix a finite time horizon $T$, then for $t \in [0,T]$ the following holds uniformly for $\epsilon \in (0,\f 1 2]$:
\begin{equation} \label{correlation-decay-2-1}
	\left( \int_0^{\f t \epsilon} \int_0^{\f t \epsilon}  \vert  \rho(u,r) \vert^m\, dr \,du\right)^{\f 12} \\
 \lesssim 
\left\{	\begin{array}{lc}
 \sqrt {\f t \epsilon  \int_0^\infty \rho^m(s) ds},  \quad  &\hbox {if} \quad H^*(m)<\f 12,\\
 \sqrt { (\f t \epsilon)  \vert \ln\left(\f 1 \epsilon \right) \vert}, \quad  &\hbox {if} \quad H^*(m)=\f 12,\\ 
 \left(  \f t \epsilon\right) ^{H^*(m)},  \quad &\hbox {if} \quad H^*(m)>\f 12.
 \end{array} \right.
\end{equation}

\begin{equation} \label{correlation-decay-2-2}
\left( \int_0^{t} \int_0^{t}  \vert  \rho^{\epsilon}(u,r) \vert^m\, dr \,du\right)^{\f 12} \\
\lesssim 
\left\{	\begin{array}{lc}
\sqrt { t \epsilon  \int_0^\infty \rho^m(s) ds} ,  \quad  &\hbox {if} \quad H^*(m)<\f 12,\\
\sqrt { t \epsilon  \vert \ln\left(\f t \epsilon \right) \vert}, \quad  &\hbox {if} \quad H^*(m)=\f 12,\\
\left(  \f t \epsilon\right) ^{H^*(m)-1},  \quad &\hbox {if} \quad H^*(m)>\f 12.
\end{array} \right.
\end{equation}		
Note if $H=\f 12$, for and any $m\in \N$, the  bound is $ \sqrt {\f t \epsilon}  \int_0^\infty \rho^m(s) ds$. The following is often used later,
\begin{equation}\label{integral-10}
t \int_0^{t} \vert \rho^{\epsilon}(s)\vert^m ds
\lesssim \f { t^{ \left(H^*(m) \vee \f 1 2\right)}}  { \alpha \left( \epsilon, H^*(m)\right)}.
\end{equation}
\end{lemma}

The following H\"older norm estimates will be used for proving Proposition \ref{th-example}.

\begin{lemma}\label{cty-lemma}
The stationary fOU process 
is uniformly H\"older continuous of order $\gamma$ over $[0,\infty)$ for any $\gamma \in (0, H)$. Furthermore, over $[0, \infty)$,
the following estimates hold:  
 $$\|y_s-y_r\|_{L^p} \lesssim   1 \wedge |s-r|^{H}, \qquad  \E  \sup_{s\not =t} \left( \f{  |y_s-y_t |} { |t-s|^{\gamma}} \right)^p \lesssim  C(\gamma, p)^p M$$
 for any $p>1$, where  $C(\gamma,p)$ is the universal constant in Garcia-Rodemich-Romsey-Kolmogorov inequality and 
$$M = \int_0^{\infty} \int_0^{\infty} \f { \E|y_s-y_r|^p}{ {|s-r|}^{ \gamma p+2}}<\infty.$$
\end{lemma}

\subsection{Some rough path theory} \label{how?}
If $X$ and $Y$ are H\"older continuous functions on $[0,T]$ with exponent $\alpha$ and $\beta$ respectively, such that $\alpha + \beta >1$, the Young integration theory enables us to define $\int_0^T Y dX$ via Riemann sums $\sum_{[u,v]\in \mathcal P} Y_u(X_v-X_u)$, where $\CP$ denotes a partition of $[0,T]$. Furthermore $(X,Y)\mapsto \int_0^T Y dX$ is a continuous map. Thus, for $X\in \C^{\f 12+}$, one can make sense of a solution $Y$ to the Young integral equation $dY_s=f(Y_s) dX_s$. If $f\in \C_b^2$,  the solution is  continuous with respect to both the driver $X$ and  the initial data \cite{Young}.
In the case of $X$ having H\"older continuity less or equal to $\f 12$,   this  fails and one can not define a pathwise integration by the above Riemann sum anymore.  
Rough path theory provides us with a machinery to treat less regular functions by enhancing the process with a second order process, giving a better local approximation, which then can be used to enhance the Riemann sum and show it  converges. If $X_s$ is a Brownian motion, taking the dyadic approximation then the usual Riemann sum leads to convergent in 
probability to It\^o integrals; but the enhanced Riemann sum provides better approximations and defines a pathwise integral
agreeing with the It\^o integral provided the integrand belongs to both domains of integration. Their domains of integration
are quite different, the first uses an additional adaptedness condition and requires arguably less regularity than the second.

We restrict ourselves to the case where  $X_t$ is a continuous path over $[0,T]$, for now we assume it takes values in $\R^d$. 
A rough path of regularity $ \alpha \in (\f 1 3 , \f 1 2)$,
is a pair of processes $\X=(X_t, \XX_{s,t})$ where $(\XX_{s,t}) \in \R^{d \times d} $ is a two parameter stochastic processes
satisfying the following algebraic conditions: for $0\le s<u<t\le T$, 
$$\XX_{s,t}-\XX_{s,u}-\XX_{u,t}=X_{s,u} \otimes X_{u,t}, \qquad \qquad  \hbox{ (Chen's relation)} $$
where $X_{s,t}=X_t-X_s$, and $ (X_{s,u} \otimes X_{u,t})^{i,j}  = X^i_{s,u}  X^j_{u,t}$ as well as the following analytic conditions,
\begin{equation}\label{geo}
\Vert X_{s,t} \Vert \lesssim |t-s|^\alpha,  \qquad \Vert\XX_{s,t}\Vert \lesssim |t-s|^{2\alpha}.
\end{equation}
The set of such paths will be denoted by $\FC^{\alpha}([0,T]; \R^d)$. The so called second order process $\XX_{s,t}$ can be viewed as a possible candidate for the iterated integrals $\int_s^t X_{s,u} dX_u$. 
\begin{remark}
Using Chen's relation for $s=0$ one obtains
$$ \XX_{u,t}= \XX_{0,t} - \XX_{0,u} - X_{0,u} \otimes X_{u,t},$$
thus one can reconstruct $\XX$ by knowing the path $t \to (X_{0,t}, \XX_{0,t})$.
\end{remark}
Given a  path $X$, which is regular enough to define its iterated integral, for example $X \in \C^1([0,T];\R^d)$, we define its natural rough path lift to be given by
$$\XX_{s,t}:=\int_s^t X_{s,u} dX_u.$$
It is now an easy exercise to verify that $\X = (X,\XX)$ satisfies the algebraic and analytic conditions (depending on the regularity of $X$), by which we mean Chen's relation and (\ref{geo}). Note that given any  function $F\in \C^{2\alpha}(\R^{d \times d})$, setting
$\tilde\XX_{s,t}=\XX_{s,t}+F_t-F_s$, $\tilde \XX$ would also be a possible choice for the rough path lift. 
Given two rough paths $\X$ and $ \Y$  we define their distance to be
\begin{equation}\label{rough-distance}
\rho_\alpha(\X, \Y)=\sup_{s\not =t} \f {\Vert X_{s,t} -Y_{s,t} \Vert } {|t-s|^\alpha} 
+\sup_{s\not =t} \f {\Vert\XX_{s,t} -\YY_{s,t}\Vert }  {|t-s|^{2\alpha}} 
\end{equation}
This defines a complete metric on $\FC^{\alpha}([0,T]; \R^d)$, this is called  the inhomogenous $\alpha$-H\"older rough path metric.
We are also going to make use of the norm like object 
\begin{equation}
\Vert \X \Vert_{\alpha} = \sup_{s \not = t \in [0,T]} \f {\Vert X_{s,t}\Vert} {\vert t-s \vert^{\alpha} } +  \sup_{s \not = t \in [0,T]} \f {\Vert \mathbb{X}_{s,t}  \Vert^{\f 1 2}} {\vert t-s \vert^{\alpha}}
\end{equation}
We also denote for any two parameter process $\XX$ a semi-norm:
$$\|\XX\|_{2\alpha} := \sup_{s \not = t \in [0,T]} \f {\Vert \XX_{s,t}  \Vert} {\vert t-s \vert^{2\alpha}}.
$$

Given a path $X$, as the second order process $\XX$ takes the role of an iterated integral, another sensible conditions to impose  is the chain rule (or  integration by parts formulae) leading to the following definition.
\begin{definition}
	A rough path $\X$ satisfying the following condition,
	\begin{equation}
	Sym(\XX_{s,t}) = \f 1 2 X_{s,t} \otimes X_{s,t}
	\end{equation} 
	is called a geometric rough path. The space of all of geometric rough paths of regularity $\alpha$ is denoted by $\FC^{\alpha}_g([0,T];\R^d)$ and forms a closed subspace of $\FC^{\alpha}([0,T];\R^d)$.
\end{definition} 
Furthermore,  one can show that if a sequence of $\C^1([0,T],\R^d)$ paths $X_n$ converges in the rough path metric to $\X$, then $\X$ is a geometric rough path. 
To obtain a geometric rough path from a Wiener process,  as $\int_0^t W_s \circ dW_s= \f {W_t^2} {2}$, one has to enhance it with its Stratonovich integral,
$\WW_{s,t}=\int_s^t (W_r-W_s)\circ dW_r$.

Given a rough path $\X \in \FC^\alpha([0,T];\R^d)$, we may define the integral $ \int_0^TY d\X$ for suitable paths $Y \in \C^{\alpha}([0,T],\mathbb{L}(\R^d,\R^m))$, which admit a Gubinelli derivative $Y'\in \C^{\alpha}([0,T],\mathbb{L}(\R^{d \times d},\R^m))$ with respect to $\X$,  meaning
$$Y_{s,t}=Y_s' X_{s,t}+R_{s,t},$$
and the two parameter function $R$, satisfies $\|R\|_{2\alpha}< \infty$.  The pair $\Y:=(Y, Y')$ is said to be a
controlled rough path, their collection is denoted by $\D^{2\alpha}_X$. The remainder term for the case $Y=f(X)$ with $f$ smooth is
the remainder term in the Taylor expansion.
This is done by showing that the enhanced Riemann sums
$$\sum_{[s,t]\in \CP} Y_{s} X_{s,t} + Y'_{s} \XX_{s, t},$$ 
converge as the partition size is going to zero, and the limit is defined to be $\int \Y \,d\X$.
With this theory of integration one can  study the equation,
$$ dY = f(Y) d\X.$$
Unlike in the theory of stochastic differential equations one now has continuous dependence on the noise $\X$. This can be interpreted as the following, once one has chosen a candidate for the iterated integral, the second order process, solving rough differential equations is a continuous operation, but obtaining the second order process is not, see the introduction to \cite{Friz-Hairer}. 

Now given  $Y \in \D_X^{2\alpha}$, then $(\int \Y\,d\X, Y)\in \D_X^{2\alpha}$, and the map $(\X,  \Y) \mapsto (\int \Y\, d\X, Y)$  is continuous with respect to $\X \in  \FC^{\alpha} $ and $Y\in \D_X^{2\alpha}$. The domain is not a product space, the continuity  is best formulated with the appropriate distances, however, we do not need the precise formulation, for details we refer to~\cite{Friz-Hairer}. 
We now state the precise theorem for our application, see also \cite{Lyons94}.
\begin{theorem}{\cite{Friz-Hairer}}\label{cty-rough}
	Let $Y_0 \in \R^m, \beta \in (\f1 3, 1), \, f \in \C^3_b(\R^m; {\mathbb L} (\R^d; \R^m)) $ and $\X \in \FC^{\beta}([0,T],\R^d)$. Then the differential equation
	\begin{equation}\label{example-sde}
	Y_t = Y_0 + \int_0^t f(Y_s) d\emph{X}_s 
	\end{equation}
	has a unique solution which belongs to $\mathcal{C}^{\beta}$. Furthermore, the solution map  $\Phi_f: ~\R^d\times \FC^{\beta}([0,T],\R)
	\to  \D_{X}^{2\beta}([0,T],\R^m)$, where the first component is the initial condition and the second component the driver, is continuous.
\end{theorem}
As continuous maps preserve weak convergence to show weak convergence of solutions to rough differential equations
$$ dY^{\epsilon} = f(Y^{\epsilon}) d\X^{\epsilon},$$
it is enough to establish weak convergence of the rough paths $\X^{\epsilon}$ in the topology defined by the rough metric. 
Obtaining convergence in  this topology follows the  convergence of the finite dimensional distributions of the rough paths $\X^\epsilon$ plus tightness in the space of rough paths with respect to that topology.  To apply this theory, we will enhance our
stochastic processes,  c.f. Proposition \ref{th-example}, Proposition \ref{proposition-single-CLT-strong-topologies} and the following proof of Theorem C and  Section \ref{joint} to bring it to this framework. 

\subsubsection{Tightness of rough paths}
\label{pre-tightness}
We now show that moment bounds on increments lead to tightness in the rough path topologies.
The following lemma is similar to the compact embedding theorems between H\"older spaces and can be obtained via an Arzela-Ascoli argument.
\begin{lemma}\label{compact}
	Let $0$ denote the rough path obtained from the $0$ function enhanced with a $0$ function, then for $\gamma > \gamma'$, the sets $\{\X \in \FC^{\gamma'} : \rho_{\gamma}(\emph{X},0) < R, \X(0)=0 \}$ are compact in $\FC^{\gamma'}$.
\end{lemma} 
The next lemma relates uniform moment bounds on the increments of the the stochastic processes and  their secondary process, to uniform moment bounds on the rough path norm.
\begin{lemma}\label{moment-conditions}
	Let $ \theta \in (0,1)$, $\gamma \in (0, \theta - \f 1 {p})$ and $\X^{\epsilon}=(X^\epsilon, \XX^\epsilon)$ such that 
	\begin{align*}
	\Vert {X}^{\epsilon}_{s,t} \Vert_{L^{p}} \lesssim \vert t-s \vert^{\theta}, \qquad 
	\Vert {\XX}^{\epsilon}_{s,t} \Vert_{L^{\f p 2}} \lesssim \vert t-s \vert^{ 2 \theta},
	\end{align*}
	then 
	$$\sup_{\epsilon \in (0,1]} \E \left(\Vert\X^{\epsilon}\Vert_{\gamma} \right)^{p} < \infty$$ 
\end{lemma}
\begin{proof}
	The proof is based on a Besov-H\"older embedding, for details we refer to \cite{Friz-Victoir,Chevyrev-Friz-Korepanov-Melbourne-Zhang}.
\end{proof}
\begin{lemma}\label{tightness-second-order}
	Let $\X^{\epsilon}$ be a sequence of rough paths and  $\gamma \in ( \f 1 3, \f 1 2 - \f 1 {p})$, such that $\X(0)=0$, and $$\sup_{\epsilon \in (0,1]}\E \left(\Vert \X^{\epsilon}\Vert_{\gamma} \right)^{p} < \infty,$$
	then $\X^{\epsilon}$ is tight in $\FC^{\gamma}$.
\end{lemma}
\begin{proof}
	Choose $\gamma' \in ( \f 1 3 ,\gamma)$, as $\rho_{\gamma'}(\X,0) \leq \Vert \X \Vert_{\gamma'} + \Vert \X \Vert_{\gamma'}^2$ we obtain
	\begin{align*}
	\P  \left( \rho_{\gamma'} (\X^{\epsilon}, 0) > R \right) &\leq \f {\E \left(  \rho_{\gamma'} (\X^{\epsilon}, 0) \right)^{\f p 2}} {R^{\f p 2}}
	 \leq \f {\E \left(  \Vert \X \Vert_{\gamma'} + \Vert \X \Vert_{\gamma'}^2 \right)^{\f p 2}} {R^{\f p 2}}
	\lesssim \f C {R^{\f p 2}}.
	\end{align*}
	This proves the claim by Lemma \ref{compact}.
\end{proof}

\subsection{Example: linear driver and kinetic fBM}
 \label{example-fou}

Here we consider the toy model on $\R$,
\begin{equation}\label{example}
\left\{ \begin{aligned}
&\dot x_t^\epsilon= \epsilon^{H-1} f(x_t^\epsilon) \, y^\epsilon_t,   
\\& x_0^\epsilon=x_0.
\end{aligned}
\right.
\end{equation}
We study this without using any of the complicated estimates obtained later,  nevertheless 
we  are already able to explain our methodology.

\begin{proposition}\label{th-example}
\
\begin{enumerate}
\item [(a)] Let $H\in (0,1)$, $\gamma \in (0,H)$, $p>1$ and fix a finite time~$T$. 
Let $X^{\epsilon}_t = \epsilon^{H-1} \int_0^{t } y^\epsilon _{s} ds$, then, for $s,t \in [0,T]$,
$$\sup_{s, t \in [0,T]} \left \| X_{s,t}^\epsilon-B_{s,t}\right\|_{L^p} \lesssim \epsilon^H, \qquad 
   \left\| \left|X^\epsilon -B\right|_{\C^{\gamma'}([0,t]}\right\|_{L^p} \lesssim  t^{\gamma} \epsilon^{H-\gamma},$$
 for any $\gamma'<\gamma<H$.
 \item [(b)]  Let  $H \in ( \f 1 3, 1)$ and $f\in \C_b^3$.  Then for any $\gamma \in (0,H)$,  $x_t^\epsilon$  converges in $L^p$
 in $\C^{\gamma'}([0,T]); \R^d)$
 to the solution of the rough differential equation:
  \begin{equation}\dot x_t=f(x_t) \;dB_t,\end{equation}
furthermore, for $t \in [0,T]$,
$$\left\|  |x^\epsilon -x|_{C^{\gamma'}([0,t])}  \; \right\|_{L^p}  \lesssim t^{\gamma}  \epsilon^{H-\gamma}.$$
\end{enumerate} 
\end{proposition}
\begin{proof}
(a) 
 Set  
$v_t^\epsilon=\epsilon ^{H-1}  y^\epsilon_{t}$,  then $v_t^\epsilon$ solves the following equation
$$dv_t^\epsilon=-\f 1 \epsilon v_t^\epsilon dt +\f 1 \epsilon d B_t$$
Using the equation for $v_t^\epsilon$ we have
$$X_{s,t}^\epsilon 
= \epsilon^{H-1} \int_s^{t } y^\epsilon _{r} dr  =\int_s^t v_r^\epsilon dr =\epsilon (v^{\epsilon}_s-v_t^\epsilon) + B_{s,t}.$$
 Therefore, for any $p>1$,
\begin{align*}
\sup_{s,t \in [0,T]} \left \| X_{s,t}^\epsilon-B_{s,t}\right\|_{L^p}  &=\sup_{s,t \in [0,T]} \left \| \epsilon (v_t^\epsilon-v_s^{\epsilon})\right\|_{L^p} 
\\&=\epsilon ^{H} \sup_{s,t \in [0,T]} \left\| y^{\epsilon}_{t}- y^{\epsilon}_s \right\|_{L^p} \lesssim \epsilon^H.
\end{align*}
In the last step we used the stationarity  of $ y_t^\epsilon$, which follows from that of  $y_t $. 
By Lemma \ref{cty-lemma}, we have the following estimates on their H\"older norms for any $t \in [0,T]$:
$$\||X^\epsilon-B|_{\C^{\gamma'}([0,t])}\|_{L^p} \lesssim \epsilon^H \left({\f t \epsilon}\right)^{\gamma},
$$  
 this holds for any $p>1$ and any $\gamma'<\gamma$, and part (a) follows.

(b)  The system of equations is clearly well posed and has global solutions.
The idea is to consider the equation as a differential equation driven by rough paths $\X^\epsilon$ as below:
$$\dot x_t^\epsilon =f(x_t^\epsilon) d\X_t^\epsilon$$
 For $H \in (\f 1 2,1)$, the integral is simply the Young integral, and Young's continuity theorem states that $x^\epsilon$ converges
 provided $\X^\epsilon$ converges.  For $H \in (\f 1 3, \f 1 2)$ it is only left to deal with the geometric rough path lift of $x^\epsilon$, which in dimension one has only symmetric part and so the convergence in the $\C^\gamma$ topology is the same as convergence in the rough path topology.
 
For the $L^p$ convergence, we start with the $L^p$ convergence of the drivers $X^\epsilon$. Since $X_t^\epsilon$ is in $\C^\gamma$ for some  $\gamma>\f 13$,  The solutions map $\Phi$ 
for this equation is Lipschitz continuous: for any $\gamma'<\gamma$,
$$|\Phi(X^\epsilon)-\Phi(B)|_{\gamma'} \lesssim  |X^\epsilon- B|_\gamma.$$
This shows the convergence in $L^p$ of the solutions.
\end{proof}

 \begin{remark}
From part (a) of the theorem we deduce that    the solution to  the equation
\begin{equation}\label{example-sde-2}
	\dot z_t^\epsilon  =v_t^\epsilon, \quad  dv_t^\epsilon=-\f 1 \epsilon v_t^\epsilon dt +\f 1 \epsilon dB_t^H, \quad 
	z_0^\epsilon\to z_0, \quad v_0^\epsilon=\epsilon^{H-1} y_0,
	\end{equation}
converges in $\C^\gamma$ weakly to  a fractional Brownian motion $z_0+\tilde B_t$. Krammer-Smoluchowski limits /Kinetic fBM's  are studied in
 \cite{Boufoussi-Ciprian, Zhang-08,Al-Talibi-Hilbert}. See also \cite{Fannjiang-Komorowski-2000, Friz-Gassiat-Lyons, Friz-Hairer}.
Here $y_0$ is distributed as $\mu$. But our initial condition is not optimal. Since the solution depends on the initial condition affinely,
  stronger scaling and $y_0$ should also work.  
\end{remark}

 \begin{remark}\label{1d-lift-is-trivial}
We explain this with the previous example.  In one dimension the canonical lift of a process  is a function of the process itself:
 \begin{align*}
\int_0^t X^\epsilon_s dX^\epsilon_s &=\epsilon^{2H-2}  \int_0^{t} \int_0^s y^{\epsilon}_s y^{\epsilon}_r dr ds
= \f 1 2  (X_t^\epsilon)^2 \to \f 1 2 \left( B_t \right)^2.
\end{align*}
This is because taking squares is a continuous operation in $\C^\alpha$,  the convergence of $(X_t^\epsilon, \int_0^t X^\epsilon_s dX^\epsilon_s)$  follows. 
By part (1), $X_t^\epsilon$ is uniformly bounded in $\FC^\gamma$ for any $\gamma<H$.  Since  we can choose $\gamma>\f 13$, 
 to show convergence in the rough path topologies we only need to establish the moment bounds of the iterated integrals. In our case,
 it follows from Proposition \ref{th-example}:
\begin{align*}
\left\Vert \int_v^t (X^\epsilon_s - X^\epsilon_v) dX^\epsilon_s \right\Vert_{L^p} 
&= \f 1 2  \left \Vert \epsilon^{2H-2}  \int_v^{t} \int_v^s y^{\epsilon}_s  y^{\epsilon}_r dr ds \right\Vert_{L^p} \\
&\leq \f 1 2   \left(\left \Vert  \epsilon^{H-1}  \int_v^{t} y^{\epsilon}_s  ds \right \Vert_{L^{2p}} \right)^{\f 1 p}\\
&\lesssim  |t-v|^{2H}.
\end{align*}

The general $G$ case will be discussed in  section \ref{section-single-scale}, before that we will  make use of Malliavin calculus to obtain the $L^p$ estimates and  a suitable central/functional  limit theorem.
 We then have to explore different scaling constants,  since the limits constitutes of components  not necessarily simultaneously in  the same universality classes. If the equation involves two functions $G_i$ with different decay rates, to which different methods were used for the convergence, 
 we work harder for their joint convergence and relatively compactness in a topology suitable for all the limiting classes.

This means that $X^\epsilon$, with its canonical lift to the step-2 rough path space, is tight in the rough path topology $\FC^\gamma$ for any $\gamma<H$.
 \end{remark}

\section{Enhanced functional limit theorem in 1-d}\label{single}

We first prove a central limit theorem  for functionals of the fractional Ornstein-Uhlenbck processes, with convergence in  finite dimensional distributions. Then we go ahead and establish moment bounds   
on the increments of our process to conclude convergence in a suitable space of rough paths. This enables us to use the continuity of the solution maps to Young \ rough differential equations to prove our single scale homogenization result.

\subsection{Convergence in finite dimensional distributions } \label{s-CLT}
The scaling limit for path integrals of functionals of the fOU can be either in the Gaussian or in the non-Gaussian universality
classes.  Traditionally the first ones are called CLT's and the latter non-CLT's.  We would call both CLT's.

\subsubsection{ CLT and non CLT for sequences}
The intuition for scalar valued CLT's comes from its counter part for sequences which we  explain in the next paragraph. If $Y_n$ is a mean zero, stationary,  and strong mixing sequence, such that
$$\sigma_n^2=\E (\sum_{i=1}^n Y_i)^2\to \infty, \qquad \E(\sum_{i=1}^n Y_i)^4=O(\sigma_n^4),$$
then the CLT holds:
$$ \f 1 {\sigma_n} \sum_{i=1}^n Y_i {\longrightarrow}  N(0,1).$$
If $Y_n$ is not strong mixing,  the CLT may fail. An example of which are Gaussian sequences with slow decaying correlation functions.
A guiding principle, for Gaussian sequences, can be found in
\cite{Breuer-Major} for short range correlations, and in \cite{Taqqu, Bai-Taqqu} for long range correlations. 
A simple version is as follows.

 Let $X_n$ be a sequence of stationary mean zero variance $1$ Gaussian random variables with auto-correlation $ {n^{-\gamma}}$ with $\gamma\in  (0, 1)$.
Let $G$ be a function with Hermite rank $m\ge 1$ and $A(n)$  a sequence such that
$$\lim_{n\to \infty} \var\left(\f 1 {A(n)} \sum_{k=1}^n G(X_k)\right) =1.$$
\begin{enumerate}
\item If $\gamma\in  (\f 1 m, 1)$,  then the theorems following holds in finite dimensional distributions,
$$\f 1 {A(n)} \sum_{k=1}^{[nt]} G(X_n)\to W(t).$$ 
\item If $\gamma\in (0, \f 1m)$, then the scaling limit is a Hermite process in the $m$-th chaos.
\item If  $\gamma=\f 1m$, then, the scaling limit is also a Wiener process. \end{enumerate} 
The scaling constant is of the order $n^{1-\f 12 \gamma m}$
in the second case,  of order $\sqrt n$ for the first  case, and of order $\sqrt {n \ln n}$ for (3). From this the continuous version CLT for $\gamma \in (0, \f 1m)$ was obtained in \cite{BenHariz}. The borderline case $\gamma = \f 1 m$ for the continuous version was analysed in \cite{Buchmann-Ngai-bordercase}.

\subsubsection{Functional CLT in finite dimensional distributions}

If $y_t$ is any stationary Gaussian process with correlation function $\rho$, then for  $m\ge 1$,
\begin{align*}
\E (H_n(y_t)H_m(y_s))&=\delta_{n,m}\left(\rho(s,t)\right)^m.
\end{align*}
Since \begin{align*}
\sqrt{ \E \left( \int_0^{t} H_m(y^{\epsilon}_s) ds \right)^2} =  \left( \int_0^{t}\int_0^{t} \rho^{\epsilon}(\vert s-r \vert)^{m} dr ds \right)^{\f 1 2},
\end{align*} 
composing this with Lemma \ref{Integrals}, we therefore expect that the correct scaling to be  $\f {1} {\sqrt{\epsilon}}$  for the case $H^*(m)< \f 1 2$; $\sqrt{\f 1  {\epsilon |\ln \epsilon|} }$  for the case $H^*(m)=\f 12$;
and  $\epsilon ^{H^*(m)-1}$ otherwise.
Observing that 
$$ \alpha(\epsilon,H^*(m)) \int_0^t H_m(y^{\epsilon}_s)ds \sim \epsilon \alpha(\epsilon,H^*(m)) \int_0^{\f t \epsilon} H_m(y_s)ds.$$
This suggest that the self-similarity of the limiting process are determined by $\alpha(\epsilon,H^*(m))$.
In the first two cases the limit will be a Wiener process,  and in the later one the limit $Z_t$ should have the scaling property: 
$$\epsilon^{H^*(m)} Z_{\f t\epsilon} \sim Z_t.$$
These limits turn out indeed to be the Hermite processes. 

We first consider $G$ with low Hermite rank: $H^*(m)> \f 1 2$. Since $m\ge1$ this  restricts to the case  $H> \f 1 2$.
\begin{lemma}\label{L^2-kernel}
Let $G=\sum_{k=m}^\infty c_k H_k$ where $m>0$ be  in $L^2(\mu)$. Let $H\in ( \f 1 2 ,1)$.
Then the following statements hold  for  the  stationary fOU $y^{\epsilon}_t$. 
If  $H^*(m)> \f 1 2$ then
$$\left\|  \epsilon^{H^*(m)-1}  \int_{0}^{t} G(y^{\epsilon}_s) ds- \f {c_m m!} { K(H^*(m),m)} Z^{H^*(m),m}_t\right\|_{L_2(\Omega)} \to 0.$$
\end{lemma}

\begin{proof}
	This result looks slightly mysterious which can be explained easily by kernel convergence, since $H^*(m)$ decreases with $m$,
	  it is sufficient to work with  $H_m$ for $m$  the Hermite rank of $G$, c.f. Equation (\ref{reduction_Hermite}).  
	The key idea is to write a Wiener integral representations for these integrals beginning with
	\begin{equation}\label{y-integral}
	\begin{aligned}
	y_t^\epsilon&=\epsilon^{-H} \int_{-\infty}^t e^{-\f {t-r} \epsilon} dB_r=\int_{\R} h_\epsilon(t,s) dW_s, \quad \hbox{ where}\\
	h_\epsilon(t,s)&= \epsilon^{-\f 12}\f {1} {c_1(H)} e^{-\f {t-s} \epsilon}  \int_0^{\f {t-s} \epsilon}  e^v v_{+}^{H-\f 32}\; dv,
	\end{aligned}\end{equation}
and  $c_1(H) = \sqrt{ \int_{-\infty}^0 \left( (1-s)^{H-\f 1 2} - (-s)^{H- \f 1 2} \right)^2 ds + \f 1 {2H} }$.
	This leads to, using properties of Hermite polynomials, the following  multiple Wiener integral representation:
	\begin{equation}\label{explicit-kernel}
	\begin{split}
	&\epsilon^{H^*(m)-1}  \int_{0}^{t} H_m(y^{\epsilon}_s) ds\\
	& =    \f {\epsilon^{H^*(m)-1}} {m!} \int_{\R^m}\left(  \int_0^t
	\prod_{i=1}^m h_\epsilon(s, u_i)  ds \; \right) dW_{u_1} \dots dW_{u_m}. \end{split}
	\end{equation}
The $L^2$ convergence then follows from the following lemma.
\begin{lemma}\label{kernel lemma} As $\epsilon\to 0$, 
$\epsilon^{H^*(m) -1} \int_0^t H_m(y^{\epsilon}_s) ds$ converges to $ \f { m!} { K(H,m)} Z^{H^*(m),m}$ in $L^2$. 
Equivalently,
	$$\left \Vert \int_0^t
	\prod_{i=1}^m h_\epsilon(s, u_i)  ds -   \int_0^t \prod_{i=1}^m (s- u_i)_+^{H-\f 32}ds\right \Vert_{L^2(\R^m)} \to 0. $$
\end{lemma}
This is shown in Appendix \ref{kernel-condition} by applying  \cite[Theorem 4.7]{Taqqu}, where weak convergence is obtained, and making a small modification.
\medskip  

We proceed with the $L^2$ convergence. The Wiener integral representation for $y_t^\epsilon$, (\ref{y-integral}) can be obtained  by applying the integral representation for fBM's:
$$\begin{aligned}
B_t^H=\int_{-\infty}^t g(t,s)dW_s,  \qquad  \hbox{ where } \quad
g(t,s)= \f {1} {c_1(H)} \int_{-\infty}^t  (r-s)_+^{H-\f 32} dr,\end{aligned}$$ and 
by repeated applications of integration by parts (to the Young integrals):
\begin{align*}
\sigma \int_{-\infty}^t e^{- \f {t-s} {\epsilon}} dB^H_s &= \sigma B^H_t - \f \sigma \epsilon \int_{-\infty}^t e^{- \f {t-s} {\epsilon}} B^H_s ds\\
&= \sigma  B_t^H - \f \sigma \epsilon \int_{-\infty}^t  e^{- \f {t-s} {\epsilon}} \left(  \int_{\R} g(s,r) dW_r \right)ds\\
&= \sigma B_t^H - \f \sigma \epsilon \int_{-\infty}^t e^{- \f {t-s} {\epsilon}} g(s,r) ds  dW_r \\
&=\sigma  \int_{\R} \int_{-\infty}^t e^{- \f {t-s} {\epsilon}} \partial_s g(s,r) ds dW_r \\
&= \f{\sigma} {c_1(H)} \int_{\R} \int_{-\infty}^t e^{- \f {t-s} {\epsilon}} (s-r)_{+}^{H- \f 3 2} ds dW_r.
\end{align*}
Alternatively,  one may use the following  for Wiener integrals  \cite{Pipiras-Taqqu}, taking $f\in L^1\cap L^2$:
$$
\int_{\R} f(u) dB^H_u = \f {1} {c_1(H)}\int_{\R}  \int_{\R} f(u) (u-s)_{+}^{H- \f 3 2} \,du  \,dW_s.$$

To return to the case of a general $G$, we apply Lemma \ref{Lp-bounds} to $G-c_m H_m$. 
\begin{equation}\label{reduction_Hermite}
\begin{split}
&\left\| \epsilon^{H^*(m)-1}  \int_{0}^{t} (G-c_m H_m)(y^{\epsilon}_s) ds\right\|_{L^p}\\
&\lesssim  \Vert G-c_m H_m \Vert_{L^p} \f{ t^{H^*(m+1)}}{ \epsilon^{H^*(m+1)-H^*(m)}}
\to 0
\end{split}
\end{equation}
as $H^*$  decreases with $m$.
This shows, in particular that in the high Hermite rank case only the first non-zero term in the Hermite expansion of $G$ contributes to the limit.
This finishes the proof
 \end{proof}

\begin{lemma}\label{lemma-single-CLT-fdd}
Let $G=\sum_{k=m}^\infty c_k H_k$, where $m>0$, be  in $L^2(\mu)$.  
 Then the following statements hold  for   the fast stationary fractional Ornstein-Uhlenbeck process $y^{\epsilon}_t$  for
every parameter $H\in (0,1)$. Let $T>0$ and $c$ be given as in Equation (\ref{c-square}),
\begin{itemize}
 \item[(a)]  If $H^*(m) < \f 1 2$ (and for all $m\in \N$ in case $H=\f 12$),   
  then,   $$ \f {1} {\sqrt\epsilon}  \int_{0} ^{t} G(y^{\epsilon}_s) ds\to c \, \hat W_t, $$
  in the sense of finite dimensional distributions.
 
\item [(b)] If $H^*(m) = \f 1 2$, then, 
$$ \sqrt{\f 1 {\epsilon |\ln \epsilon|} } \int_{0} ^{t} G(y^{\epsilon}_s) ds\to 
c  \, \hat W_t,$$ in the sense of finite dimensional distributions.
\item[(c)]  Finally if  $H^*(m)> \f 1 2$, then,
$$ \epsilon^{H^*(m)-1}  \int_{0}^{t} G(y^{\epsilon}_s) ds\to  c Z^{H^*(m),m}_t,$$
in the sense of finite dimensional distributions. 
\end{itemize}
\end{lemma}

\begin{proof}

As mentioned above, convergence in finite dimensional distributions  in case (a) was shown in \cite{BenHariz} and for case (b) in \cite{Buchmann-Ngai-bordercase}.
For case (c) Lemma \ref{L^2-kernel} proves $L^2$ convergence , thus , in particular convergence in finite dimensional distributions.
This finishes the proof in all cases.
\end{proof}

\begin{remark}
\label{remark-variance}
\	
\begin{enumerate}
\item 
For $H<\f 12$ and $m=1$, Proposition \ref{lemma-single-CLT-fdd} appears to contradict with Proposition \ref{th-example}: in the first we claim the limit is a Brownian motion and in the second we claim it is  a fraction Brownian motion.  Both results are correct and can be easily explained.   It lies in the fact that $\int_\R \rho(s)\, ds$ vanishes if $H<\f 12$, and so the 
Brownian motion limit is degenerate. 
 Since  according to   \cite{Cheridito-Kawaguchi-Maejima}, \begin{equation}\label{cor5}
\rho(s) = \sigma^2 \frac{\Gamma(2H+1) \sin(\pi H)}{2 \pi} \int_{\R} e^{isx} \frac{\vert x \vert^{1-2H}}{1 + x^2} dx,
\end{equation}
and by the decay estimate from (\ref{cor1}), $\rho$ is integrable,
$s(\lambda) $ is the value at zero of the inverse Fourier transform of $\rho(s)$, which is
up to a multiplicative constant  $\frac{\vert \lambda \vert^{1-2H}}{1 + \lambda^2}$.
This is also the spectral density of $y_t$ and has value $0$ at $0$. 
This means we have scaled too much and the correct scaling is to multiply the integral $\int_0^{\f t \epsilon} y_s ds$ by 
$\epsilon^{H}$ in which case we have $B_t^H$ as a limit.
\item For $m>1$ and $H<\f 12$, the Gaussian limit is not trivial. Indeed,
$$\begin{aligned}
\int_\R  \rho(s)^m ds
&=C \int_\R \stackrel{m} {\overbrace{ \int _\R \dots \int_{\R} }}   \prod_{k=1}^m e^{isx_k} \frac{\vert x_k \vert^{1-2H}}{1 + |x_k|^2} dx_1\dots dx_m\, ds\\
&= C  \stackrel{m} {\overbrace{  \int _\R \dots \int_{\R}}} \frac{\vert x_2+\dots + x_m \vert^{1-2H}}{1 + |x_2+\dots + x_m|^2} \prod_{k=2}^m \, \frac{\vert x_k \vert^{1-2H}}{1 + |x_k|^2} \not =0.
\end{aligned}$$

\item For $H=\f 12$ and for any $m\in \N$,  the CLT is included in part (1), as the Ornstein-Uhlenbeck process driven by a Wiener process has exponentially decaying 
correlations.

\end{enumerate}
\end{remark}

 In the next section we bound the $L^p$ norm of the random variable $$X^\epsilon:= \alpha(\epsilon, H^*(m))\int_0^{t} G(y^{\epsilon}_r)  dr $$ where 
$y^{\epsilon}_t$ is the  rescaled stationary fractional Ornstein-Uhlenbeck process and $G$ an $L^p$ function of Hermite rank at least one.
Since  $\E \left( H_m(y_r) H_m(y_s) \right) \sim \left(\E \left( y_r y_s \right)\right)^m$,  these are trivial to obtain for functions in the finite chaos expansion.
We  show that the upper bounded of its $L^p$ norm  is of order $\f {1} {\alpha(\epsilon,H^*(m))} $. Hence it is of order $\f 1 {\sqrt \epsilon }$ if and only if $H^*(m)<\f 12$; otherwise it is one of the higher orders :  $\f {|\ln \epsilon|}  {\sqrt \epsilon }$ or $\epsilon ^{-H^*(m)+1}$.

\subsection{Moment bounds}
\label{pre}

We will use some results from Malliavin Calculus. Let $x_s$ be a stationary Gaussian process with $\alpha(s) = \E \left( x_s x_0 \right)$, such that $\alpha(0)=1$. As a real separable Hilbert space we use  $\mathscr{H} = L^2(\R_+,\nu)$ where for a Borel-set $A$ we have $\nu(A) = \int_{\R_+} \1_{A} (s) d\alpha_s$. We can  replace $\R_+$ by $\R$ or by $[0,1]$. Let $\H^{\otimes q}$ denote the $q$-th tensor product of $\H$. For $h\in \H$, we may define the Wiener integrals $W(h)=\int_0^\infty h_s dx_s$ by $W([a,b])=x(b)-x(a)$ (where $a, b\ge 0$), linearity and the Wiener isometry 
($\<\1_{[0,t]}, \1_{[0,s]}\>=\alpha(t-s)$).
Iterated Wiener integrals are defined similarly and by its values on indictor functions:
$I_m(\1_{A_1\times \dots \times A_m})=\prod_{i=1}^m W (A_i)$ where $A_i$ are pairwise disjoint Borel subsets of $\R_+$.
If $\F$ denotes the $\sigma$-field generated by $x$, 
then any $\F$-measurable $L^2$ function $F$ has the chaos expansion:
$F=\E F+ \sum_{m=1}^\infty I_m(f_m)$ where $f_m\in L^2( \R_+^m)$, the latter space is with respect to the product measures. This is due to the fact that $L^2(\Omega)=\bigoplus_{m=0}^{\infty} \H_m$ where $\H_m$ is the closed linear space generated by $\{ H_m(W(h)): |h|_{L^2}=1\}$, $H_m$ are the $m$-th Hermite polynomials, and that $\H_m=I_m(L_{\hbox{sym}}^2(T^m))$.
The last fact is due to $H_m(W(h))= I_m( h^{\otimes ^{m} })$. In the following
${\mathbb{D}^{k,p}(\mathscr{H}^{\otimes m})}$ denotes the closure of Malliavin smooth random variables under the following norm
$\Vert u \Vert_{\mathbb{D}^{k,p}(\mathscr{H}^{\otimes m})} = \left( \sum_{j=0}^{k} \E \left( \Vert D^j u \Vert_{\mathscr{H}^{\otimes m}}^p  \right) \right)^{\f 1 p}$.

\begin{lemma}[Meyer's inequality] \cite{Nourdin-Peccati} 
	\label{Meyer}
	Let $\delta$ denote the  divergence operator (one can think of $\delta^m$ as an $m$ times iterated  Wiener-It\^o-integral), then  for $u \in \mathbb{D}^{k,p}(\mathscr{H}^{\otimes m})$,
	\begin{equation}
	\Vert \delta^m(u) \Vert_{L^p} \lesssim  \sum_{k=0}^m  \Vert u \Vert_{\mathbb{D}^{k,p}(\mathscr{H}^{\otimes m})}.
	\end{equation}
\end{lemma}

\begin{lemma}\label{representation}
	\cite{Campese-Nourdin-Nualart}
	If $G:\R\to \R$ is a function of Hermite rank $m$, then $G$ has the following  multiple Wiener-It\^o-integral representation:
	\begin{equation}\label{Lp-eq}
	G(x_s) = \delta^m\left( G_m(x_s) \1_{[0,s]}^{\otimes m}\right),
	\end{equation}
	where $G_m$ has the following properties: \begin{itemize}
		\item [(1)]   $\Vert G_m(x_1) \Vert_{L^p} \lesssim \Vert G(x_1) \Vert_{L^p}$,
		\item  [(2)]
		$G_m(x_1)$ is $m$ times Malliavin differentiable and its $k^{th}$ derivative, denoted by  $G_m^{(k)}(x_1)\1_{[0,1]}^{\otimes k}$,
		satisfies $ \Vert G_m^{(k)}(x_1) \Vert_{L^p} \lesssim  \Vert G(x_1) \Vert_{L^p}$.
	\end{itemize}
\end{lemma}
In the lemma below we estimate the moments of $\int_0^{t} G(x_{\f r \epsilon})  dr$, where we need the multiple Wiener-It\^o-integral representation above to transfer the
correlation function to $L^2$ norms of indicator functions. We use an idea from \cite{Campese-Nourdin-Nualart} for the
estimates below. 

\begin{lemma}\label{Lp-bounds} 
	Let  $x_t=W([0,t])$ be a stationary Gaussian process with correlation $\alpha(t)=\E(x_t x_0)$ and $\H$ the $L^2$ space over $\R_{+}$ with measure $\alpha(r) dr$.
	If $G$ is a function of Hermite rank $m$ and $ G \in L^p(\mu)$, then 
	\begin{equation}\label{le5.4-1}
	{\begin{split}
		\left\Vert  \f {1} {\epsilon} \int_0^{t } G(x_{\f r \epsilon})  dr \right\Vert_{L^p} 
		&\lesssim \Vert G \Vert_{L^p(\mu)} \left ( \int_0^{\f t \epsilon} \int_0^{\f t \epsilon} \alpha(\vert u -r \vert)^m dr du \right)^{\f 1 2}.\end{split}}
	\end{equation}
	For the fast fractional OU process $y^{\epsilon}_t$, we have
	\begin{equation}\label{le5.4-2}
	{\begin{split}
		\left\Vert  \f {1} {\epsilon} \int_0^{t } G( {y_r^\epsilon})  dr \right\Vert_{{L^p} } 
		&\lesssim \left\{	\begin{array}{lc}
		\Vert G \Vert_{L^p(\mu)}\; \sqrt {\f t \epsilon  \int_0^\infty \rho^m(s) ds } ,  \quad  &\hbox {if} \quad H^*(m)<\f 12,\\
		\Vert G \Vert_{L^p(\mu)} \; \sqrt {\f t \epsilon \ln| \f {1} \epsilon|}, \quad  &\hbox {if} \quad H^*(m)=\f 12,\\
		\Vert G \Vert_{L^p(\mu)} \; \left(  \f t \epsilon\right) ^{H^*(m)},  \quad &\hbox { otherwise.}
		\end{array} \right.
		\end{split}},
	\end{equation}
	in particular,
	\begin{equation}
	\left\Vert \int_0^{t } G(y^{\epsilon}_r)  dr \right\Vert_{L^p} \lesssim  \f {\Vert G \Vert_{L^p(\mu)} t^{H^*(m) \vee \f 1 2}} {\alpha(\epsilon,H^*(m))}.
	\end{equation}
	
\end{lemma}

\begin{proof}
	We first use Lemma \ref{representation} and then apply Meyer's inequality from Lemma \ref{Meyer} to obtain 
	\begin{align*}
	\left\Vert \f {1} {\epsilon} \int_0^{t } G(x_{\f r \epsilon})  dr \right\Vert_{L^p}   &=\left\Vert \int_0^{\f t \epsilon} G(x_r)  dr \right\Vert_{L^p}   \\
	&=\left \Vert \int_0^{\f t \epsilon}\delta^m\left( G_m(x_r) \1_{[0,r]}^{\otimes m}\right)\, dr\right \Vert_{L^p}  \\
	&\lesssim  \sum_{k=0}^m \left \Vert \int_0^{\f t \epsilon}  D^k \left(G_m(x_r) \1_{[0,r]}^{\otimes m}\right)  dr 
	\right\Vert_{L^{p}(\Omega,\mathscr{H}^{\otimes  m +k})}\\
	&= \sum_{k=0}^m\left \Vert \int_0^{\f t \epsilon}  G_m^{(k)} (x_r) \1_{[0,r]}^{\otimes {m +k}}  dr 
	\right\Vert_{L^{p}(\Omega,\mathscr{H}^{\otimes  m +k})}.
	\end{align*}
	We estimate the individual terms, Using linearity of the inner product, and the isometry $\<  \1_{[0,r]},  \1_{[0,s]}\>_\H=\E(x_rx_s)=\alpha(r-s)$,
	\begin{align*} &\left( \left \Vert \int_0^{\f t \epsilon}  G_m^{(k)} (x_r) \1_{[0,r]}^{\otimes {m +k}}  dr 
	\right\Vert_{\mathscr{H}^{\otimes  m +k}}\right)^2\\
	&=  \left\<  \int_0^{\f t \epsilon}  G_m^{(k)} (x_r) \1_{[0,r]}^{\otimes {m +k}}  dr ,  \int_0^{\f t \epsilon}  G_m^{(k)} (x_u) \1_{[0,r]}^{\otimes {m +k}}  du 
	\right\>_{\H^{\otimes  m +k}}\\
	& = \int_0^{\f t \epsilon} \int_0^{\f t \epsilon}    G_m^{(k)} (x_r) G_m^{(k)} (x_u) \<\1_{[0,r]}^{\otimes {m +k}} , \1_{[0,u]}^{\otimes {m +k}} \>_{\H^{\otimes  m +k}}\, dr du\\
	& = \int_0^{\f t \epsilon} \int_0^{\f t \epsilon}    G_m^{(k)} (x_r) G_m^{(k)} (x_u) \alpha(r-u)^{m+k} drdu.
	\end{align*}
	Using  Minkowski's inequality  we obtain
	\begin{align*}
	&\sum_{k=0}^m\left \Vert \int_0^{\f t \epsilon}  G_m^{(k)} (x_r) \1_{[0,r]}^{\otimes {m +k}}  dr 
	\right\Vert_{L^{p}(\Omega,\mathscr{H}^{\otimes  m +k})}\\
	&\leq  \sum_{k=0}^m \left( \left\Vert
	\int_0^{\f t \epsilon} \int_0^{\f t \epsilon}    G_m^{(k)} (x_r) G_m^{(k)} (x_u) \alpha(r-u)^{m+k} drdu
	\right\Vert_{L^{\f p 2}(\Omega)}\right)^{\f 1 2}\\
	&\leq  \sum_{k=0}^m \left( 
	\int_0^{\f t \epsilon} \int_0^{\f t \epsilon}   \left\Vert G_m^{(k)} (x_r) G_m^{(k)} (x_u)  \right\Vert_{L^{\f p 2}(\Omega) } \alpha(r-u)^{m+k} drdu
	\right)^{\f 1 2}.
	\end{align*}
	We then  estimate $\E |G_m^{(k)} (x_r) G_m^{(k)} (x_u)|^{\f p 2}$ by H\"older's inequality and the fact that $x_t$ is stationary. The right hand side is then controlled by
	\begin{align*}
	RHS & \leq  \sum_{k=0}^m \Vert G_m^{(k)}(x_1) \Vert_{L^p} 
	\left( \int_0^{\f t \epsilon}\int_0^{\f t \epsilon} \alpha(\vert u-r \vert)^{m+k} dr du\right)^{\f 1 2}\\
	&\lesssim  \Vert G \Vert_{L^p(\mu)} \left( \int_0^{\f t \epsilon}\int_0^{\f t \epsilon} \alpha(\vert u-r \vert)^{m} dr du \right)^{\f 1 2},
	\end{align*}
	concluding (\ref{le5.4-1}).
	We  finally apply Lemma \ref{Integrals} to conclude (\ref{le5.4-2}).
\end{proof}

\subsection{Limit theorems  in 1-d}\label{section-single-scale}

\begin{proposition}[Enhanced limit theorem 1-d]\label{proposition-single-CLT-strong-topologies}
Let $H\in (\f 1 3,1)$
and $G=\sum_{k=m}^\infty c_k H_k$ in $ L^p(\mu)$ where $m>0$ and $p>2$. Set $$X^{\epsilon}_t= \alpha(\epsilon,H^*(m))  \int_{0} ^{t} G(y^{\epsilon}_s) ds,   \qquad\XX^{\epsilon}_{s,t} = \int_0^t (X^{\epsilon}_{r} -X^{\epsilon}_{s}) dX^{\epsilon}_{r}.$$
	Let $T>0$, then for $c$ as in Equation (\ref{c-square}),	\begin{itemize}
		\item[(a)]  If $H^*(m) \leq \f 1 2$ (and for all $m\in \N$ in case $H=\f 12$),  
		then  for any $\gamma \in (0,\f 1 2 - \f 1 {p})$,   
		$$ \X^{\epsilon} = (X^{\epsilon},\XX^{\epsilon})\to  c \, \hat \W, $$
		weakly  in  $\FC^{\gamma} ([0,T])$, where  $\hat \W$ denotes a Wiener process enhanced with its Stratonovich integral.
		\item[(b)]  If  $H^*(m)> \f 1 2$ and  $\f 1 p< H^*(m) -  \f 1 2$, then for any $\gamma \in (0,H^*(m) - \f 1 {p})$,
		$$ X^{\epsilon}_t \to c Z^{H^*(m),m}_t,$$
		weakly  in $\C^{\gamma}([0,T]) $. 
	\end{itemize}
\end{proposition}
\begin{proof}
Convergence in finite dimensional distributions for both cases was shown in Lemma \ref{lemma-single-CLT-fdd}, using the moment bounds obtained in Lemma \ref{Lp-bounds} together with an application of Kolmogorv's Lemma we obtain that $X^{\epsilon}$ is tight in $\C^{\gamma}$ for the stated~$\gamma$.
Now $\XX_{s,t}^\epsilon=\f 12 (X_t^\epsilon-X_s^\epsilon)^2$, and so the joint convergence in the finite dimensional distribution follows.
The convergence in the rough path topology is concluded by Lemmas \ref{moment-conditions} and \ref{tightness-second-order}.
See also Remark \ref{1d-lift-is-trivial}.
\end{proof}

\subsection{Homogenization/proof of Thm C}
We have all the ingredients at hand for proving Theorem C. Consider
\begin{equation}\label{bm-eq}
dx_t^\epsilon =\alpha(\epsilon) f(x_t^\epsilon) G(y^{\epsilon}_t), \quad G= \sum_{k=m}^{\infty} c_k H_k
\end{equation}
where $\alpha(\epsilon) = \alpha(\epsilon,H^*(m))$.
We show $x_t^\epsilon\to x$  where $x$ is  the solution to
$\dot x_t=c \; f(x_t) \;dX_t$
 where $X_t$ is either a Wiener process or a Hermite processes, depending on $H^*(m)$, and the constant $c$ is given as in Equation \ref{c-square}.
\begin{proof}  \label{proof-theorem-C}
For the first case we rewrite our equation $dx_t^\epsilon =\alpha(\epsilon) f(x_t^\epsilon) G(y^{\epsilon}_t)$ into the Young differential equation $dx_t^\epsilon = f(x_t^\epsilon) dX^{\epsilon}_t$.
By Proposition \ref{proposition-single-CLT-strong-topologies} $X^{\epsilon}_t$ converges weakly in $\C^{\gamma}$ to a Hermite process  $Z^{H^*(m),m}$. Using the continuity of the solution map of Young differential equation with respect to its driver we obtain the first result.
Concerning the second result, we rewrite our equation $dx_t^\epsilon =\alpha(\epsilon) f(x_t^\epsilon) G(y^{\epsilon}_t)$ into the rough differential equation $dx_t^\epsilon = f(x_t^\epsilon) d\X^{\epsilon}_t$.
Now, by Proposition \ref{proposition-single-CLT-strong-topologies}, $\X^{\epsilon}$ converges weakly in $\FC^{\gamma}$ for $\gamma \in (\f 1 3, \f 1 2 - \f 1 p)$. Therefore, using the continuity Theorem \ref{cty-rough} finishes the proof.
\end{proof}

\section{Enhanced functional limit theorems}\label{joint}

The following convention is in place throughout this section unless otherwise stated.
\begin{convention}
 For $k=1, \dots, N$, each $G_k:\R\to \R$  is an $L_2$ function with  Hermite ranks $m_k \geq 1$.
Set $$X^{\epsilon}_t = \left( X^{1,\epsilon}_t, \dots X^{N,\epsilon}_t \right),$$
where 
\begin{equation}\label{X-component}
X^{k,\epsilon}_t= \alpha_k(\epsilon) \int_0^{t} G_k(y^\epsilon_s) ds, \qquad  \alpha_k(\epsilon)=\alpha(\epsilon,H^*(m_k)).
\end{equation}
Furthermore, we remind the reader that $n \geq 0$ is a number such that for $k \leq n$ we have $H^*(m_k) \leq \f 12$ and for $k>n$, $H^*(m_k)> \f 1 2$. We also set
\begin{align*}
X_{t}^{W,\epsilon} &= \left( X^{1,\epsilon}_t, \dots X^{n,\epsilon}_t \right)\\
X_{t}^{Z,\epsilon} &= \left( X^{n+1,\epsilon}_t, \dots X^{N,\epsilon}_t \right),
\end{align*}
so that  $X^{\epsilon}_t=(X^{W,\epsilon}_{t},X^{Z,\epsilon}_{t})$.
\end{convention}

\subsection{Convergence of the vector valued processes of mixed type}  
The convergence of each component has already been proved earlier, so it is only left to show that they converge jointly. We must specify the correlations between the limiting components.  If they converge jointly and if the limiting distribution is independent, 
 then the covariance has to converge to zero also, this we do not expect to hold in general.  For example if 
all $G_i$ are equal  then all the components of the limiting driver are the same.   If we have $H_i$ and $H_j$ where $i\not =j $, we may expect 
non-trivial correlations.  On the other hand we know different scales $y_{t/\epsilon}$ and $y_{t/\epsilon^\alpha}$ where $\alpha \not =1$ are `expected' to have uncorrelated scaling limits, this is reflected in the different scaling constants. First we will establish joint convergence under the assumption that each component converges to a Wiener process. 
We then show the joint convergence for the case each component limit is a Hermite process, and then for the case the component limits can be  either a 
Brownian motion or a Hermite process.  Due to a reduction lemma ( Lemma \ref{reduction} below), the joint convergence can be reduced to  $G_i$  being a finite sum of  Hermite polynomials.

\begin{lemma}[CLT-Gaussian]\label{CLT-Gaussian-multi}
\label{joint-clt}
Fix $H \in (0,1)\setminus\{\f 12\}$.
Here we consider the fist $n$ components of $X^{\epsilon}$, which are denoted by $X^{W,\epsilon}$, 
  so $G_k\in L^2(\mu)$, $k\le n$,  with Hermite rank $m_k>0$ and $H^*(m_k)\le \f 12$.
 \begin{enumerate}
\item Then, as  $\epsilon \to 0$, the following converges in  finite dimensional distribution:
$$X^{W,\epsilon} \to (X^1, X^2, \dots X^n) = X^W.
$$
 
\item The limiting distribution is Gaussian with covariance between the $i$th and the $jth$ components  given by
$$ \E [ X^{i}(t) X^{j}(s)] =2 (s \wedge t) \int_0^\infty \E (G_i(y_r) G_j(y_0) )dr.$$
\item If, in addition, $G_k \in L^{p_k}(\mu)$ for $p_k>2$, then  the convergence is weakly in $\C^\gamma$ where $\gamma \in (0, \min_{k=1 \dots n } \f 12-\f 1 {p_k})$.
\end{enumerate}
\end{lemma}

\begin{proof}

First we define the truncated functions $G_{k,M}=\sum_{j=m_k}^{M} c_{k,j} H_j$ and set
$$
X^{k, \epsilon}_M= \alpha_k(\epsilon) \int_0^{t} G_{k,M}(y^{\epsilon}_s) ds.
$$
Then, by  Lemma \ref{reduction} below,  it is sufficient to show the convergence of $(X^{1,\epsilon}_M, \dots , X^{n,\epsilon}_M)$ for every $M$. By earlier considerations each $X^{k,\epsilon}_M$
converges to a Wiener process $X^k_M$. As each $X^{k,\epsilon}_M$ belongs to a finite chaos we can make use of the normal approximation theorem from \cite[Theorem 6.2.3]{Nourdin-Peccati}: if each component of a family of mean zero vector valued stochastic processes, with components of the form $I_{q_i}(f_{i,n}) $ where $f_{i,n}$ are symmetric $L^2$ functions in $q_i$ variables,  converges in law to a Gaussian process, then they converge  jointly in law to a vector valued Gaussian process, provided that their correlation functions converge. Furthermore, the correlation functions of the limit distribution are:  $\lim_{\epsilon \to 0} \E [ X^{i,\epsilon}(t) X^{j,\epsilon}(s)] $.
Let $m= \min(m_i,m_j)$  we use $$\E (H_n(y_t)H_m(y_s))=\delta_{n,m} \left(\E(y^{\epsilon}_sy^{\epsilon}_t)\right)^m$$ to obtain, for $s \leq  t$, 
$$\begin{aligned} &\E\left[ \alpha_i(\epsilon) \alpha_j(\epsilon)\int_0^{t}G_{i,M}(y^{\epsilon}_u)du  \int_0^{s} G_{j,M}(y^{\epsilon}_r) dr \right] \\
&=  \sum_{k=m}^M \alpha_i(\epsilon) \alpha_j(\epsilon) c_{i,k} c_{j,k} (k!)^2 \int_0^{t}\int_0^{s} (\E(y^{\epsilon}_r y^{\epsilon}_u))^k  dr du\\
&=  \sum_{k=m}^M  \alpha_i(\epsilon) \alpha_j(\epsilon) c_{i,k} c_{j,k} (k!)^2 \left( \int_0^{s} \int_0^{s} \rho^{\epsilon}(u-r)^k dr du + \int_{s }^{t} \int_0^{s} \rho^{\epsilon}(u-r)^k  dr du \right)
\end{aligned}$$
By Lemma \ref{Integrals} we obtain,  for $\epsilon \to 0$,
\begin{align*}
\alpha_i(\epsilon) \alpha_j(\epsilon) \int_{s }^{t} \int_0^{s} \rho^{\epsilon}(u-r)^k dr du \to 0.
\end{align*} 
Hence,
$$\begin{aligned}
\lim_{\epsilon \to 0} RHS &= 2 \sum_{k=m}^M c_{i,k} c_{j,k} (k!)^2  \lim_{\epsilon \to 0} \left(\alpha_i(\epsilon) \alpha_j(\epsilon) s \int_0^{\f s {\epsilon}} (\rho(v))^k  dv\right)  \\
&= 2 \left( s \wedge t \right)  \sum_{k=m}^M c_{i,k} c_{j,k} (k!)^2 \int_0^\infty \rho(u)^k du\\
&= 2 \left( s \wedge t \right) \,\int_0^\infty \E (G_{i,M}(y_s) G_{j,M}(y_0) )ds
,\end{aligned}$$
proving the finite chaos case.	 
We now prove that the correlations of the limit converge as $ M \to \infty$. Indeed,
\begin{align*}
\lim_{M \to \infty}  2 \left( s \wedge t \right)  \sum_{k=m}^M c_{i,k} c_{j,k} (k!)^2 \int_0^\infty \rho(u)^k du
 &=  2 \left( s \wedge t \right) \sum_{k=m}^{\infty} c_{i,k} c_{j,k} (k!)^2 \int_0^\infty \rho(u)^k du\\
&=  2  \left( s \wedge t \right) \,\int_0^\infty \E (G_i(y_s) G_j(y_0) )ds.
\end{align*}
As $G_{i,M} \to G_i$ in $L^2$, and similarly for $j$, this proves the first two claims. 
The convergence in H\"older spaces follows from Lemma \ref{Lp-bounds}, which states these processes are tight in $\C^\gamma$, c.f. Proposition \ref{proposition-single-CLT-strong-topologies}.
This  concludes the proof. 
\end{proof}

Now we prove the reduction lemma for the high Hermite rank case.

\begin{lemma}[Reduction Lemma]\label{reduction}
Fix $H \in (0,1)\setminus\{\f 12\}$. 
For  $M\in \N$, define $$X^{k,\epsilon}_{M}(t)= \alpha_k(\epsilon) \int_0^{t} G_{k,M}(y^{\epsilon}_s)ds.$$
If for every $M \in \N$, 
$$(X^{1,\epsilon}_{M}, \dots, X^{N,\epsilon}_{M})\longrightarrow  (X^{1}_{M}, \dots, X^{N}_{M})$$ 
in finite dimensional distributions, then,
$$(X^{1,\epsilon}, \dots, X^{N,\epsilon}) \longrightarrow  (X^{1}, \dots, X^{N}),$$
in finite dimensional distributions.

 \end{lemma} 
\begin{proof}
Firstly we show,  for any sequence of positive numbers $ \{t_{\gamma_{k,l}}, k\le N, l\le A \} $,   
 $\sum_{k=1}^N \sum_{l=1}^A \gamma_{k,l} X^{k,\epsilon}_M(t_l)$ converges
as $M\to \infty$. By the triangle inequality we can reduce 
 $$\left\Vert \sum_{k,l} \gamma_{k,l} \left( X^{k,\epsilon} (t_l) - X^{k,\epsilon}_M  (t_l)  \right) \right\Vert_{L^2} \to 0.$$
 to $\Vert   X^{k,\epsilon}  (t) -   X^{k,\epsilon}_{M} (t) \Vert_{L^2} \to 0$.
Now,
	\begin{align*}
	 X^{k,\epsilon}(t) -  X^{k,\epsilon}_{M}(t) &=  \alpha_k(\epsilon) \int_0^{t} \Big( G_k(y^{\epsilon}_s)- G_{k,M} (y^{\epsilon}_s)\Big)ds\\
	 &
		=  \alpha_k(\epsilon) \int_0^{t} \sum_{j=M+1}^\infty c_{k,j} H_j(y^{\epsilon}_s)ds.
	\end{align*}
	Using properties of the Hermite polynomials we obtain
	\begin{align*}
		&\E \left(   \alpha_k(\epsilon) \int_0^{t} \sum_{j=M+1}^\infty c_{k,j} H_j(y^{\epsilon}_s)ds \right)^2 \\
		&=  \alpha_k(\epsilon)^2 \int_0^{t} \int_0^{t} \sum_{j=M+1}^\infty (c_{k,j})^2 \E \left(  H_j(y^{\epsilon}_s) H_j(y^{\epsilon}_r) \right) dr ds\\
		&=   \alpha_k(\epsilon)^2 \sum_{j=M+1}^\infty (c_{k,j})^2 j! \int_0^{t} \int_0^{t}  \rho^{\epsilon}(\vert s-r \vert)^j dr ds\\
		&\lesssim  t \int_0^\infty \rho(u)^{M+1}du \,\sum_{j=M+1}^\infty (c_{k,j})^2 j!
	\end{align*}
	As $\sum_{j=m}^\infty (c_{k,j})^2 j! < \infty$ we obtain $ \sum_{j=M+1}^\infty (c_{k,j})^2 j!\to 0$ as  $M \to \infty$.  Then,
\begin{equation}\label{Bill}
\lim_{M \to \infty}\lim_{\epsilon \to 0}\E \left(   \alpha_k(\epsilon) \int_0^{t} G_k(y^{\epsilon}_s)ds -  \alpha_k(\epsilon) \int_0^{t} G_{k,M}(y^{\epsilon}_s)ds \right)^2 \to 0,
\end{equation}
and finally using theorem 3.2 in \cite{Billingsley}  we conclude  the proof.
\end{proof}

Now, we go ahead and deal with the low Hermite rank case, so focus on the vector component whose entries satisfy $H^*(m_k)> \f 1 2$. Recall,
 this implies ${H> \f 1 2}$. 
\begin{lemma}[CLT- Hermite]
\label{joint-clt-2}
Fix $H \in  (0,1)\setminus\{\f 12\}$. Write $G_k=\sum_{j=m_k} ^N c_{k,j} H_j$.
Suppose that $m_k\ge 1$ and $ H^*(m_k) > \f 1 2$.
\begin{enumerate}
\item Then,  $(X^{n+1,\epsilon} , X^{2,\epsilon}  \dots, X^{N,\epsilon} )$ converges in finite dimensional distribution and 
for every $t \in [0,T]$ \begin{align*}
\left\|(X^{n+1,\epsilon}_t , X^{n+2,\epsilon}_t  \dots, X^{N,\epsilon}_t )- (X^{n+1}_t, X^{n+2}_t,   \dots, X^N_t )\right\|_{L_2(\Omega)} \to 0.
\end{align*} 
\item The marginals of the limit are the following Hermite processes, each given by the representation \ref{Hermite} with a common Wiener process $W$ 
 $$X^k=  \frac{ c_{k,m_k} m_k!}{K(H^*(m_k),m_k)} Z^{H^*(m_k),m_k}.$$ 

\item If in addition $G_k \in L^{p_k}$ for $p_k>2$, then  the convergence is weakly in $\C^\gamma$ on any finite time interval where $\gamma \in (0, \min_{k= n+1, \dots , N} H^*(m_k )-\f 1 {p_k})$.

\end{enumerate}\end{lemma}
\begin{proof}
By Lemma \ref{L^2-kernel} each component convergences in $L^2$, hence the above converges as well in $L^2$ yielding convergence in finite dimensional distributions by an application of the Cramer-Wold theorem.
 The convergence in H\"older spaces follows from Lemma \ref{Lp-bounds}, which states these processes are tight in $\C^\gamma$, c.f. Proposition \ref{proposition-single-CLT-strong-topologies}.
\end{proof}

\begin{proposition}[CLT-mixed]\label{joint-clt-3}
For each $k$, write  $$G_k=\sum_{j=m_k}^\infty c_{k,j} H_j.$$
 Then, for $X^k$ given in Lemmas \ref{CLT-Gaussian-multi} and  \ref{joint-clt-2},
\begin{align*}
X^{\epsilon} =(X^{1,\epsilon} , X^{2,\epsilon}  \dots, X^{N,\epsilon} )\longrightarrow (X^{1} , X^{2}  \dots, X^{N} )= (X^W,X^Z)
\end{align*}
in  finite dimensional distributions. Furthermore, 
\begin{enumerate} 
\item $X^W$ and $X^Z$ are independent.
\item
 If $i, j\le n$, so $X^i$ and $X^j$ are Gaussian, their correlation is
$$ 2(s\wedge t) \int_0^\infty  \int_0^\infty \E (G_i(y_s)G_j(y_0))\,ds.$$
\item If $i,j \ge n+1$, both $X^i$ and $X^j$ are Hermite processes, then their correlation is given by a common Wiener
process $W_t$. Specifically,  
 \begin{align*}
&\cov(X^i,X^j) \\
=& \delta_{m_i,m_j} c_{i,m_i} c_{j,m_j} \int_{0}^t \! \! \int_{0}^t  \! 
\left(  \int_{\R}\left( s - \xi\right)_{+}^{\hat H(m) - \f 3 2}  \left( r - \xi\right)_{+}^{\hat H(m) - \f 3 2} d\xi \right)^{m_i}dr ds. 
\end{align*}
\item  If in addition $G_k \in L^{p_k}$ for $p_k>2$, then  the convergence is weakly in $\C^\gamma$ for every $\gamma \in (0, \min_{k=1, \dots n} \f 12 - \f 1 {p_k} \wedge \min_{k=n+1, \dots N} H^*(m_k) - \f 1 {p_k} )$.
\end{enumerate}
\end{proposition}
\begin{proof}
Using Lemma \ref{CLT-Gaussian-multi} and \ref{joint-clt-2}, $X^{W,\epsilon} \to X^W$ and $X^{Z,\epsilon} \to X^Z$ in finite dimensional distributions.
 By  Lemma \ref{reduction} and Equation (\ref{reduction_Hermite}) , we may reduce the problem to 
 $$G_i= \sum_{k=m_i}^M c_{i,k} H_k, \quad   G_j=  c_{j,m_j} H_{m_j},  \qquad 1 \leq i \leq n,\; j>n.$$
 Now, we can rewrite $H_m(y^{\epsilon}_s)= I_m(f^{m,\epsilon}_s)$, where $I_m$ denotes a $m$-fold Wiener-Ito integral and a function $f^{m,\epsilon}_s \in L^2(\R^m,\mu)$. 
 Now, for $1 \leq i \leq n$ we obtain,
 \begin{align*}
\alpha_i(\epsilon) \int_0^t G_i(y^{\epsilon}_s)ds &= \alpha_i(\epsilon) \int_0^t \sum_{k=m_i}^{M} c_{i,k} H_k(y^{\epsilon}_s) ds
\\
&=  \alpha_i(\epsilon)\int_0^t \sum_{k=m_i}^{M} c_{i,k} I_k(f^{k,\epsilon}_s) ds= \sum_{k=m_i}^{M} c_{i,k} I_k(\hat {f}^{k,\epsilon}_t),
 \end{align*}
where  $\hat f^{k,\epsilon} = \int_0^t f^{k,\epsilon}_s ds$. Similarly for $j>n$,   
 \begin{align*}
\int_0^t G_j(y^{\epsilon}_s)ds &= \int_0^t  c_{j,m_j} H_{m_j}(y^{\epsilon}_s) ds
= c_{j,m_j} I_{m_j}(\hat {f}^{m_j,\epsilon}_t).
\end{align*}
Hence, we only need to show that the collection of stochastic processes of the form $I_{m_k}(\hat {f}^{m_k,\epsilon)}$ converge jointly in finite dimensional distribution.
It is then sufficient to show for every finite collection of times $t_{l} \in [0,T]$,
that the vector
$$\left\{I_{m_k}(\hat {f}^{k,\epsilon}_{t_{l}}), k=m, \dots M \right\},$$
converges jointly, where $m = \min_{k=1, \dots, N} m_k$.
 Let $n_0$ denote the smallest natural number such that $H^*(n_0) \leq \f 1 2$. For $k \leq n_0$, the collection  $I_k(\hat {f}^{k,\epsilon}_{t_{l}})$ converges to a normal distribution and therefore, by the fourth moment theorem \cite[Theorem 1]{Nualart-Peccati},
$$
\Vert \hat {f}^{k,\epsilon}_{t_{l}} \otimes_r \hat  {f}^{k,\epsilon}_{t_{l}} \Vert_{\H^{2k-2r}} \to 0, \qquad  r =1 ,\dots , k-1.
$$

By Cauchy-Schwartz we obtain
 for $r =1, \dots, k_1$, 
$$\begin{aligned}
&\left \Vert \hat {f}^{k_1,\epsilon}_{t_{l_1}} \otimes_r \hat {f}^{k_2,\epsilon}_{t_{l_2}}\right \Vert_{\H^{k_1+k_2-2r}}\\
&\leq \left\Vert  \hat {f}^{k_1,\epsilon}_{t_{l_1}} \otimes_r  \hat {f}^{k_1,\epsilon}_{t_{l_1}} \;
\right \Vert_{\H^{p-r}} \; \left \Vert  \hat {f}^{k_2,\epsilon}_{t_{l_2}} \otimes_r  \hat {f}^{k_2,\epsilon}_{t_{l_2}} \right\Vert_{\H^{q-r}}
 \to 0,
 \end{aligned}$$
for all $ t_{l_1},t_{l_2} \in \R, \,  1 \leq k_1 < n_0 \leq k_2 \leq M$.
We can now apply the asymptotic independent theorem (see Proposition \ref{proposition-spit-independence} in the Appendix), to conclude  the joint convergence
 in finite dimensional distributions of $X^\epsilon$ to $(X^W,X^Z)$. Furthermore $X^W$ is independent of $X^Z$.

The correlations between $X^i_t$ and $X^j_{t'}$,  where $i,j>n$, are $0$ if $m_i \not = m_j$, otherwise given by the $L^2$ norm of their integrands. They follow from the It\^o isometries:
\begin{align*}
& c_{i,m_i} c_{j,m_j} \E \int_{0}^t\! \! \! \int_{0}^{t'}  \! \!  H_{m_i}\left(\int_{\R} (s-u)_{+}^{ \hat H(m_i)-\f 3 2} dW_u\right) 
H_{m_i}\left(\int_{\R} (r-u)_{+}^{ \hat H(m_i)-\f 3 2} dW_u\right) dr ds\\
&= c_{i,m_i} c_{j,m_j} \int_{0}^t \! \! \!\int_{0}^{t'}  \! \! \int_{\R^{m_i}} \prod_{i=1}^{m_i} \left( s - \xi_i\right)_{+}^{\hat H(m) - \f 3 2} \prod_{i=1}^{m_i} \left( r - \xi_i\right)_{+}^{\hat H(m) - \f 3 2} d\xi_1 \dots \xi_{m_i} dr ds.
\end{align*}

The convergence in  H\"older spaces follows from Lemma \ref{Lp-bounds}, which states these processes are tight in $\C^\gamma$, c.f. Proposition \ref{proposition-single-CLT-strong-topologies}, completing the proof.

\end{proof}

\subsection{Convergence in finite dimensional distributions of the rough paths}
We study the canonical lifts  of $X^\epsilon$ to a rough path.
We denote by $\XX^{\epsilon}$ the   canonical/geometric  lift. Its components are $$\XX^{i,j,\epsilon}_{0,t}= \alpha_i(\epsilon) \alpha_j(\epsilon) \int_0^{t} \int_0^{s} G_i(y_s^\epsilon) G_j(y_r^\epsilon) dr ds.$$
From here on we assume  Assumption \ref{assumption-multi-scale}, and Convention \ref{X-component} so the high rank functions are in the first $n$ components, for which $X^{k,\epsilon}$ converges to a Wiener process.
We first work on the  case where one of the components of the iterated integral corresponds to a low Hermite rank.
\begin{lemma}\label{young-lift}
[Young integral case]
Assume Assumption \ref{assumption-multi-scale}. Below $i,j \in \{ 1, \dots N , i \vee j >n  \}$. Then,
\begin{equation}\label{lift-Hermite}
(X^{\epsilon}, \; \XX^{i,j,\epsilon}),
\end{equation}
 converges in  finite dimensional distributions to $(X,\XX^{i,j})$ where $\XX^{i,j}=\int_0^t X^i dX^j$,  and these integrals are well defined as Young integrals.
\end{lemma}
\begin{proof}
By Assumption \ref{assumption-multi-scale}, the functions $G_k$ posses enough integrability such that each component of $X^{\epsilon}$ converges in a H\"older space. Furthermore, by Assumption \ref{assumption-multi-scale} (2) there exist numbers $\eta$ and $\tau$, with $\eta + \tau >1$, such that the H\"older regularity of the  limits corresponding to a Wiener processes, are bounded below  by $\eta$,
 and the ones corresponding to a Hermite process  bounded from below by $\tau$. Therefore,   taking the integrals
$$ 
\alpha_i(\epsilon) \alpha_j(\epsilon)\int_0^{t} \int_0^s G_j(y^{\epsilon}_s) G_i(y^{\epsilon}_r) dr ds = \int_0^t X^{i,\epsilon}_s dX^{j,\epsilon}_s
$$ is a continuous and well-defined operation from $C^\eta\times C^\tau\to C^\tau$ or $C^\tau \times C^\eta \to C^\eta$ , thus weak convergence in $C^\eta$  follows, and in particular convergence in finite dimensional distributions.
\end{proof}
\begin{remark}
Note that thee proof shows convergence in H\"older space $C^\eta$, however $\eta$ here could be a very small number, below $\f 12$, we would work harder in a later section on the tightness in the rough path space topology. For the convergence in rough topology, we want this to work in $\C^{2\alpha}$ for a $\alpha>\f 1 3$. 
We would finally prove  tightness of the iterated integrals in higher H\"older spaces.
\end{remark}
Now it is only left to deal with the parts of the natural rough path lift involving  two Wiener scaling terms, this is carried out in the next section.

\subsubsection{Approximations of iterated integrals: It\^o integral case}
We  proceed to establish convergence of  the iterated integrals where both components correspond to the high Hermit rank case,
appearing in (\ref{lift-Hermite}).
 \begin{remark}
We further assume $H^*(m_k) < 0 $ for each $k$ which gives rise to a Wiener scaling, we do not obtain Logarithmic terms and therefore work with the $\f {1} {\sqrt \epsilon}$ scaling from here on. Furthermore, in this case $\alpha(\epsilon) \int_0^t G(y^{\epsilon}_s) ds$ equals $\sqrt{\epsilon} \int_0^{\f t \epsilon} G(y_s) ds$ in law and for simplicity we will work with the latter in this chapter.
\end{remark}
 In this section, from here onwards we assume  that both  $G_i$ and $G_j$ give rise to a Wiener process in the homogenization process, so $i,j \leq n$ and
 $$X^{k,\epsilon}_t= \sqrt{\epsilon} \int_0^{\f t \epsilon} G(y_s) ds,$$
 with the corresponding iterated integrals.

By Lemma \ref{CLT-Gaussian-multi}, we know that  $( X^{i,\epsilon} ,  X^{j,\epsilon} ) \to (W^i,W^j)$, 
we will see their integral  $\int_0^t X^{i,\epsilon} d X^{j,\epsilon}$, a double integral,  can be discretised and decomposed into integrals over strips
of two significant regions, the integral on the region away from the diagonal is of the form
$$\sum _{k=1} ^{[\f t \epsilon]}  \int _{k-1}^{k} G_i(y_s) ds  \int_0^{k-1}  G_j(y_r)dr,$$
which resembles a Riemann sum for an integral which we might hope to be the stochastic integral $\int_0^t W^j_s dW^i_s$. 
This is not quite true, however its martingale approximation does converge to the stochastic integral. We want to show that  
\begin{align*}
\int_0^{\f t \epsilon}   X^{j,\epsilon}_s dX^{i,\epsilon}_s &= \epsilon \int_{0}^{\f t \epsilon} \int_0^s  G_i(y_s) G_j(y_r) dr ds\\
&= I_1(\epsilon) + I_2(\epsilon),
\end{align*} 
where $I_1(\epsilon) \to \int_0^t W^j_s dW^i_s $, where the integral is understood in the It\^{o}-sense, weakly and $I_2(\epsilon) \to t A^{i,j}$ in probability.
For this we want to use the continuity property of stochastic integrals with respect to martingales and should approximate  $X^{i,\epsilon}$  with a martingale that is predictably uniform-tight, c.f. Lemma \ref{joint-joint-lemma}.  We begin to describe this approximation.

For any $L^2$ functions $U,V$ we introduce the stationary process:
$$\begin{aligned}\Phi_U(t)&= \int_t^\infty U(y_r) dr.
\end{aligned}$$
which unfortunately  does not have good integrability properties.  We would explore  a local independent decomposition of the FOU.
It turns out that for every $t$ there exists a decomposition,
$$y_t = \overline{y}^k_t + \tilde{y}^k_t,$$
where the first term $\overline{y}^k_t$ is $\mathcal{F}_k$ measurable and $\tilde{y}^k_t$ is independent of  $\mathcal{F}_k$,
where $\F_k$ is  the filtration generated by the driving fractional Brownian motion. We  will show later, in Proposition \ref{integrable-lemma}, 
for $H > \f 1 2$ and $H^*(m)<0$,
\begin{equation}
\label{integrable}
\sup_k \sup_{q\geq m} \int_{k-1}^{\infty} \int_{k-1}^{\infty}   \E \left(    {\overline{y}^k_s} {\overline{y}^k_t}  \right)^q \, dt \,ds < \infty,
\end{equation}
We therefore define 
\begin{equation} \begin{aligned}\hat U(k):= \int_{k-1}^\infty  \E( U(y_r)\,|\, \F_k) \, dr, \\
\hat V(k):= \int_{k-1}^\infty  \E( V(y_r))\,|\,\F_k)  \,dr.\end{aligned}
 \end{equation}
 Note both $\hat U$ and $\hat V$ are  shift covariant, i.e.$(\hat U \circ \tau)(k) =\hat U({k+1})$ where $\tau$ is the shift operator. 
 To proceed further we need a couple of lemmas.
 
 \begin{lemma}
For $x,y,a,b \in \R$ such that $a^2+b^2=1$,
\begin{equation}
H_m(ax+by) = \sum_{j=0}^m  \binom{m}{j} a^{j} b^{m-j} H_{   j}(x) H_{m-   j}(y).
\end{equation}
\end{lemma}

\begin{lemma}
Let $H>\f 12$. Set $a_t= \Vert \overline{y}^k_t \Vert_{L^2} $. Then
$$\E[ H_m(y_t)|\mathcal{F}_k] = (a_t)^m  H_m\left( \f  {\overline{y}^k_t}  {a_t} \right).$$
\end{lemma}
\begin{proof}
Set  $b_t=\Vert  \tilde{y}^k_t \Vert_{L^2} $. By the independence of $\overline{y}^k_t $ and $ \tilde{y}^k_t$ we obtain 
\begin{align*}
1 &=\Vert y_k \Vert_{L^2}^2 = \Vert \overline{y}^k_t \Vert_{L^2}^2 + \Vert  \tilde{y}^k_t \Vert_{L^2}^2
= (a_t)^2 + (b_t)^2.
\end{align*}
 Now we decompose $H_m(y_t)$ using the above identity and obtain,
\begin{align*}
H_m(y_t)&=H_m\left( \overline{y}^k_t + \tilde{y}^k_t \right)
=H_m\left( a_t \left(  \f  {\overline{y}^k_t}  {a_t} \right) + b_t \left( \f {\tilde{y}^k_t}  {b_t} \right) \right)\\
&= \sum_{j=0}^m  \binom{m}{j} a_t^j  b_t^{m-j}  H_j\left( \f  {\overline{y}^k_t}  {a_t} \right)    H_{m-j}\left( \f {\tilde{y}^k_t}  {b_t} \right).
\end{align*}
Note that by construction $\f  {\overline{y}^k_t}  {a_t}$ and $ \f {\tilde{y}^k_t}  {b_t}$ are standard Gaussian random variables.  Therefore, by the independence $ \tilde y^k_t$ of $\F_t$, 
\begin{align*}
\E[ H_m(y_t)|\mathcal{F}_k] &= \sum_{j=0}^m \binom{m}{j} (a_t)^j  (b_t)^{m-j} 
H_j\left( \f  {\overline{y}^k_t}  {a_t} \right)  \E\left [H_{m-j} \left( \f {\tilde{y}^k_t}  {b_t} \right)| \mathcal{F}_k\right]\\
&= (a_t)^m  H_m\left( \f  {\overline{y}^k_t}  {a_t} \right).
\end{align*}
We have used  the fact that  $\E H_j\left( \f {\tilde{y}^k_t}  {b_t} \right)$ vanishes for any $j\ge 1$ and $H_0=1$.
\end{proof}

\begin{proposition}\label{prop-integrability}[See \S\ref{prove-lemma-int}]
	Let $H>\f 12$ and  $U\in L^2(\R,\mu)$ with Hermite rank  $m >\f 1 {1-H}$.
	Then the process $\hat U(j)$ is bounded in  $L^2$ (provided (\ref{integrable}) holds.)
\end{proposition}

\begin{proof}
The proof is less straightforward due to the lack of the strong mixing property.  
Here we rely on (\ref{integrable}), whose proof is lengthy and independent of 
the error estimates here and is therefore postponed  to  \S\ref{prove-lemma-int}.

We compute the $L^2$ norm, using  the definition of $\hat U$ and the Hermite expansion  $U=\sum_{q=m}^{\infty} c_q H_q$,
\begin{align*}  
\Vert \hat U(k)\Vert_{L^2}
&=\int_{k-1}^{\infty} \int_{k-1}^{\infty} \E\left( \E[     U(y_s)| \mathcal{F}_k] \E[   U(y_r)| \mathcal{F}_k]\right) \, dr ds \\
&= \int_{k-1}^{\infty} \int_{k-1}^{\infty} \sum_{q=m}^{\infty} \sum_{j=m}^{\infty} c_q c_j \E\Big( \E[ H_q(y_s)| \mathcal{F}_k] \,\E[H_j(y_r)| \mathcal{F}_k]  \Big) dr \, ds\\
&= \int_{k-1}^{\infty} \int_{k-1}^{\infty} \sum_{q=m}^{\infty} (c_q)^2   \E \left( (a_s)^q (a_r)^q   H_q\left( \f  {\overline{y}^k_s}  {a_s} \right) H_q\left( \f  {\overline{y}^k_r}  {a_r} \right) \right)  dr \,ds\\
&= \int_{k-1}^{\infty} \int_{k-1}^{\infty} \sum_{q=m}^{\infty} (c_q)^2  \,q! \,(a_s)^q  (a_r)^q \E \left(  \f  {\overline{y}^k_s}  {a_s}  \f  {\overline{y}^k_r}  {a_r}  \right)^q   dr\, ds\\
&=\int_{k-1}^{\infty} \int_{k-1}^{\infty} \sum_{q=m}^{\infty} (c_q)^2 \, q! \, \E \left(    {\overline{y}^k_s} {\overline{y}^k_r}  \right)^q  dr \, ds.
\end{align*}
By summability of  $(c_q)^2 q!$, following from $U\in L^2$. 
The proof follows from the assumption that 
$\sup_{q\ge m,k}\int_{k-1}^{\infty} \int_{k-1}^{\infty}   \E \left(    {\overline{y}^k_s} {\overline{y}^k_r}  \right)^q  dr ds $ is finite. This concludes the lemma. \end{proof}

As a  corollary of Proposition \ref{prop-integrability} we have
\begin{corollary}\label{cor-integrability}  
The process  $(M_k, k\ge 1)$, where
	$$M_k =\sum_{j=1}^{k}\left( \hat U(j)   -\E\left  (\hat U(j) |\F_{j-1} \right) \right),$$
	is an $\F_k$-adapted  $L^2$  martingale with shift covariant martingale difference. 
Similarly, 
$$ N_k = \sum_{j=1}^k \left( \hat V(j) - \E ( \hat V(j) | \mathcal{F}_{j-1}) \right),$$
defined also an $\F_k$-adapted  $L^2$  martingale.
\end{corollary}
\begin{proposition}\label{prop-area}
There exists a function $\err(\epsilon)$ converging to zero in probability as $\epsilon\to 0$  such that
\begin{equation}\label{area1-2} \begin{aligned}
&\epsilon \int_{0}^{\f t \epsilon} \int_0^s  U(y_s) V(y_r) dr ds 
&=\epsilon \sum _{k=1} ^{[\f t \epsilon]} (M_{k+1}-M_k)N_k +  (s \wedge t) \gamma+\err_1(\epsilon)\end{aligned}
\end{equation}
where $$\gamma
= \int_0^\infty \E (U(y_s) V(y_0) )ds.$$
\end{proposition}
The proof for this is given in the rest of the section.  Note that the It\^o integral approximations work well while the processes involved have independent increments or satisfy strong mixing properties. To tackle the lack of these properties, we use a locally independent decompositions of the fOU. We also use Birkhoff's ergodic theorem.
After proving the Proposition,  in the next section we show that $\epsilon \sum _{k=1} ^{[\f t \epsilon]} (M_{k+1}-M_k)N_k$ converges to  the relevant It\^o integrals of the limits of
$\sqrt \epsilon \int_0^{[\f t \epsilon]} U(y_r) dr$ and $\sqrt \epsilon \int_0^{[\f t \epsilon]} V(y_r) dr$.

\begin{lemma}
The stationary Ornstein-Uhlenbeck process is ergodic.
\end{lemma}
A stationary Gaussian process is ergodic if its spectral measure has no atom, 
 \cite{Cornfeld-Fomin-Sinai, Samorodnitsky}. The spectral measure $F$  of a stationary Gaussian process is obtained from  
 Fourier transforming  its correlation function and 
$\rho(\lambda)=\int_\R e^{i \lambda x} dF(x)$.
According to   \cite{Cheridito-Kawaguchi-Maejima}:
 \begin{equation}\label{cor5-2}
\rho(s) =  \frac{2 \Gamma(2H+1) \sin(\pi H)}{2 \pi} \int_{\R} e^{isx} \frac{\vert x \vert^{1-2H}}{1 + x^2} dx,
\end{equation}
so the spectral measure is absolutely continuous with respect to the Lebesgue measure with spectral density $s(x) = c \f { \vert x \vert^{1-2H}}{1+ x^2}$. 

For $k=1,2,\dots$, we define the $\F_k$-adapted processes:
\begin{align*}
	I(k)&= \int_{k-1}^{k} U(y_s) ds= \Phi_U(k) -\Phi_U(k-1)  \\
	J(k)&= \int_{k-1}^{k}  V(y_s) ds=\Phi_V(k) -\Phi_V(k-1).
\end{align*}

	\begin{remark}\label{remark-mart-dif}
We note the following useful identities. For $k=1,2, \dots$, 
\begin{equs}\label{mart-dif2}
&\hat U(k) = I(k)+ \E [\hat U(k+1) \, |\, \F_k],  \\
&M_{k+1}-M_k=I(k)+\hat U(k+1)-\hat U(k),\label{mart-dif} \\
& \int_0^k U(y_r) dr =M_k-\hat U(k)+\hat U(1)-M_1.
\end{equs}
\end{remark}

\begin{proposition}\label{lemma-6.13}
Suppose that $U$ and $V$ satisfy the assumptions imposed above, then
the triple below converges in finite dimensional distributions.
$$\lim_{\epsilon\to 0} \left(  {\sqrt{\epsilon}}  M_{[\f t \epsilon]}, \; {\sqrt{\epsilon}}  N_{[\f t \epsilon]}, \;
\epsilon \sum _{k=1} ^{[\f t \epsilon]} (M_{k+1}-M_k)N_k\right)
= \left( W_t^1, \;W_t^2, \; \int_0^t W_s^1 \,d W_s^2\right).$$
Here  $ W^1,  W^2$ are standard Wiener processes with covariance $\int_0^\infty \E(U(y_r)V(y_0)) dr$, and variances
respectively $\int_0^\infty \E(U(y_r)U(y_0)) dr$ and $\int_0^\infty \E(V(y_r)V(y_0)) dr$.
The integration  is in It\^ o sense.

\end{proposition}

\begin{proof}
Define $$M^{\epsilon}_{t} = {\sqrt{\epsilon}}  M_{[\f t \epsilon]}, \qquad N^{\epsilon}_{t} 
= {\sqrt{\epsilon}}  N_{[\f t \epsilon]},$$
Using the identity (\ref{mart-dif}) we show that 
$$\begin{aligned}
M^{\epsilon}_{t}  &=\sqrt \epsilon \sum_{k=1}^{[\f t \epsilon]} (M_{k+1} -M_k) +\sqrt \epsilon \,M_1\\
&=\sqrt \epsilon \int_0^{[\f t \epsilon]} U(y_r) ds
+\sqrt \epsilon \hat U([\f t \epsilon])-\sqrt \epsilon \hat U(1) +\sqrt \epsilon M_1.
\end{aligned}$$
Since $\hat U$ is $L^2$ bounded, the joint convergence of $M^{\epsilon}_{t} $ and $N_t^\epsilon$, in finite dimensional 
distributions follows from Lemma \ref{CLT-Gaussian-multi}. Next observe that,
 $$ \epsilon \sum _{k=1} ^{[\f t \epsilon]} (M_{k+1}-M_k)N_k
 = \int_0^{t} M^{\epsilon}_{s} dN^{\epsilon}_s.$$
The joint convergence follows since
 $\E \left( M^{\epsilon}_t \right)^2 \lesssim t + o(\epsilon)$, see Lemma \ref{joint-joint-lemma} and Lemma \ref{P-UT-condition}, we can use the continuity theorems on integrals with respect to martingales with jumps.
\end{proof}

Henceforth, in this section we set $L=L(\epsilon)=[\f t \epsilon]$.

\begin{lemma}\label{lemma2-area}
There exists a function $\err_1(\epsilon)$, which converges to zero in probability as $\epsilon\to 0$,   such that
\begin{equation}\label{area1} \begin{aligned}
&\epsilon \int_{0}^{\f t \epsilon} \int_0^s  U(y_s) V(y_r) dr ds\\
&=\epsilon\sum _{k=1} ^{L} I(k) \sum_{l=1}^{k-1} J(l)+t \int _0^1 \int_0^s \E \left( U(y_s)  V(y_r)\right) \,dr  ds
+\err_1(\epsilon)\end{aligned}
\end{equation}
\end{lemma}

\begin{proof} 
Let us divide the integration region  $0\le r\le s \le L$ 
main region and the other negligible regions.
\begin{align*}
	\int_{0}^{L} \int_0^s  U(y_s) V(y_r) dr ds
+\epsilon \int_{L} ^{\f t \epsilon }  \int_0^s  U(y_s) V(y_r) dr ds.\end{align*}
The second term, integration in the small region,  is of order $o(\epsilon)$, since
 $\| \int_{L} ^{\f t \epsilon }  U(y_s) ds\|_{L^2}$ is bounded by stationarity   of $y_r$ and 
   $\|\sqrt \epsilon \int_0^{\f t \epsilon}  V(y_r) dr\|_{L^2}$ is bounded by  Lemma \ref{Lp-bounds}.
We compute the integration in the main region:
\begin{align*}
	&\int_{0}^{L} \int_0^s  U(y_s) V(y_r) dr ds\\
	 =&\sum _{k=1} ^{L} \int _{k-1}^{k} U(y_s) 
	\left(   \int_0^{k-1}  V(y_r)dr+ \int_{k-1}^s  V(y_r)dr  \right)   ds\\
	=&\sum _{k=1} ^L\int _{k-1}^{k} U(y_s) ds  \int_0^{k-1}  V(y_r)dr + \sum _{k=1} ^{L} \int _{ \{ k-1\le r\le s \le k\}} U(y_s)  V(y_r)dr  ds
	\\=&\sum _{k=1} ^{L} I(k) \sum_{l=1}^{k-1} J(l)+ \sum _{k=1} ^{L} \int _{ \{ k-1\le r\le s \le k\}} U(y_s)  V(y_r)dr  ds.
\end{align*}
The stochastic processes  $Z_k = \int _{ \{ k-1\le r\le s \le k\}} U(y_s)  V(y_r)dr  ds$
 are shift invariant and the shift operator is ergodic with respect to the probability distribution on the path space, generated by the fOU process,
 hence, by Birkhoff's ergodic theorem, 
$$ \f 1 {L} \sum_{k=1}^ {L} Z_k \stackrel{(\epsilon\to 0)} {\longrightarrow}  \E Z_1= \int _0^1 \int_0^s \E \left( U(y_s)  V(y_r)\right) dr  ds .$$
This complete the proof.
\end{proof}

\begin{lemma}\label{lemma-6.15}
The following converges in probability:
$$\lim_{\epsilon\to 0} \epsilon  \sum_{k=1} ^{L} \left(I(k) \sum_{l=1}^{k-1} J(l)-  (M_{k+1}-M_k)N_k\right) 
= \int_1^{\infty} \int_0^1 \E\left( U(y_s) V(y_r)\right) dr ds.$$
\end{lemma}
\begin{proof}

{\bf A.} 
Summing from $1$ to $k$ of the identity  (\ref{mart-dif}), we see that 
 $$  \sum_{l=1}^{k-1} J(l) =N_{k} -\hat V(k)+\hat V(1)-N_1,$$
With this and relation (\ref{mart-dif}) for the martingale difference $M_{k+1}-M_k$,  

we obtain:
 \begin{align*}
 & \sum_{k=1} ^{L} I(k) \sum_{l=1}^{k-1} J(l) -  \left(  M_{k+1} - M_k \right) N_k  \\
 &= \sum_{k=1} ^{L} I(k)\,\left ( N_k -\hat V(k)+\hat V(1)-N_1 \right)-\left(I(k)+  \hat U ({k+1})-\hat U(k) \right) N_k\\
 &=  \sum_{k=1} ^{L} -I(k)\, \hat V(k) + \sum_{k=1} ^{L} I(k)(\hat V(1)-N_1)- \sum_{k=1} ^{L} (\hat U (k+1)-\hat U(k) ) N_k.
	  \end{align*}
Firstly, by the  shift invariance of the  summands, below, and Birkhoff's ergodic theorem, we obtain
\begin{equation}\label{2nd-limit}
\begin{aligned}
-\, \epsilon \sum_{k=1}^{L } I(k) \hat V(k)&\longrightarrow (-  t)\, \E [I(1) \hat V(1)]=(-t)  \E \left( \int_0^1 U(y_r) dr\int_0^{\infty} V(y_s) ds \right).
 \end{aligned}
\end{equation}
Next, since $\hat V(1)-N_1=\E [\hat V(1) \, |\, \F_0]$, 
  $$\begin{aligned}
 \E \left| \epsilon  \sum_{k=1}^{L } I(k)(\hat V(1)-N_1) \right|^2
=&\E \left| \epsilon \int_0^{L } U(y_r) \;dr   \; \E [\hat V(1) \, |\, \F_0]\right|^2\\
\lesssim &\epsilon^2\, \E[\hat V(1)]^2  \int_0^{L }  \int_0^{L }  \E[ U(y_r)U(y_s)]\, ds \,dr,\end{aligned}$$
which by Lemma \ref {Lp-bounds}  is of order $\epsilon$.

{\bf B.} It remains to discuss the convergence of 
$$\epsilon \sum_{k=1}^{L }  (\hat U (k+1)-\hat U(k) ) N_k.$$
We do not have shift invariant and therefore
break it down into increments.  
We change the order of summation to obtain the following decomposition
$$ \begin{aligned}
& \sum_{k=1}^L  (\hat U _{k+1}-\hat U(k) ) N_k
\\=& \sum_{k=1}^L(\hat U(k+1) -\hat U(k)) \left[  \sum_{j=1}^{k-1}  (N_{j+1}-N_j) +N_1\right]\\
=&  \sum_ {j=1}^{L-1} (N_{j+1}-N_j)  \sum_{k=j+1}^L (\hat U(k+1) - \hat U(k)) 
 + \sum_{k=1}^L(\hat U(k+1) -\hat U(k))  N_1\\
 =&\ \sum_ {j=1}^{L-1} (N_{j+1}-N_j)   \hat U(L+1)  - \sum_{j=1}^{L-1} (N_{j+1}-N_j) \,\hat U(j+1)
 +\left (\hat U(L+1) - \hat U(1)\right)  N_1.
\end{aligned}$$
We may now apply Birkhoff's ergodic theorem to  the first term, taking $\epsilon\to 0$,
$$\lim_{\epsilon\to 0} - \epsilon\ \sum_ {j=1}^{L-1} (N_{j+1}-N_j)   \hat U(L +1) =0,\quad $$ 
in probability.
By the same ergodic theorem, the second term
$$ \epsilon\left (N_L -N_1\right)  \hat U(L +1) \longrightarrow  t\, \E\left( \hat U(2) (N_2-N_1) \right), $$ 
in probability.
By Proposition \ref {prop-integrability}, $\hat U(j)$ is bounded in $L^2$, hence we obtain for the third term,
$$\epsilon \left|\left (\hat U(L +1) - \hat U(1)\right)  N_1\right|_{L^2} \lesssim \epsilon. $$ 
Overall we end up with \begin{equation}
\epsilon \sum_{k=1}^L  (\hat U _{k+1}-\hat U(k) ) N_k \longrightarrow -\, t\; \E\left( \hat U(2) (N_2-N_1) \right),
\end{equation}
and so  \begin{equation}\label{diff}
\begin{split}
&\lim_{\epsilon\to 0} 
 \epsilon \sum_{k=1}^{L}\left(  \sum_{l=0}^{k} I(k) J(l)- \left(  M_{k+1} - M_k \right) N_k\right)\\
&=  t\left[  \E\left( \hat U(2) (N_2-N_1) \right) -  I(1) \hat V(1) \right].
\end{split}
\end{equation}
  
{\bf C.} We look a better expression for this limit, 
$$\begin{aligned}
\E \hat U(2) (N_2-N_1)&=\E\left(  \int_1^\infty U(y_r) dr \left(\hat V(2) -\E (\hat V(2)|\F_1)\right)\right)\\
&= \E\left(  \int_0^\infty U(y_r) dr \left(\hat V(2) -\E (\hat V(2)|\F_1)\right)\right).
\end{aligned}$$
Since
$$\hat V(2) -\E (\hat V(2)|\F_1)-\int_0^1 V(y_s) ds=\hat V(2)-\hat V(1),$$
hence we work with 
$$\E\left(  \int_0^\infty U(y_r) dr \left(\hat V(2) -\hat V(1)\right)\right).   $$
we compute $\sum_{k=2}^L \E[I(k)J(1)]$. We first use 
Remark \ref {remark-mart-dif} to write individual terms by the martingale differences. Specifically we use
$${\begin{split} J(1)&= (\hat V(1) -\hat V(2) )+(N_2-N_1), \\  I(k)&= M_{k+1}-M_k-( \hat U(k+1)-\hat U(k))\end{split}}$$
for obtaining
$$\begin{aligned}
& \sum_{k=2}^L  I(k)J(1)- \sum_{k=2}^L  \left(M_{k+1}-M_k \right) (N_2-N_1) \\
= & \sum_{k=2}^L  I(k)(\hat V(1) -\hat V(2)+\sum_{k=2}^L  I(k) (N_2-N_1) - \sum_{k=2}^L  \left(M_{k+1}-M_k \right) (N_2-N_1)\\
= & \sum_{k=2}^L  I(k)(\hat V(1) -\hat V(2))
-\sum_{k=2}^L  \left(-\hat U(k+1)+\hat U(k)\right) (N_2-N_1)\\
=&\sum_{k=2}^L I(L)[\hat V(L+1-k)- \hat V(L+2-k)]
- \sum_{k=2}^L ( \hat U(k+1)-\hat U(k)) (N_2-N_1) \\
=&I(L) \hat V(1)-I(L) \hat V(L) + [\hat U(2) (N_2-N_1)]-
 [\hat U(L+1) (N_2-N_1)].
\end{aligned}$$
In the second step, we  used the stationary property by which we also have
$$\E [I(L) (\hat V(1)-\hat V(L)) ]=\E [I(1)\hat V(L)]-\E[ I(1)\hat V(1)].$$
Since
$$\begin{aligned}
&\E \Big( I(L) \hat V(1) -\hat U(L+1) (N_2-N_1)]\Big)^2 \to 0 
\end{aligned}$$
by Corrolary \ref{cor-integrability}. This concludes that
$$\sum_{k=2}^{L} \E (I(k) J(1)) =-\E[I(1)\hat V(1)] +\E  [\hat U(2) (N_2-N_1)]+o(\epsilon).$$
On the other hand,   $$\sum_{k=2}^{L} \E (I(k) J(1)) = \int_1^{L} \int_0^1 \E\left( U(y_s) V(y_r)\right) dr ds.$$ 
Since  $L= \lfloor \f t \epsilon \rfloor \to \infty$ as $\epsilon \to 0$, this complete the proof for Lemma \ref{lemma-6.15}.
\end{proof}
 
Now we return to Proposition \ref{prop-area}.  Taking $\epsilon \to 0$ in Equation (\ref{area1}) we obtain 
\begin{align*}
&\lim_{\epsilon \to 0} \epsilon \int_0^{\f t \epsilon} \int_0^s U(y_s) V(y_r) dr ds \\ &=
\lim_{\epsilon \to 0} \epsilon \sum_{k=0}^{\lfloor \f t \epsilon \rfloor} (M_{k+1} - M_k) N_k 
+t \int_0^1 \int_0^s \E \left(U(y_s) V(y_r) \right) dr ds  + t \int_{1}^{\infty} \int_0^1  \E \left(U(y_s) V(y_r) \right) dr ds\\
&= \int_0^t W^1_s dW^2_s + t \int_0^\infty  \E\left( U(y_0) V(y_u)\right)  du,
\end{align*}
which completes the proof for Proposition \ref{prop-area}.

\subsubsection{Enhanced functional limit theorem (in f.d.d)}

To put everything together we first need to state a lemma.

\begin{lemma}\label{joint-joint-lemma}
	Suppose the stochastic processes $(X^n, Y^n, Z^n, R_n^1, R_n^2)$  satisfy the conditions (1)-(3) below. Then if the trio $(X^n, Y^n, Z^n) \to (X,Y,Z)$ in law as $n\to \infty$, in the c\`adl\`ag topology,  so does the quadruple 
	$$\Big(X^n,Y^n+R_n^1, \, \int Y^n_s \,dX^n_s,\, Z^n+R_n^2\Big) \to \Big (X,\, Y, \, \int Y_s dX_s, \, Z\Big).$$
	The integrals are in It\^o sense.
	\begin{itemize} 
		\item [(1)] Each $(X^n_s, s \in [0,t] )$  is a  c\`adl\`ag $d$-dimensional semi-martingales, on a filtered probability space 
		$(\Omega_n, \F_n, P_n)$,   satisfying the following predictable uniform tightness (P-UT)  condition: for every $t>0$, 
		$$\lim_{C\uparrow \infty} \sup_{H^n} P_n\left (\int_0^t H^n_s \, dX_s^n>C \right )=0,$$
		where  the supremum is taken over all elementary processes $H^n$ uniformly bounded by $1$ and adapted to $\F_n$.
		\item[(2)]  $Y^n$ is a family of $d'\times d$- dimensional c\`adl\`ag stochastic processes,  $Z^n$ 
		another family of multi-dimensional c\`adl\`ag stochastic processes, both adapted to $\F^n$.
		\item [(3)] $R_n^1, R_n^2$ are  adapted stochastic processes converging to zero in 
		probability.
	\end{itemize} 
\end{lemma}
The Lemma is essentially  \cite[Theorem 6.22]{Jacod-Shiryaev} , it is only left to include the terms $R^1_n$ and $R^2_n$.
On a separable probability space,  convergence in finite dimensional distributions is equivalent to convergence in the Prohorov  metric. The Prohorov metric is
defined by $d(\mu_1, \mu_2)= \inf_{\delta>0} \{ \mu_1(A)\le \mu_2(A_\delta)+\delta, \mu_2(A)\le \mu_1(A_\delta)+\delta\}$ where $A$ is any Borel measurable set 
and $A_\delta$ its $\delta$-enlargement set.  Thus $X^n$ converges in distribution and $R_n\to 0$ implies that $X_n+R_n\to 0$ in distribution. 
We can then apply Theorem 6.22 from \cite{Jacod-Shiryaev}.
Without the $R_n$ terms, the convergence of the quadruple is trivial for $H_n$ uniformly bounded elementary processes. 
The rest follows from a loalization procedure, a density argument applied to integrands, the predictable uniform tightness condition on the integrands, $X_n$, allows reversing the order of taking limits.
For $Z^n_t =R^1_n=R^2_n=0$ one can also apply \cite[Theorem 2.7]{Kurtz-Protter}.
\begin{lemma}\label{P-UT-condition}
If a sequence of martingales $N^n_s$ is uniformly bounded in $L^2$, then they satisfy the P-UT condition. 
\end{lemma}
	\begin{proof}
	
	Given an elementary process
	$H^n=a_0 \1_{\{0\}} +\sum_{i=1}^k a_i \1_{(s_i, s_{i+1}]}(t)$
	where $a_i\in \F_{s_i}$ is assumed to be bounded by $1$ and $s_i\le [tn]$, we see that
	$$\begin{aligned}
	&P_n\left (\int_0^t H^n_s dN_s^n> C \right )
	\le \f 1 {C^2} \E \left|\int_0^t H^n_s dN_s^n\right|^2\ \\
	&\le  \f 1 {C^2}  \E \left| \sum_{i=1}^k a_i \left (N_{  s_{i+1}   \wedge (nt)} -N_{ s_{i}  \wedge (nt)} \right)\right|^2\\
	&\le  \f 1 {C^2}  \sum_{i=1}^k \E \left( a_i^2 \left (N_{ s_{i+1}   \wedge (nt)} -N_{  s_{i}  \wedge (nt)} \right)\right)^2
	\le  \f 1 {C^2}  \sum_{i=1}^k \E  \left (N_{  s_{i+1}   \wedge (nt)} -N_{  s_{i}  \wedge (nt)} \right)^2
	\\&\le  \f 1 {C^2} \E \left (N_{ s_{k+1}   \wedge (nt)}\right)^2  -  \f 1 {C^2}\E \left (N_{  s_{1}   \wedge (nt)}\right)^2\;  \stackrel{(C\to \infty)} {\to}\; 0.
	\end{aligned}$$
\end{proof}

Now, recall,
\begin{align*}
X^\epsilon&:=\left(\alpha_1(\epsilon) \int_0^{t} G_1(y^{\epsilon}_s)ds, \,\dots, \alpha_N(\epsilon) \int_0^{t} G_N(y^{\epsilon}_s)ds\right)\\
A^{i,j} &=  \left\{	\begin{array}{cl} 
\int_0^{\infty} \E\left( G_i(y_s) G_j(y_0) \right) ds,  &\hbox{ if }i,j \leq n ,\\
0, &\hbox{ otherwise.}
\end{array} \right.	\\
 X&= \lim_{\epsilon \to 0} X^{\epsilon}.
\end{align*}

\begin{proposition}\label{lift-CLT}
Assume that $G_k \in L^{p_k}(\mu)$  satisfy Assumption \ref{assumption-multi-scale}.
Set  $$ {\begin{split} \XX^{i,j,\epsilon}&= \alpha_i(\epsilon) \alpha_j(\epsilon) \int_0^{t} \int_0^s G_i(y^{\epsilon}_s) G_j(y^{\epsilon}_r) dr ds, \\
\XX^{i,j} & =  \int_0^t X^i dX^j, \end{split}}$$
where the second  integral is to be understood in the It\^o-sense if two Wiener processes appear and in the Young sense otherwise.
 Then,  as $\epsilon\to 0$,
$$\XX^{\epsilon} \to \XX +  A (t-s),$$
 in the sense of finite dimensional distributions .
Furthermore, $$ \X^{\epsilon} = \left(X^\epsilon, \XX^{\epsilon}  \right) \to \X = \left(X,\XX + A(t-s) \right),$$ 
in finite dimensional distributions.
\end{proposition}
 
\begin{proof}

To apply Lemma \ref{joint-joint-lemma} we first define the multi-dimensional martingales to deal with the part converging to a Wiener process. For $i \leq n$ set $$M^{i}_L =  \sum_{k=1}^{L} \hat G_i(k) - \E[ \hat G_i(k)  | \mathcal{F}_{k-1}].$$
Now, by Lemma \ref{lemma2-area} and Lemma \ref{lemma-6.15}, for $i,j \leq n$,
$$\XX^{i,j,\epsilon}= \epsilon \sum_{k=1}^{[\f t \epsilon]} (M^j_{k+1} -M^j_k) M^i_k+ t \int_{0}^{\infty} \E \left( G_i(y_s) G_j(y_0)  \right) ds + E_{\epsilon},$$
 where  $E_{\epsilon} \to 0$  in probability. Hence,
 it is enough to establish convergence of $ \epsilon \sum_{k=1}^{L} (M^j_{k+1} -M^j_k) M^i_k$ in finite dimensional distributions.
To do so we  define the piecewise constant c\`adl\`ag
 $L^2$-martingales $(M^{j,\epsilon}_{t}, t \ge 0)$  via
\begin{align*}
	M^{j,\epsilon}_{t} = {\sqrt{\epsilon}}  M^j_{[\f t \epsilon]}
\end{align*} 
to obtain 
$$\epsilon \sum_{k=1}^{[\f t \epsilon]} (M^j_{k+1} -M^j_k) M^i_k = \int_0^{[ \f t \epsilon]} M^{i,\epsilon}_{s} dM^{j,\epsilon}_s.$$
According to Remark \ref{remark-mart-dif},
$$ M^{j,\epsilon}_{\lf \f t \epsilon \rf} = \sqrt{\epsilon} \int_0^{\f t \epsilon} G^j(y_s) ds + o(\sqrt{\epsilon}). $$
Thus, by Lemma \ref{joint-clt}, $(M^{j,\epsilon}_{t})_{j=1, \dots, n}$ converge jointly in some H\"older space, which implies of course convergence in the Skorokhod topology. Therefore it is only left to establish the uniform $L^2$ bounds, which follows from Corollary \ref{cor-integrability} or as below:
\begin{align*}
\Vert M^{j,\epsilon}_{t} \Vert_{L^2} &= {\sqrt{\epsilon}} \Vert M_{[\f t \epsilon]} \Vert_{L^2}
= {\sqrt{\epsilon}} \left\Vert\int_0^{\f t \epsilon} G_j(y_s) ds + o({\sqrt{\epsilon}})\right\Vert_{L^2}
\lesssim 1.
\end{align*}
Now by Lemma \ref{P-UT-condition} our martingales satisfy the predictable uniform tightness condition.
We apply  Lemma \ref{joint-clt-3}, Lemma \ref{joint-joint-lemma} with $X^{\epsilon}=Y^{\epsilon}=M^{\epsilon}=(M^{1,\epsilon}, \dots M^{n,\epsilon})$, $Z^{\epsilon}=(X^k,\XX^{i,j})$ for $k >n$ and $i,j$ such that $i \vee j >n$ and $R^1_n= \err_1(\epsilon)$ as in Proposition \ref{prop-area} to conclude the claim.
\end{proof}

\subsection{Tightness of iterated integrals}
\label{tightness}
In this section we establish moment bounds on the iterated integrals to show that the above proven convergence takes place in a suitable space of rough paths.

Now, let $G_i, G_j$ be two functions in $L^2$ with Hermite ranks $m_{G_i}$ and $m_{G_j}$ respectively. Set $\alpha_i=\alpha(\epsilon,H^*(m_{G_i}))$ and $\alpha_j(\epsilon)=\alpha(\epsilon,H^*(m_{G_j}))$.
To obtain tightness for 
$$
\XX^{i,j,\epsilon}(t)=\alpha_i(\epsilon) \alpha_j(\epsilon)\int_0^{t} \int_0^s G_i(y^{\epsilon}_s) G_j(y^{\epsilon}_r) dr ds
$$
we assume a certain decay in the Hermite expansion of our functions.
If $G_i$ and $G_j$ are in  finite chaos,  $\E \left( \XX^{i,j,\epsilon} \right)^p $ is  the sum of a finite number of
terms which are controlled by integrals of the form, 
$$  \alpha_i(\epsilon)^p \alpha_j(\epsilon)^p \int_0^{t} \dots  \int_0^{t}   \E \left( \prod_{k=1}^{2p} H_{m_k} (y^{\epsilon}_{s_k} ) \right) ds_1\dots ds_{2p},$$
where $H_{m_k}$ are Hermite polynomials with
$m_k \geq \min(m_{G_i},m_{G_j})$and
$s_1\le s_2\le \dots \le s_{2p}$.
Let us try to compute $\E \left( \prod_{k=1}^{2p} H_{m_k}(y^{\epsilon}_{s_k})\right)$.
If $X_i$ are random variables, there is the product formula
$$\E\left [\prod_{k=1}^{2p} X_k \right] = \sum_{\pi \in P(1, \dots, 2p)} \prod_{B \in \pi} \; \E ^c[X_k: k\in B],$$
where $\pi$ denotes a partition of $\{1, 2, \dots, {2p}\}$ and $\E^c[X_k: k \in B]$ denotes the joint  cumulant of the variable $X_k$ with $k$ in a block $B$.
These joint cumulants can also be given by linear combination of the form $\prod_{B\in \pi}\E[\prod_{k \in B} X_k]$. Gaussian processes
have vanishing third or higher order cumulants, and since  expectations for products of Hermite polynomials of $X_k$
are related to these cumulants, there are simpler ways for computing them. Using convolution with heat kernels fewer terms from the graph remain.
The expectations of the product are then given by product of expectations of all pairings, and summed over all such graphs.
For a particular graph, we denote by $n(l, k)$ the number of edges connecting $l$ to $k$, so it takes values in $\{0,1,\dots \min (m_l,m_k)\}$,
and consider the pairings in an ordered way so that each pairing is counted only once.  We have $\sum_{k=1}^{2p} n(l,k)= {m_l}$. Since edges are only allowed to connect with different nodes we have $n(k,k)=0$ for every $k$.
We then observe that  $\E\left( \prod_{k=1}^{2p} H_{m_k} (y^{\epsilon}_{s_k} ) \right)$ is  the sum  over the finite number of graphs of pairings, for any
given graph this is
$$\prod_{k=1}^{2p}  \prod_{l=k+1}^{2p} \E(y^{\epsilon}_{s_k}y^{\epsilon}_{s_{l}}) =\prod_{k=1}^{2p} \prod _{\{l: l>k, \, l\in \Gamma_k\}}  \rho(s^{\epsilon}_l-s^{\epsilon}_k)^{n(l, k)},$$
where $\Gamma_k$ denotes the subgraph of nodes connected to $k$ from nodes in  the forward direction.
A {\it complete pairing} of a graph with $2p$ nodes, with respectively $m_k$ edges, is a graph in which 
each edge from a node is connected to an edge with a different nod.

\begin{lemma}\label{basic-graph}
	\
	\begin{enumerate}
		\item 	Let $\Gamma$ denote a graph of pairings of edges (no self-connection is allowed).  
		Define:
		\begin{align*}
		I(\epsilon,p)&= \overbrace{ \int_0^{t} \dots  \int_0^{t} }^{2p}
		\left|  \prod_{\{  (s_k, s_l) \} \in \Gamma}  \E(y^{\epsilon}_{s_k} y^{\epsilon}_{s_l})\right|\;ds_1\dots ds_{2p}.
		\end{align*}
		Let $m_k$ denote the number of edges issuing from the node $s_k$.
		Then 
		\begin{equation}
		I(\epsilon,p) 	
		\lesssim \; \prod_{k=1}^{2p}  
		\sqrt{ t \int_{-t} ^{t}  | \rho^{\epsilon}(s)| ^{m_k }\; ds }
		\lesssim  \prod_{k=1}^{2p} \f{ t^{ H^*(m_k) \vee \f 1 2} } {\alpha \left(\epsilon, H^*(m_k) \right)}.
		\end{equation}

		\item		Let $G_i, G_j: \R\to \R$ be  functions in finite chaos with Hermite ranks $m_{G_i}$ and $m_{G_j}$ respectively.
		Then,
		\begin{align*}
		\Vert \XX^{i,j,\epsilon} \Vert_{L^p}&:=  \alpha_i ( \epsilon))
		\alpha_j ( \epsilon)\,\left\| \int_0^{t} \int_0^s G_i(y^{\epsilon}_s) G_j(y^{\epsilon}_r) dr ds\right\|_{L^p}\\
		&\lesssim t^{H^*(m_{G_i}) \vee \f 1 2 + H^*(m_{G_j}) \vee \f 1 2}.
		\end{align*} 
		
	\end{enumerate}
\end{lemma} 
\begin{proof}
	Different  graphs yield different asymptotics, the `worst' graph is the one with exactly $m$ edges at each node, and all edges of a given node are linked to the same node.
	For a general graph, let us deal with the first variable $s_1$. We first count forward and observe
	$$ \prod_{\{  (s_k, s_l) \} \in \Gamma}  \E(y^{\epsilon}_{s_k} y^{\epsilon}_{s_l})=\prod_{k=1}^{2p}  \prod_{l=k+1}^{2p}  \E(y^{\epsilon}_{s_k}y^{\epsilon}_{s_{l}}) 
	=\prod_{k=1}^{2p} \prod _{\{l: l>k, \, l\in \Gamma_k\}} (\rho^{\epsilon}(s_l-s_k))^{n(l, k)},$$
	where $\Gamma_k$ denotes the subgraph of nodes  connected to $k$, from the nodes with index greater than $k$.
	Using H\"older's inequality we obtain
	\begin{align*} 
	&\int_0^{t}  \prod _{\{k:k >1, \, k\in \Gamma_1\}} | \rho^{\epsilon}(s_1-s_k)|^{n(1, k)}ds_1 \\
	&\le   \prod _{\{k: k >1, \, k \in \Gamma_1\}}\left(  \int_0^{t}  |\rho^{\epsilon}(s_1-s_k)|^{m_1}\, d s_1  \right)^{ \f {n(1, k)} {m_1} } \\
	&\le  \int_{-t} ^{t}  | \rho^{\epsilon}(s_1)| ^{m_1} ds_1.
	\end{align*}	
	We have used $\sum_{\{ k>1:  k\in \Gamma_1\}} n(1,k)=m_1$, the number of edges at node $1$. 
	We then peel off the integrals layer by layer, and proceed with the same
	technique to the next integration variable.
	For example suppose the remaining integrator containing $s_{2}$ has the combined exponent $\tau_2=\sum_{k=2}^{2p} n(2,k)$, ($\tau_1=m_1$). By the same procedure as for $s_1$ we score a factor
	$$ \int_{-t} ^{t}  | \rho^{\epsilon}(s_2)| ^{\tau_{2}} ds_2. 
	$$
	By induction and putting estimates for each integral together,
	$$ \overbrace{ \int_0^{t} \dots  \int_0^{t} }^{2p}
	\prod_{k=1}^{2p} \prod _{\{l: l>k, \, l\in \Gamma_k\}} (\rho^{\epsilon}(s_l-s_k))^{n(l, k)}\;ds_1\dots ds_{2p} \lesssim  \prod_{k=1}^{2p}  \int_{-t} ^{t}  | \rho^{\epsilon}(s)| ^{\tau_k} ds.$$
	Following \cite{BenHariz},  we reverse the procedure in the estimation for the integral kernel, take $\xi_k$ to be the number of edges connected to
	nodes in  the backward direction, then $\xi_{k}=\sum_{l=1}^{k} n(l,k)$ and the same reasoning leads to the followign estimate:
	$$  \overbrace{ \int_0^{t} \dots  \int_0^{t} }^{2p}
	\prod_{k=1}^{2p} \prod _{\{l: l>k, \, l\in \Gamma_k\}} (\rho^{\epsilon}(s_l-s_k))^{n(l, k)}\;ds_1\dots ds_{2p}
	\lesssim  \prod_{k=1}^{2p}  \int_{-t} ^{t}  | \rho^{\epsilon}(s)| ^{\xi_k } ds.$$
	Since $\tau_k+\xi_k=m_k$ and	$$  \int_{-t} ^{t}  | \rho^{\epsilon}(s)| ^{\tau_k } ds \int_{-t} ^{t}  | \rho^{\epsilon}(s)| ^{\xi_k } ds  
	\le  2 t  \int_{-t} ^{t}  | \rho^{\epsilon}(s)| ^{m_k } ds,$$
	and therefore
	\begin{align*}\label{single-graph}
	\left( \overbrace{ \int_0^{t} \dots  \int_0^{t} }^{2p}
	\prod_{k=1}^{2p} \prod _{\{l: l>k, \, l\in \Gamma_k\}} (\rho^{\epsilon}(s_l-s_k))^{n(l, k)}\;ds_1\dots ds_{2p}\right)^2
	&   \lesssim  
	\; \prod_{k=1}^{2p} \left( t  \int_{-t} ^{t}  | \rho^{\epsilon}(s)| ^{m_k } ds  \right).
	\end{align*}
	By Lemma \ref{Integrals} for each~$k$,
	$$ \alpha \left(  \epsilon, H^*(m_k)\right)^2   t
	\int_{-t}^{t}  | \rho^{\epsilon}(s)| ^{m_{k}} ds  \lesssim\left(  t^{ H^*(m_k) \vee \f 1 2}\right)^2,$$
	the first part of the lemma follows.
	
	Next let us consider $G_i=H_k$ and $G_j=H_l$. Then we are in a position to apply the first part of the lemma:
	\begin{align*}
	&\E\left( \int_0^{t}\!\!\! \! \int_0^s H_k (y^{\epsilon}_s) H_l(y^{\epsilon}_r) dr ds \right )^p\\
	= &\int_0^{t}\!\!\! \!  \int_0^{s_p}\dots \int_0^{t} \!\!\! \! \int_0^{s_1} \E \left( \prod_{i=1}^p H_{k}(y^{\epsilon}_{s_i}) H_l(y^{\epsilon}_{r_i} )\right)
	\prod_{i=1}^p dr_i \, ds_i\\
	\leq&\sum_{\Gamma}	\int_0^{t}\!\!\! \!  \int_0^{s_p}\dots \int_0^{t} \!\!\! \! \int_0^{s_1} 
	\prod_{k=1}^{2p} \prod_{\{l: l>k, \, l\in \Gamma_k\}}  \rho^{\epsilon}(s_l-s_k)^{n(l, k)} \prod_{i=1}^p dr_i \, ds_i\\
	\lesssim &\, C_{k,l} \f { t^{p ( H^*(k) \vee \f 1 2) + p(H^*(l) \vee \f 1 2)} }{	 \alpha(\epsilon, H^*(k) )^p \alpha(\epsilon, H^*(l))^p}
	\; ,
	\end{align*} 
	where the summation is over all graphs of complete parings and 
	$C_{k,l}$ denotes the number of graphs needed for computing the expectations.

	For $G_i=\sum_{k=m_{G_i}}^N a_{i,k} H_k$ and $G_j=\sum_{k=m_{G_j}}^N a_{i,k} H_k$, we expand the products in the multiple integrales.
	Each summand then has exactly $p$ factors  from $G_i$, for those $k\ge m_{G_i}$,  and $p$ from $G_j$ for those $k\ge m_{G_j}$. 
	Splitting them accordingly we have,
	\begin{equation}\label{single-graph}
	\Vert \XX^{i,j,\epsilon} \Vert_{L^p}     \lesssim 
	\alpha_i(\epsilon) \left(\prod_{k_1=1}^{p} t \int_{-t}^{t}  | \rho^{\epsilon}(s)| ^{m_{k_1}} ds \right)^{\f 1p}
	\alpha_j(\epsilon)  
	\left(\prod_{k_2=1}^{p}  t \int_{-t}^{t}  | \rho^{\epsilon}(s)| ^{ m_{k_2}}  \, ds  \right)^{\f 1p}.
	\end{equation}
	where $m_{k_1} \ge m_{G_i}$ and $m_{k_2} \ge m_{G_i}$. The treatment for the two products are the same. Let us consider the first factor  on the right hand side.
	Since
	$$\int_{-t}^{t}  | \rho^{\epsilon}(s)| ^{m} ds$$
	decreases with $m$, then those terms with the Hermite ranks of $G_i$ and $G_j$ as exponents give the fastest possible blow up, or the slowest convergence rate.
	Since both $G_i$ and $G_j$ belong to the finite chaos,$( \Vert \XX^{i,j,\epsilon} \Vert_{L^p} )^p$
	is  the sum of  finite terms of the form
	$$ \alpha_i(\epsilon) \alpha_j(\epsilon) \int_0^{t} \int_0^s\dots \int_0^{t} \int_0^s \prod_l \E [H_{k}(y^{\epsilon}{s_i}) H_l(y^{\epsilon}_{r_l} )]
	ds_l \, dr_l,$$
	each of these has the same type type bound (by the previous computation), with the constant $C_{k,l}$  in front of the relevant expansions uniformly bounded.
	We may conclude that
	$\Vert \XX^{i,j,\epsilon} \Vert_{L^p}     \lesssim   t^{ ( H^*(k) \vee \f 1 2) + (H^*(l) \vee \f 1 2)}$, finishing the proof.
\end{proof}

For functions not belonging to the infinite chaos we must count the number of graphs in the computation,  and need some assumptions. 
Let $M(\{m_1, \dots, m_{2p}\})$ denote the cardinality of  complete pairings of a graph ith $2p$ nodes, with respectively $m_k$ edges.
In \cite{Graphsnumber} it was shown that 
$$ M\left(m_1,m_2, \dots, m_{2p} \right) \leq \prod_{k=1}^{2p} (2p-1)^{\f {{m_k}} {2}} \sqrt{m_k}.$$
This leads to Assumption \ref{assumption-multi-scale} (1), which restricts the $G_i$ to the class of functions whose coefficients in the Hermite expansion decays sufficiently fast.

\begin{proposition}\label{tightness-lemma}
	Suppose that $G_k \in L^{p_k}$  and satisfies Assumption \ref{assumption-multi-scale}. 
	Then  one has for each $i,j =1, \dots N$,
	$$   \left\Vert \alpha_i(\epsilon) \alpha_j(\epsilon) \int_0^{t}\!\!\! \int_0^s G_i(y^{\epsilon}_s) G_j(y^{\epsilon}_r) dr ds \right\Vert_{L^p} \lesssim t^{H^*(m_{G_i}) \vee \f 1 2 + H^*(m_{G_j}) \vee \f 1 2}.$$
	In particular the family 
	$$\left\{ \alpha_k(\epsilon) \int_0^t G_k(y^{\epsilon}_s) ds, \alpha_i(\epsilon) \alpha_j(\epsilon) \int_0^{t} \int_0^s G_i(y^{\epsilon}_s) G_j(y^{\epsilon}_r) dr ds, k,i,j=1, \dots N \right\}$$ is tight in $\FC^\gamma$ for any
	$\gamma \in \left( \f 1 3,\left( H^*(m_{G_i}) \vee \f 1 2 + H^*(m_{G_j}) \vee \f 1 2 \right) - \f 1 {p} \right)$.
\end{proposition}
\begin{proof}
	Using the above estimates we compute, and the fact that $\rho(s) >0$,
	\begin{align*}
	&\E \left( \alpha_i(\epsilon) \alpha_j(\epsilon) \int_0^t \int_0^s G_i(y_s) G_j(y_r) dr ds \right)^p \\
	& \leq \alpha_i(\epsilon)^p \alpha_j(\epsilon)^p\left \vert \E \left( \int_0^t\!\!\! \int_0^s \sum_{k,k'} c_{i,k} c_{j,k'} H_k(y_s) H_{k'}(y_r) dr ds \right)^p\right \vert.
	\end{align*}
	We estimate the integrals on the right hand side:
	\begin{align*}
	&\left \vert \E \left( \int_0^t\!\!\! \int_0^s \sum_{k,k'} c_{i,k} c_{j,k'} H_k(y^{\epsilon}_s) H_{k'}(y^{\epsilon}_r) dr ds \right)^p\right \vert \\
	&\leq \left\vert
	\sum^{\infty}_{k_1, \dots k_p = m_{G_i}  } \sum_{ k'_1, \dots k'_p = m_{G_j}}^{\infty} 
	\prod_{l=1}^{p}  c_{i,k_l} c_{j,k'_l} \int_0^t\!\!\! \int_0^s \dots
	\int_0^t \!\!\!\int_0^s \prod_{l=1}^{p} \E \left(H_{k_l}(y^{\epsilon}_{s_l})  (H_{k'_l}(y^{\epsilon}_{r_l}) \right) dr_l ds_l \right\vert \\
	&\leq  \sum^{\infty}_{k_1, \dots k_p = m_{G_i}  } \sum_{ k'_1, \dots k'_p = m_{G_j}}^{\infty} 
	\prod_{l=1}^{p} \vert c_{i,k_l} c_{j,k'_l}
	\vert  \int_0^t \!\!\!\int_0^s \dots \int_0^t \!\!\!\int_0^s  \sum_{\Gamma} \prod_{v=1}^{2p} \prod _{\{u: u>v, \, u\in \Gamma_v\}}  \rho^{\epsilon}(s_u-s_v)^{n(u, v)}  ds_u  ds_v  \\ 
	&\leq \sum^{\infty}_{k_1, \dots k_p = m_{G_i}  } \sum_{ k'_1, \dots k'_p = m_{G_j}}^{\infty} \prod_{l=1}^{p} \vert c_{i,k_l} c_{j,k'_l} \vert  
	\overbrace{ \int_0^t \dots \int_0^t}^{2p}  \sum_{\Gamma} \prod_{v=1}^{2p} \prod _{\{u: u>v, \, u\in \Gamma_v\}}  \rho^{\epsilon}(s_u-s_v)^{n(u, v)} ds_u  ds_v.
	\end{align*}
	We then apply estimates from the first part of Lemma \ref{basic-graph},
	\begin{align*}
	&\E \left( \alpha_i(\epsilon) \alpha_j(\epsilon) \int_0^s G_i(y^{\epsilon}_s) G_j(y^{\epsilon}_r) dr ds \right)^p \\
	&\lesssim t^{p \left( H^*(m_{G_i}) \vee \f 1 2 + H^*(m_{G_j}) \vee \f 1 2 \right)} \sum^{\infty}_{k_1, \dots k_p = m_{G_i}  } \sum_{ k'_1, \dots k'_p = m_{G_j}}^{\infty} \prod_{l=1}^{p} \vert c_{i,k_l} c_{j,k'_l} \vert  M(k_1, \dots ,k_p , k'_1 , \dots , k'_p) \\
	&\lesssim t^{p \left( H^*(m_{G_i}) \vee \f 1 2 + H^*(m_{G_j}) \vee \f 1 2 \right)} \sum^{\infty}_{k_1, \dots k_p = m_{G_i}  } \sum_{ k'_1, \dots k'_p = m_{G_j}}^{\infty} \prod_{l=1}^{p} \vert c_{i,k_l} c_{j,k'_l} \vert \sqrt{k_l! k'_l!} (2p-1)^{\f {k_l + k'_l} {2}},
	\end{align*}
	By assumption, the double power series  is finite. This proves the required moment bounds for the second order process.
	For tightness we argue by Lemmas \ref{Lp-bounds} and \ref{tightness-second-order}, concluding the proof.
\end{proof}

\subsection{Weak convergence in $\FC^{\gamma}$, concluding Theorem B.}
\label{iterated-CLT}

We are ready to show weak convergence of the rough path lifts. Denote by $X^W$ the limiting Wiener processes in Lemma \ref{joint-clt-3}. Then, for $i,j \leq n$, we can form the It\^o integrals $\int X_s^i dX_s^j$ and denote it by $\XX^{i,j}$. 
If  either the ith or the jth component limit is not given by a Wiener process,  we will see that $\XX^{i,j,\epsilon}_{0,t} = \int_0^t X^{i,\epsilon} dX^{j,\epsilon}$ converges weakly to a process with higher regularity, which, as a rough path, does not exert any influence on the interpretation of the rough integral,  and therefore has no effect on  the effective equation.

\begin{theorem}\label{theorem-weak-convergence-in-rough-path-topology}
	Assume that $G_k \in L^{p_k}(\mu)$ satisfy Assumption \ref{assumption-multi-scale}, then
	$$ \X^{\epsilon}=(\X^{\epsilon},\XX^{\epsilon}) \to \X=(X,\XX+A(t-s)),$$
	weakly in $\FC^{\gamma}$ for $\gamma \in (\f 1 3, \f 1 2 - \f {1} {\min_{k\leq n } p_k})$.
\end{theorem}

\begin{proof}
	By Proposition \ref{lift-CLT} $\X^{\epsilon}$ converges in finite dimensional distributions and by Lemma \ref{moment-conditions}, Lemma \ref{tightness-second-order},  Proposition \ref{tightness-lemma} and  Lemma \ref{Lp-bounds} show that the convergence takes place is the respective H\"older spaces.
\end{proof}

With Theorem \ref{theorem-weak-convergence-in-rough-path-topology} together with Proposition \ref{joint-clt-3}  we conclude Theorem B.

\section{ Multi-scale homogenization theorem }\label{conclusion}

We are now in the position to complete the proof for  the  multi-scale homogenization theorem for the long range dependent case, this is Theorem A, which 
we formulate here how it is proved.
\begin{theorem}
\label{proof} Let $H \in  \left( \f 1 2 ,1 \right)$, $f_k\in \C_b^3( \R^d ;\R^d)$, and $G_k$  satisfies Assumption~\ref{assumption-multi-scale}.
 Then, the following statements hold.
\begin{enumerate}
\item The solutions  $x_t^\epsilon$ of (\ref{multi-scale}) 
converge weakly  in  $\C^\gamma$ on any finite interval and for any $\gamma \in (\f 1 3 ,\f 1 2 - \f 1 {\min_{k \leq n } p_k})$.
\item 
The limit  solves  the rough differential equation
\begin{equation}
\label{effective-rde}
 d {x_t} =f(x_t) d \X_t \quad  x_0=x_0
 \end{equation}
 Here $f=(f_1, \dots, f_N)$ and $ \X=(X,\XX_{s,t}+(t-s)A)$ is a rough path over $\R^N$.
 (They will be described  in Theorem B below.)
 \item  Equation (\ref{effective-rde}) is equivalent to the  stochastic  equation below:
 $$ d{x}_t =\sum_{k=1}^n f_k(x_t) \circ d X^k_t+\sum_{l=n+1}^N f_l(x_t) d X^l_t, \quad x_0=x_0,$$
where $\circ$ denotes  Stratonovich integral, otherwise a Young integral. 
\end{enumerate}
\end{theorem}

\begin{proof}
We want to formulate our slow/fast random differential equation  as a family of rough differential equations, such that the drivers converge in the rough path topology. Using the continuity of the solution map, we obtain weak convergence of the solutions to  a rough differential equation. In the appendix, we interpret the
rough differential equation as a mixed Stratonovich and Young integral equation and the  coefficients of the equation will be computed.

Let us set  $F: \R^d \to {\mathbb L} (\R^N, \R^d)$ as below:
$$F (x) (u_1, \dots, u_m)= \sum_{k=1}^N    u_kf_k(x),$$
and for the standard o.n.b. $\{e_i\}$ of $\R^d$ we set $F_i(x)=F(x)(e_i)$.
If we set $$G^\epsilon=\Big(\alpha _1(\epsilon) G_1, \dots, \alpha_N(\epsilon) G_N\Big),$$
we may then write the equation as follows:
$\dot x_t^\epsilon =F(x_t^\epsilon) G^\epsilon(y_t^\epsilon)$.
Now define the rough path $\X^{\epsilon}=(X^{\epsilon},\XX^{\epsilon})$, where
\begin{align*}
X^\epsilon_t &=\Big(\alpha _1(\epsilon) \int_0^{t} G_1(y^\epsilon_s)ds , \dots,  \alpha _N(\epsilon) \int_0^{t} G_N(y^\epsilon_s)ds \Big)\\
&=(X_t^{1,\epsilon}, \dots, X_t^{N,\epsilon})\\
\XX^{i,j,\epsilon}_{s,t}&= \int_s^t (X^{i,\epsilon}_r - X^{i,\epsilon}_s ) dX^{j,\epsilon}_r.
\end{align*}
We may therefore rewrite our  equation  as  a rough differential equation with respect to $\X^\epsilon$:
$$d x_t^\epsilon =F(x_t^\epsilon) d \X^\epsilon(t).$$
with covariance as specified in  Theorem \ref{theorem-CLT}.

By Theorem \ref{theorem-weak-convergence-in-rough-path-topology},
$\X^{\epsilon}$ converges to $\X=(X, \XX + (t-s)A)$ in $\FC^{\gamma}$ where $\gamma \in (\f 1 3 , \f 1 2 - \f {1} {\min_{k \leq n } p_k })$.

Since  $\gamma>\f 13$ by Assumption \ref{assumption-multi-scale}, 
We may apply the continuity theorem for rough differential equations, Theorem \ref{cty-rough},
to conclude that the solutions converge to the solutions of the rough differential differential equation
$$\dot x_t=F(x_t) d\X_t.$$
Since $F$ belongs to $\C_b^3$, this is well posed in the rough path equation sense. 
We completed the proof for the convergence.
\end{proof}

\section{Appendix}
\subsection{Interpreting the effective dynamics by classical equations}

We now explain what  the limiting equation means in the classical sense.  Although it merits a verification for the integrand consists of correlated Winer with drift 
block and correlated Hermite block, the answer is obvious and believed for those working in rough path theory.  For the slow/fast and homogenization community
this is the mysterious part, and in any case  the multi-dimension path notation needs some classification. 
Our set up is the following. 

\begin{assumption}
Let $X_t=( X^W_t, X^Z _t)$ where $X^W_t$ is a multi-dimensional  possibly correlated Wiener process and $X^Z_t$ a multi-dimensional Hermite process. The two components $X^W_t$ and $X^Z_t$ are not correlated, we  denote by $A$ the block matrix
$$A:=\left(\begin{matrix}\cov (X^W)&0\\0 &0
\end{matrix}\right). $$
We write  $A^{i,j}$ for the component of $A$.
We are concerned with the classical interpretation for the rough differential equation
$$\dot x_t=F(x_t) d\X_t,$$
where $F: \R^d \to \L(\R^N, \R^d)$ is a $BC^3$ map and $\X=(X, \XX+(t-s)A)$
where $\XX=(\XX^{i,j})$ denotes the `Canonical lift' of $X$, 
$$\XX^{i,j}_t =\int _0^t X_s^i \;dX_s^j$$
interpreted as  It\^o integrals if $i,j \leq n$, otherwise as Young integrals.  
\end{assumption}

According to the general theorems on rough differential  equation there exists a unique solution
 in the controlled rough path space $D_{X}^{2\alpha}([0,1]; \R^d)$  where $\alpha>\f 13$.  The solution exists global in time 
and the full controlled process is given by $(x_s, F(x_s))$. See \cite{Lyons94, Friz-Hairer}.

\bigskip

We begin with setting the notation and at the same time explaining the raison d'\^etre for the definition of rough integrals.
Given a rough path $(X, \XX)$ and a controlled rough path in 
	$\D_X^{2 \alpha}$ is a pair of processes $(Y, Y')$ with the properties
	$$Y_{s,t}=Y_s' X_{s,t}+R_{s,t},$$
	where $Y'\in \C^\alpha(\L(\R^d, \L(\R^N,\R^d))$ and the two parameter function $R$, satisfies $\|R\|_{2\alpha}< \infty$. 
	Let $\alpha>\f 13$.
The rough path integral is given by the enhanced Riemann sums	$$\begin{aligned}
	&\int_s^t Y \,d{\mathbf X} =\lim \sum_{[u,v]\in \CP} Y_{u} X_{u,v} + Y'_{u} \XX_{u, v}.	\end{aligned} $$ 
The remainder $R_{s,t}$'s contribution  is  of order $|t-s|^{3\alpha}$ term, which sums to zero for $\alpha>\f 13$ and can be ignored.
The limit is along any sequences of partitions with mesh converging to zero.

 We therefore seek two processes $L_s$, $R$ and an expression
$$F(x_t)-F(x_s)=L_s(X_t-X_s) +R_{s,t},$$
where $\sup_{s\not =t, s,t \in [0,1]} \vert \f{R_{s,t}} {|t-s|^{2\alpha}} \vert \le C$.  Taylor expanding $F(x_t)$, one can deduce that $L_s=DF(x_s) F(x_s)$.
By Taylor's theorem,
$$\begin{aligned}
F(x_t)&=F(x_s)+DF(x_s) (x_t-x_s)\\
&+\f 12 \int_0^t(1-u) \Hess (F) (x_s+u(x_t-x_s)) (x_t-x_s, x_t-x_s) du.\end{aligned}$$
Here $DF(x)(v)=\sum_{i=1}^d \f {\partial  F}{\partial x_i } v_i  $, 
$\Hess F$ is the Hessian of $F$, and
$$\Hess F(x)(e,v)=\sum_{i,j=1}^d \f {\partial ^2F}{\partial x_i \partial x_j} e_i v_j.$$
The last terms in the Taylor expansion is  of order $\C^{2\alpha}$, one $\alpha$ each from $x_t-x_s$, and so goes into the $R$ term. 
We express $x_t-x_s$ in terms of $X_t-X_s$:
$$\begin{aligned}DF(x_s) (x_t-x_s)
&=DF(x_s)  (\int_s^t F(x_r) dX_r)\\
&\sim DF(x_s)  (F(x_s) (X_t-X_s)) + R_{s,t}^1.
\end{aligned}$$
The  $R^1$ term is of order $|t-s|^{2\alpha}$.
$$x_t-x_s\sim \sum F(x_u)(X_v-X_u) +DF(x_u)  (F(x_u) \XX_{v,u}).$$
If $\X=(W,\WW )$ to be the standard Brownian motion with its It\^o lift, i.e. $\WW^{i,j}_{s,t}= \int_s^t  (W_r^i-W_s^i)dW_r^j$
then
$$ \sum DF(x_u)F(x_u)\WW_{u,v}\to 0$$
in probability. This means the equation is the It\^o integral.  
If $X$ is a correlated Wiener process,  choose $U$ such that $U^TU=A$. Suppose that
 the Wiener process block is: 
$(X,  \XX+A t) =(U W,   \int_s^t  U(W_r-W_s)dUW_r)+\frac 1 2At )$. This leads to an SDE with Stratonovich integral
$$dx_t=F(x_t)U\circ dW_t.$$
This comes from the following fact.  Let  $\X=(X, \XX)$, $Z=(Z, \ZZ)$ be two rough paths in $\C^{\alpha}$, and  $(Y,Y')$ is controlled by $X$, i.e. $(Y, Y') \in \D_X^{2\alpha}$.
	If $g$ is a $2\alpha$-H\"older continuous functions such that 
	$$Z_t=X_t. \qquad  \ZZ_{s,t}=\XX_{s,t} +g(t)-g(s),$$
	then according to \cite{Friz-Hairer}, $(Y, Y') \in \D_Z^{2\alpha}$, and
	$$\begin{aligned}
	&\int Y \,d{\mathbf Z} =\lim \sum_{[s,t]\in \CP} Y_{s} Z_{s,t} + Y'_{s} \ZZ_{s, t}\\
	&=\lim  \sum_{[s,t]\in \CP}( Y_{s} X_{s,t} + Y'_{s} \XX_{s, t}+(g(t)-g(s)) Y_s' ) \\
	&=\int Y_s \, d\X_s+  \int Y_s' dg.
	\end{aligned} $$ 
	the last integral is a Young integral which is well defined since $Y'\in C^{\alpha}$ and $g\in C^{2\alpha}$.
	
	For the Hermite component $Z$  the secondary process makes no visible contribution in the limit of the enhanced Riemann sum  
	and the rough integral $\int_0^t f_i(x_s) d\X_s^i$ agrees with the Young integral.	
 Take $X^i$ or $X^j$, whose sum of regularity is then greater than $1$,  the secondary process $\int_s^t X^i_{s,r}   dX_r^j$, thus
  the enhanced Riemann sum limit is a Young integral. We may now conclude.

\begin{remark}
The solution of the rough differential equation $\dot x_t=F(x_t) d\X_t$, where $\X_t=( UW_t+\f 12 At,Z_t)$ with  canonical lift,  agrees almost surely with the solution
$$dx_t=F_1(f_t) U\circ dW_t+F_2(x_t) dZ_t,$$
where $F_1$ denotes $F$ restricted to its first $n$ components and $F_2$ denotes $F$ restricted to the remaining $N-n$ components.
By the standard theorem, also,
	the mixed integral equations is well posed,  global, and continuous in the initial data.
\end{remark}

\subsection{Auto-correlation and moments of  fOU}
{\bf  Lemma \ref{correlation-lemma}}
	Let  $H\in (0, \f 12)\cup (\f 12, 1)$.  For any $t\not =s$, \begin{equation}\label{cor1-II}
	|\rho(s,t)| \lesssim 1\wedge |t-s|^{2H-2}.
	\end{equation}

\begin{proof}
	We give an indicative proof and fix $H>\f 12$. It is sufficient to prove this for $|t-s|$ large, then
	$${\begin{split}
		\f{\rho(s,t)}{ \sigma^2 \, H(2H-1) }&= \int_{-\infty}^\infty \int_{-\infty}^{ (2t-v) \wedge (2s+u) }  e^{-(s+t-u) }  |v|^{2H-2} \; du\;dv\\
		&= \int_{-\infty}^\infty  e^{- |v-(t-s)| }  |v|^{2H-2} dv
		=  \int_{-\infty}^\infty  |v+t-s|^{2H-2} \; e^{- |v| } \;  dv.
		\end{split}}
	$$
	The integration region breaks up into three:
	$$(A) \qquad  \f 12 |t-s| \le |v| \le  2 |t-s|, \qquad (B) \quad    |v|\le \f 12 |t-s| \quad \hbox{ or } \quad   |v| \ge  2 |t-s|.$$
	In region (A), $$
	\int_{ \f 12 |t-s|  \le |v|\le 2|t-s|}  |v+t-s|^{2H-2} \;  e^{-\f 12  |t-s| }\; dv
	\le   |t-s|^{2H-1} e^{-\f 12  |t-s| }.$$
	For $|t-s|$ large this gives better bound than $|t-s|^{2H-2}$.
	In region (B),  since $2H-2<0$, 
	$$  \int_B   |v+t-s|^{2H-2} \;  e^{- |v| }\; dv
	\le   |t-s|^{2H-2}\int_{-\infty}^\infty  e^{- |v| }\; dv,$$
	giving the correct rate.
For $H< \f 12$, we have on one hand the large time asymptotics from
		\cite{Cheridito-Kawaguchi-Maejima}: 
		$\rho(s)=2\sigma^2 H(2H-1) s^{2H-2} +O(s^{2H-4})$, on the other hand 
		$\E (y_s y_t ) \leq  \Vert y_s \Vert_{L^2} \Vert y_y \Vert_{L^2} \leq 1$, showing that $\rho$ is locally bounded and concluding the proof.
\end{proof}

{\bf Proof for Lemma  \ref{Integrals}.} This comes down to the following statement:  we only need to show that for $\epsilon \in (0,\f 1 2]$, the following holds uniformly :

\begin{equation} \label{correlation-decay-2-3}
\left( \int_0^{t} \int_0^{t}  \vert  \rho^{\epsilon}(u,r) \vert^m\, dr \,du\right)^{\f 12} \\
\lesssim 
\left\{	\begin{array}{lc}
\sqrt { t \epsilon  \int_0^\infty \rho^m(s) ds} ,  \quad  &\hbox {if} \quad H^*(m)<\f 12,\\
\sqrt { t \epsilon  \vert \ln\left(\f t \epsilon \right) \vert}, \quad  &\hbox {if} \quad H^*(m)=\f 12,\\
\left(  \f t \epsilon \right) ^{H^*(m)-1},  \quad &\hbox {if} \quad H^*(m)>\f 12.
\end{array} \right.
\end{equation}		

\begin{proof}
We first observe that
\begin{equation}\label{correlation-decay-3}
 \int_0^\infty \rho^m(s) ds <\infty \quad \Longleftrightarrow \quad  H^*(m)<\f 12\quad 
 \Longleftrightarrow \quad H<1 -\f 1{2m}.
\end{equation}	
By a change of variables and using  estimate (\ref{cor1}) on the decay of the auto correlation function (\ref{cor1}),
\begin{align*}
\int_0^{\f t \epsilon} \int_0^{\f t \epsilon} \vert \rho(\vert r-u \vert)\vert^m dr du 
&=2 \f t \epsilon  \int_0^{\f t \epsilon} \vert \rho(s)\vert^m ds\\
&\lesssim  \left\{	\begin{array}{cl} 
\f t \epsilon \int_0^\infty \rho^m(s) ds,  &\hbox{ if }H^*(m)<\f 12,\\
\left( \f t \epsilon \right)^{2H^*(m)}, &\hbox{ otherwise.}
\end{array} \right. ,
\end{align*}
For the case $H^*(m)=\f 1 2$ we use 
\begin{align*}
\int_0^{\f t \epsilon} \vert \rho(s) \vert^m ds &\leq \int_0^{\f T \epsilon} \vert \rho(s) \vert^m ds
\lesssim  \int_0^{\f T \epsilon} (1 \wedge \f 1 s) ds \lesssim    \vert \ln \left( \f T \epsilon \right) \vert 
\lesssim  \vert \ln \left( \f 1 \epsilon \right) \vert.
\end{align*}
To complete the proof we observe that by a simple change of variables,
\begin{align*}
\int_0^{t} \int_0^{t}  \vert  \rho^{\epsilon}(u,r) \vert^m\, dr \,du &= \epsilon^2 \int_0^{\f t \epsilon} \int_0^{\f t \epsilon}  \vert  \rho(u,r) \vert^m\, dr \,du,
\end{align*}
concluding the proof.
\end{proof}

{\bf Lemm \ref{cty-lemma}.}
For any $\gamma \in (0, H)$, $p>1$, and  $s,t\in [0, \infty)$,
the following estimates hold:  
 $$\|y_s-y_t\|_{L^p} \lesssim   1 \wedge |s-t|^{H}, \qquad  \E  \sup_{s\not =t} \left( \f{  |y_s-y_t |} { |t-s|^{\gamma}} \right)^p \lesssim  C(\gamma, p)^p M.$$

\begin{proof}

We  use the Ornstein-Uhlenbeck equation
$$y_s-y_r=-  \int_r^s y_udu+B_s-B_r,$$
to obtain $\E|y_s-y_r|^2 \lesssim   (s-r)^2 \E \vert y_1 \vert ^2+q|s-r|^{2H}$.
Using stationality of $y_t$,  one has also $ \E|y_s-y_r|^{2} \leq  2\E |y_1|^2=2$.
Since for Gaussian random variables,  the estimate for the  $L^{2p}$ norm is the same as for its  $L^2$ norm, we have
 $$\|y_s-y_r\|_{L^p} \lesssim \left\{\begin{aligned} 
& 1,  &\hbox{ if } |s-r|\ge 1;\\
& |s-r|^{H}, \quad &\forall |s-r|\le 1. \end{aligned}\right.$$

By symmetry, and change of variables, 
$$\begin{aligned} &\int_0^T \int_0^T \f { \E|y_s-y_r|^p}{ {|s-r|}^{ \gamma p+2}}
\\
&\lesssim \int_0^1 \int_{-v}^{v} v^{Hp-\gamma p-2}\, du\,dv
+\left( \int_1^T \int_{-v}^v + \int_T^{2T}  \int_{2T+v}^{2T-v}\right)  v^{-\gamma p-2}du dv\\
&\lesssim  1+ T^{-\gamma p}\lesssim 1.
\end{aligned} $$
The first term is finite as soon as $\gamma <H$, the second term is finite as soon as $\gamma$ is positive.
The remaining claim follows from the well known Garcia-Rodemich-Romsey inequality below.
\end{proof}

\begin{lemma}
[Garcia-Rodemich-Romsey-Kolmogorov inequality]
\label{lemma-GRR}
Let $T>0$.
\begin{itemize}
\item [(1)]  Let $\theta: [0,T]\to \R^d$. For any positive numbers $\gamma, p$,  there exists a constant $C(\gamma, p)$ such that 
$$\sup_{s\not =t , s,t \in [0,T]} \f{|\theta(t)-\theta (s)|}{|t-s|^\gamma} \le C(\gamma, p)
\left(  \int_0^T \int_0^T \f { |\theta_s-\theta_r|^p}{ {|s-r|}^{\gamma p+2}}ds dr\right)^{\f 1p}.$$
\item[(2)] Let $\theta$ be a stochastic process.
Suppose that for $s,t\in [0, T]$, $p>1$ and $\delta>0$, 
$$\E|\theta(t)-\theta (s)|^p \le c_p |t-s|^{1+\delta},$$ where $c_p$ is a constant. Then for $\gamma<\f \delta p$,
$$\||\theta|_{C^\gamma([0,T])}\|_p \le C(\gamma, p) ( c_p)^{\f 1p} \left(\int_0^T \int_0^T |u-v|^{\delta -\gamma p-1} dudv\right)^{\f 1p},
$$ the right hand side is finite when $\gamma \in (0, \f \delta p)$.

\end{itemize}
\end{lemma}

 \subsection{Proof of the conditional integrability of fOU}
\label{prove-lemma-int}

 The aim of this section is to prove the estimate (\ref{integrable}),  we therefore restrict ourselves to the case $H> \f 1 2$.
Firstly, we  compute the  conditional expectations of  $\E(G(y_t) | \mathcal{F}_k)$ where $G \in L^2(\mu)$.

We begin with  decomposing the fractional Ornstein-Uhlenbeck process into a part which is $\mathcal{F}_s$ measurable and another one independent of $\mathcal{F}_s$, where $\F_s$ is the filtration generated by 
$B^H$.  In \cite{Additive} it was shown that such a decomposition is available for fractional Brownian motion with $H> \f 1 2$ by using the Mandelbrot-Van Ness representation it was shown that, for $k<t$,
\begin{align*}
B_t -B_k &= \f {1} {c_1(H )} \left( \int_{-\infty}^{k} (t-r)^{H - \f 3 2} - (k-r)^{H- \f 3 2} dW_r + \int_k^{t} (t-r)^{H- \f 3 2} dW_r\right)\\
&= \overline{B}^k_t + \tilde{B}^k_t,
\end{align*}
where $\overline{B}^k_t$ is $  \F_k$ measurable and $\tilde{B}^k_t$ is independent of $  \F_k$ . Furthermore the filtration generated by the fractional Brownian motion is the same as the one generated by the two-sided Wiener process $W_t$.
Using the above we have,
\begin{align*}
y_t &= \int_{-\infty}^{t} e^{-(t-r)} dB_r
= \int_{-\infty}^k e^{-(t-r)} dB_r +   \int_{k}^{t} e^{-(t-r)} d( B_r - B_k)\\
&=  \left( \int_{-\infty}^k e^{-(t-r)} dB_r +\int_{k}^{t} e^{-(t-r)} d\overline{B}^k_t \right)+  \int_{k}^{t} e^{-(t-r)} d\tilde{B}^k_t\\
&= \overline{y}^k_t + \tilde{y}^k_t,
\end{align*}
where the first term $\overline{y}^k_t$ is $\mathcal{F}_k$ measurable and $\tilde{y}^k_t$ is independent of  $\mathcal{F}_k$.
In case $s \leq k$ we set $\overline{y}^k_t = y_s$.

 It is only left so show that for $q \geq m$, $$\int_{k-1}^{\infty} \int_{k-1}^{\infty}   \E \left(    {\overline{y}^k_s} {\overline{y}^k_r}  \right)^q  dr ds $$
is bounded in $q$ and $k$. 
We make use of the following classical result,
\begin{lemma}[\cite{Huang-Cambanis}]\label{Huang-Cambanis}
If $X_s$ is a Gaussian process with covariance $R(t,s) = \E \left( X_s X_t \right)$. Suppose that  $\partial_t \partial_s R(t,s)$  is integrable over every bounded region. Then we have, for any $f,g: \R\to \R$ satisfying  the assumption 
$$\int_a^b \int_a^b \vert f_t \,g_s \, \partial_t \partial_s R(t,s) \vert ds dt < \infty,$$
the following It\^o-isometry,
\begin{equation*}
\E \left( \int_a^b f_t \, dX_t \int_a^b g_s\, dX_s \right) = \int_a^b \int_a^b f_t\, g_s \,\partial_t \partial_s R(t,s)\, ds \,dt.
\end{equation*}
\end{lemma}
\begin{lemma}\label{lemma-corr}
Setting $R(t,s) = \E  \Big(  \overline{B}^k_t   \overline{B}^k_s  \Big)$ and 
$S(t,s) = \partial_t \partial_s R(t,s)$. Then $S$ is regular in the region $\{t\not =s\}$, and 
$$S(t,s) \lesssim  (t\wedge s-k)^{2H- 2} \qquad \forall s,t>k.$$ 
\end{lemma}
\begin{proof}
Recall that $H>\f 12$ and  observe that for $k<t$,

\begin{align*}
{ \overline{B}^k_t}&=  \f{1} {c_1(H)} \int_{-\infty}^{k} \left(  (t-r)^{H - \f 1 2}-(k-r)^{H-\f 1 2} \right) dW_r\\
&= \f{H- \f 1 2}  {c_1(H)} \int_{-\infty}^{k} \int_{k}^{t} (u-r)^{H- \f 3 2} \,du \,dW_r.
\end{align*}
By the It\^o isometry for Wiener processes, for $k<\min(t,s)$,
\begin{align*}
&\E \left( {  \overline{B}^k_t \overline{B}^k_s} \right) \\
&= \left( \f {H- \f 1 2} {c_1(H)}\right)^2 \int_{ {k} }^{t} \int_{ {k} }^{s} \E \left(  \int_{-\infty}^{k} (u-r)^{H- \f 3 2}  dW_r  \int_{-\infty}^{k}  (w-v)^{H- \f 3 2} dW_v \right) dw\, du\\
&= \left( \f {H- \f 1 2} {c_1(H)}\right)^2  \int_{ {k} }^{t} \int_{ {k} }^{s} \int_{-\infty}^{k} (u-y)^{H- \f 3 2} (w-y)^{H- \f 3 2} \,dy \,dw \,du,
\end{align*}
By a change of variables, for  $t>s>k$
\begin{align*}
S(t,s) &= \partial_t \partial_r R(t,r)\\
&=  \left( \f {H- \f 1 2} {c_1(H)}\right)^2  \int_{-\infty}^{k} (t-y)^{H- \f 3 2} (s-y)^{H- \f 3 2} \, dy\\
&=   \left( \f {H- \f 1 2} {c_1(H)}\right)^2  (t-s)^{2H- 2} \int_{\f {s-k} {t-s}}^{\infty} (1+w)^{H- \f 3 2} w^{H- \f 3 2} \,dw,
\end{align*}
which is regular on $\{t>s\}$.  Now
 $$ \int_{\f {s-k} {t-s}}^{\infty}(1+w)^{H- \f 3 2} w^{H- \f 3 2} \,dw \le  \int_{\f {s-k} {t-s}}^{\infty}  w^{2H- 3} \,dw 
 = \f 1 {2-2H}  \left( \f {s-k} {t-s}\right) ^{2H-2}$$ is finite for $H>\f 12$  giving the bound 
 $$S(t,s) \lesssim (t \wedge s-k)^{2H- 2},$$ 
 concluding the proof for $(t,s)$ off diagonal. 
 
On the diagonal, 
\begin{align*}
 \partial_t \partial_t \E \Big( c_1(H) \overline{B}^k_t \Big)^2
=&2 \int_{-\infty}^{k} (t-y)^{H- \f 3 2} (t-y)^{H- \f 3 2} dy\\
&+2  \int_k^t \int_{-\infty}^k \left(H-\f 32\right) (t-v)^{H-\f 52} (s-v)^{H-\f 32} \, dv\,ds \\
= &2\int_{t-k}^{\infty} w^{2H-3} dw\\
&+2 \f{H-\f 32}{H-\f 12}  \int_{-\infty}^k  (t-v)^{H-\f 52} \Big[ (t-v)^{H-\f 12}
-(k-v)^{H-\f 12} \Big] \, dv.
\end{align*}
The second term has no singularity at $k$.
Consequently,
\begin{equation*}
\begin{aligned}
 &\partial_t \partial_t \E \Big( c_1(H) \overline{B}^k_t \Big)^2\\
 &= 2\f {(t-k)^{2H-2}} {2H-2}+ \f{2H-3}{H-\f 12}  \int_{-\infty}^k  (t-v)^{H-\f 52} \Big[ (t-v)^{H-\f 12}
-(k-v)^{H-\f 12} \Big] \, dv\\
&\le  2\f {(t-k)^{2H-2}} {2H-2}+ \f{2H-3}{H-\f 12}  \int_{-\infty}^k  (t-v)^{2H-3}\, dv = C(t-k)^{2H-2}.
\end{aligned}
\end{equation*}
We have used the fact that $ \int_{-\infty}^k  (t-v)^{H-\f 52} (k-v)^{H-\f 12}\, dv$ is a finite number.
This means the double derivative is  only singular at  $t=k$, but not along the whole diagonal.  Moreover, the singularity is integrable.
\end{proof}

\begin{remark}
For the fractional Brownian motion $B_t$ and  for  $t>s$,
 $$\partial_t \partial_s \E \left( B_t B_s \right) = 2H(2H-1) (t-s)^{2H-2},$$ which is singular along the diagonal, whereas as 
 shown in the proof of Lemma \ref{lemma-corr},
$ \partial_t \partial_t \E \Big( \overline{B}^k_t \Big)^2$
has only singularity at  $t=k$, but not along the whole diagonal.  Moreover, the singularity is integrable, we may therefore use the  It\^o isometry,  Lemma \ref{Huang-Cambanis}, freely.
\end{remark}
Now we go ahead and establish estimates for the decay of $\E \left( \overline{y}^{k}_t \overline{y}^{k}_s \right)$.

\begin{lemma}
\label{lemma-integral}
For any $k<s$,
$$ \int_k^s e^{-(s-v)}  (v-k)^{2H-2} \,dv  \lesssim  1 \wedge (s-k)^{2H-2}.$$ 
The constant is independent of $k$.
\end{lemma}
\begin{proof}
This can be seen by splitting the integral  into three regions $\int_k^{k+1}+\int_{k+1}^{\f s 2} +\int_{\f s 2}^s$, and integration by part with the second two terms. 
This leads to: $\int_k^{k+1}   e^{-(s-v)}  (v-k)^{2H-2} \,dv \le  \f 1 {2H-1} e^{-(s-k-1)}$. Furthermore for $s>3k$,
$$\begin{aligned}
&\int_{k+1}^{s}  e^{-(s-v)}  (v-k)^{2H-2} \,dv\\
&=  (s -k )^{2H-2} -e^{-(s-k-1)} (2H-3)  -(2H-2) \left( \int_{k+1} ^{\f s 2} +\int_{\f s 2}^s \right) e^{-(s-v)} (v-k)^{2H-3}  dv\\
&\lesssim (s-k)^{2H-2} -e^{-(s-k)} -e^{-\f s 2} (v-k)^{2H-2} |_{k+1}^ {\f s 2} -(v-k)^{2H-2} |_{\f s 2}^s\\
&\lesssim  (s-k)^{2H-2}+e^{-\f s 2} +(\f s 2 -k)^{2H-2}\\
&\lesssim (s-k)^{2H-2}.
\end{aligned}$$
This gives the required estimate.
\end{proof}

\begin{proposition}\label{integrable-lemma}
 Let $H>\f 12$ and suppose that  $H^*(m)<0$. Then,
$$\sup_k \sup_{q\geq m} \int_{k-1}^{\infty} \int_{k-1}^{\infty}   \E \left(    {\overline{y}^k_s} {\overline{y}^k_t}  \right)^q \, dt \,ds < \infty.$$ 
\end{proposition}

\begin{proof}
We  first compute,  
\begin{align*}
\E \left( \overline{y}^{k}_t \overline{y}^{k}_s \right) =& \E \left( \left( e^{-(t-k)} y_k +\int_{k}^{t} e^{-(t-r)} d\overline{B}^k_r \right) \left( e^{-(s-k)} y_k +\int_{k}^{s} e^{-(s-v)} d\overline{B}^k_v \right) \right)\\
&= e^{-(t-s)} + I\!\!I + I\!\!I\!\!I + I\!V,
\end{align*}
where the first term is due to $\E(y_k)^2=1$ and $e^{-(t-k)} e^{-(s-k)}\le e^{-(t-s)}$, and
$$\begin{aligned}
&I\!\!I =  \E \left( e^{-(t-k)} y_k \int_{k}^{s} e^{-(s-v)} d\overline{B}^k_v \right),\\
&I\!\!I\!\!I = \E \left( e^{-(s-k)} y_k \int_{k}^{t} e^{-(t-r)} d\overline{B}^k_r \right), \\
&I\!V = \E \left( \int_{k}^{t} e^{-(t-r)} d\overline{B}^k_r \int_k^{s} e^{-(s-v)} d\overline{B}^k_v \right),
\end{aligned}$$
for $s \gg k$.
For $I\!\!I$ we use Cauchy-Schwartz inequality and the Lemma \ref{lemma-corr}, 
\begin{align*}
(I\!\!I)^2
&=\E \left( e^{-(t-k)} y_k \int_{k}^{s} e^{-(s-v)} d\overline{B}^k_v \right) ^2
= e^{-2(t-k)} \E \left( \int_{k}^{s} e^{-(s-v)} d\overline{B}^k_v \right)^2  \\
 &\lesssim  e^{-2(t-k)} \int_{k}^{s} \int_{k}^{s}  e^{-(s-v)}e^{-(s-u)} \partial_u \partial_v\; \E\left(\overline{B}^k_u \overline{B}^k_v \right) \,du \,dv\\
&\lesssim   e^{-2(t-k)}  \int_{k}^{s} \int_{k}^{s}  e^{-(s-v)}e^{-(s-u)} (u\wedge v -k)^{2H-2} \, du\,dv.
\end{align*}
We continue with the computation, making use of Lemma \ref{lemma-integral} in the final step: 
\begin{align*}
& \int_{k}^{s} \int_{k}^{s}  e^{-(s-v)}e^{-(s-u)} (u\wedge v -k)^{2H-2} \, du\,dv\\
&\lesssim  2 \int_{k}^{s}  e^{-(s-v)} (v-k)^{2H-2}\left(   \int_{v}^{s} e^{-(s-u)}  \, du\right)\, dv\\
&\lesssim   \int_{k}^{s}  e^{-(s-v)} (s-v)^{2H-2} \, dv\\
&\lesssim (s-k)^{2H-2}.
\end{align*}
Putting them together we have,
$$ I\!\!I \le  e^{-(t-k)}  (s-k)^{H-1}.$$

 Analogous arguments lead to $I\!\!I\!\!I \lesssim (t-k)^{2H-2}$.
Assuming, without loss of generality that  $t\ge s$,

\begin{align*}
I\!V&=\E \left( \int_{k}^{t} e^{-(t-r)} d\overline{B}^k_r \int_k^{s} e^{-(s-v)} d\overline{B}^k_v \right) \\
&= \int_{k}^{t} \int_{k}^{s }  e^{-(t-r)}  e^{-(s-r)} S(r,v) dv dr\\
&\lesssim  e^{-(t-s)}  \int_{k}^{t} \int_{k}^{s }  e^{-(s-r)}  e^{-(s-r)} (r\wedge v- k)^{2H-2} dv dr\\ 
&=2 e^{-(t-s)}  \int_{k}^{t} \int_{r}^{s }  e^{-(s-r)}  e^{-(s-r)} (r\wedge v- k)^{2H-2} dv dr\\
&\lesssim  2 e^{-(t-s)}     \int_k^t  (r-k)^{2H-2} e^{-(s-r)} \int_r^s  e^{-(s-r)}\, du  \, dr\\
&\lesssim  2   \int_k^t  (r-k)^{2H-2} e^{-(t-r)}  \, dr\\
&\lesssim  2    (t- k)^{2H-2}.
\end{align*}
We have again applied Lemma  \ref{lemma-integral} .
Putting everything we and using
\begin{align*}
\E \left( \overline{y}^{k}_t \overline{y}^{k}_s \right) &\leq \Vert \overline{y}^{k}_t \Vert_{L^2} \Vert \overline{y}^{k}_s \Vert_{L^2}
\leq 1,
\end{align*}  
we  obtain
\begin{align*}
\E \left( \overline{y}^{k}_t \overline{y}^{k}_s \right) \lesssim 1 \wedge (t\wedge s -k)^{2H-2}.
\end{align*}
Now recall that $H^*(m)< 0 $ is equivalent to $ (H-1)m +1 < 0$.
Consequently, for $q\ge m$ and $(2H-2)m +2 <0$,
 \begin{align*}
\int_{k-1}^{\infty} \int_{k-1}^{\infty} \E \left( \overline{y}^{k}_t \overline{y}^{k}_s \right)^q 
&\leq C \int_{k-1}^{\infty} \int_{k-1}^{\infty} 1 \wedge (t\wedge s -k)^{(2H-2)q} ds dt\\
&<\infty,
\end{align*}
where $C$ is a constant independent of $q$.  We have reached the conclusion of the proposition.
\end{proof}

\subsection{Kernel convergence of scaling path integrals}\label{kernel-condition}
We prove  Lemma \ref{kernel lemma}. 	In \cite{Taqqu}, instead of $y^{\epsilon}_t$, a moving average of the form 
$$X_t = \int_{\R} p(t-\xi) dW_{\xi}.$$
for a suitable function $p$ was considered and limit theorems (convergence in finite dimensional distributions) were proven for 
$$\epsilon^{H^*(m)} \int_0^{\f t \epsilon} G(X_s) ds.$$
 In this setup,  in order to prove weak convergence one uses the self-similarity of a Wiener process, $\sqrt{\lambda} W_{\f t \lambda} \sim W_t$, leading to weak convergence as this equivalence of course is only in law. Nevertheless, in our case we can write directly, without using self-similarity properties, $y_t^{\epsilon} = \int_{\R} \hat{g} (\f {t-\xi} {\epsilon}) dW_{\xi}$, and thus avoid using self-similarity and, hence, obtain convergence in $L^2$.
In \cite[Theorem 4.7]{Taqqu} using \cite[Lemma 4.5, Lemma 4.6]{Taqqu} the following result was obtained 
$$ \E \left(  \int_{\R^m}  \int_0^t \prod_{i=1}^m p\left( \f {s- \xi} {\epsilon}\right) \epsilon^{H- \f 3 2} ds dW_{\xi} - \f{Z^{H^*(m),m}} {K(H,m)} \right) \to 0,$$
using the Wiener integral representation of the Hermite processes this is equivalent, by multiple Wiener-Ito isometry, to
$$ \int_{\R^m} \left( \int_0^t \prod_{i=1}^m p\left( \f {s- \xi_i} {\epsilon}\right) \epsilon^{H - \f 3 2} ds - \int_0^t \prod_{i=1}^m \left( s- \xi_i \right)_{+}^{H- \f 3 2} ds \right)^2 d\xi_1 \dots d\xi_m.$$
To apply \cite[Theorem 4.7]{Taqqu} we rewrite our kernels in the above moving average form.
For the rescaled fractional Ornstein-Uhlenbeck process we obtain by above computation,
\begin{align*}
y^{\epsilon}_t
&=  \f {1} {c_1(H)} \epsilon^{-H} \int_{\R} \int_{-\infty}^t e^{- \f {t-u} {\epsilon} }  (u-s)_{+}^{H- \f 3 2} du dW_s\\
&= \f {1} {c_1(H)} \epsilon^{-H} \int_{\R} e^{- \f {t-s} {\epsilon}} \int_{-\infty}^{t-s} e^{\f v {\epsilon}} v_{+}^{H- \f 3 2}  dv dW_s\\
&= \f {1} {c_1(H)}  \epsilon^{-\f 1 2} \int_{- \infty}^{t} e^{- \f {t-s} {\epsilon}} \int_{0}^{\f {t-s} {\epsilon}} e^{v} v_{+}^{H- \f 3 2} dv dW_s\\
&=     \epsilon^{-\f 1 2}\int_{- \infty}^{t} g\left( \f{t-s} {\epsilon} \right) dW_s,
\end{align*}
where
\begin{equation}\label{g} g(s)=  \f 1 {c_1(H)}  e^{-s} \int_0^s e^u u_{+}^{H- \f 3 2} du.
\end{equation}
To apply \cite[Theorem 4.7]{Taqqu} we will verify the conditions for this $g$.

Now, the term we consider  in (\ref{explicit-kernel})  has the following form,
$$ \epsilon^{H^*(m)-1 }\int_{\R^m} \int_0^t \prod_{i=1}^m g\left(\f{s-\xi_i} {\epsilon}\right) \epsilon^{- \f m 2} ds dW_{\xi_1} \dots dW_{\xi_m}.$$
Using $H^*(m) = (H-1)m +1$ we obtain that this is equal to 
$$\int_{\R^m} \int_0^t \prod_{i=1}^m g\left(\f{s-\xi_i} {\epsilon}\right) \epsilon^{H- \f 3 2} ds dW_{\xi_1} \dots dW_{\xi_m},$$
the required Lemma \ref{explicit-kernel} will follow from the reformulated convergence below.
\begin{lemma}
	Let $\lambda$ denote the Lebesgue measure, then 
	\begin{equation}\label{explicit-kernel-2}
	\left \Vert \int_0^t  \prod_{i=1}^m g\left(\f{s-\xi_i} {\epsilon}\right) \epsilon^{H- \f 3 2} ds  -   \int_0^t \prod_{i=1}^m (s- \xi_i)_+^{H-\f 32}ds\right \Vert_{L^2(\R^m,\lambda)} \to 0. \end{equation}
\end{lemma}

\begin{proof}
We are now in the above framework and it is only left the check the conditions imposed on the functions $p$ in \cite[Theorem 4.7]{Taqqu}.
Examining Taqqu's proof,  we note that in fact the $L^2$ convergence of (\ref{explicit-kernel-2})
is obtained under the following conditions.
\begin{itemize}
	\item[{1}] $\int_{\R} p(s)^2 ds < \infty$.
	\item[{2}] $ \vert p(s) \vert \leq C s^{H- \f 3 2} L(u)$ for almost all $s>0$.
	\item[{3}] $p(s) \sim s^{H - \f 3 2 } L(s)$ as $s \to \infty$.
	\item[{4}] There exists a constant $\gamma$ such that $0<\gamma< (1-H)\wedge  (H- (1-\f 1 {2m}))$ such that $\int_{-\infty}^0 \vert p(s) g(xy+s)\vert ds = o(x^{2H-2} L^2(x)) y^{2H-2-2\gamma}$ as $x \to \infty$ uniformly in $y \in (0,t]$.
\end{itemize}
where $L$ denotes a slowly varying function (for every $\lambda >0$ $\lim_{x \to \infty} \frac{L(\lambda x)}{L(x)} ) =1$).
Now we go ahead and  show that $g$ defined by (\ref{g})  satisfies these conditions, to increase readability we suppress the constant $\f 1 {c_1(H)}$ in the computations.
For  $s<1$,
\begin{align*}
e^{-s} \int_0^s e^u u^{H- \f 3 2 } du  &\leq   \int_0^s u^{H- \f 3 2 } du
\lesssim  s^{H - \f 1 2}.
\end{align*}
We calculate for $s>1$ via integration by parts
\begin{align*}
e^{-s} \int_0^s e^u u^{H- \f 3 2 } du  &\leq 
e^{-s} \int_0^1 e^u u^{H- \f 3 2 } du + e^{-s} \int_1^s e^u u^{H- \f 3 2 } du \\
&\lesssim  e^{-s} +s^{H- \f 3 2} -1+  e^{-s} \int_1^s e^u u^{H- \f 5 2 } du\\
&\lesssim   s^{H- \f 3 2}.
\end{align*}

This of course implies that $g$ is $L^2$ integrable.
Finally observe that $$\int_{-\infty}^0 \vert g(s) g(xy+s)\vert ds = 0$$ as $g(s)=0$ for $s<0$.
With these we apply  \cite[Theorem 4.7]{Taqqu} to conclude the $L^2$ convergence of the kernels.
\end{proof}

\subsection{Joint convergence by asymptotic independence}
If a sequence of  random variable $x_n$ converges to $x$  and another sequence $y_n$ converges to $y$, does $(x_n, y_n)$
converge jointly and what is the correlation of the limit. When both $x_n$ and $y_n$ are Gaussian sequences, they converge jointly 
to $(x,y)$ where $x$ and $y$ are taken to be independent, provided the correlation between $x_n$ and $y_n$ converges to zero.
In \cite{Nourdin-Rosinski} these are generalised to moment determinant random variables as
limits, in which case it is sufficient to show $\cov (x_n^2, y_n^2)\to 0$. By Torsten Carleman's theorem,
a real valued random variable is moment determinant  if its $p$-moment grows no faster than 
$(\f n C)^n (\log n)^n$ (e.g. if it has exponentially decaying density).
An $L^2$  random variable in the second chaos has exponential tails, beyond the second chaos, the tails grow
decays slower than exponentially.
Extensions to high order chaos were obtained in  \cite{Nourdin-Nualart-Peccati}. Neither of these results are sufficient for the need in this paper, we therefore
present a generalisation, which can be easily deduced from the reasoning in \cite{Nourdin-Nualart-Peccati}.
 The result in \cite{Bai-Taqqu} has moment determinant limit as restrictions, has restrictions to Wiener chaos $1$ and $2$, i.e. corresponding to BM, fBM and Rosenblatt processes in our case.

These results benefits from three insights. The first being 
the characterisation, of \"{U}st\"{u}nel-Zaksi \cite{Ustunel-Zakai} for independence of two iterated integrals  with vanishing of contractions of their integral kernels.
The second is the characterisation of the vanishing of the contractions of their integral kernels by covariances of their squares
the third is the characterisation of independence of random variables by the covariance of their squares  \cite{Nualart-Peccati, Nourdin-Nualart-Peccati}.

Denote by $I_p(f)$ the $p^{th}$  iterated It\^o-Wiener integral
$$I_p(f)= p!\int_{-\infty}^{\infty} \int_{-\infty}^{s_{p-1}}\dots \int_{-\infty}^{s_2} f(s_1, \dots, s_p)dW_{s_1} dW_{s_2} \dots dW_{s_p}.$$
We take the $L^2$ function $f$ to be symmetric functions of appropriate number of variables throughout.
Given  $f\in L^2(\R^p)$ and $g\in L^2(\R^q)$,  where $p, q\ge 1$,
 their is
$$f\otimes_1 g=\int_{\R} f(x_1, \dots, x_{p-1}, s) g(y_1, \dots, y_{q-1}, s) ds.$$
Similarly 
$$f\otimes_r g=\int_{\R^r} f(x_1, \dots, x_{p-r}, s_1, \dots, s_r) g(y_1, \dots, y_{q-r}, s_1, \dots, s_r) ds_1 \dots ds_r.$$
If $f\otimes_1g=0$, so do all higher order contractions.

By \cite[Thm.6]{Ustunel-Zakai}  two  integrals  $I_p(f)$ and $I_q(g)$ are  independent, 
 if and only if the 1-contraction between $f,g$ vanishes almost surely. The necessity comes from 
 the product formula,
 $$
 I_p(f)I_q(g) =\sum_{m=1}^{p\wedge q} \f {p!q!}{m!(p-m)!(q-m)!} I_{p+q-2m} (f\otimes_m g).$$
 and the independence: $\E(I_p(f)I_q(g))^2=p!q!\|f\otimes g|_{L^2}$. All terms in the binomial expansion for the product  drop to zero
 except for the $m=0$ term. 
 The following asymptotic independent result is proven in \cite[Thm. 3.1] {Nourdin-Nualart-Peccati}.
  \begin{lemma}\cite{Nourdin-Nualart-Peccati} \label{contraction-covariance}
Given $F_{\epsilon}=I_p(f^{\epsilon})$ and $G_{\epsilon}=I_q(g^{\epsilon})$, then
$$\cov({F^{\epsilon}}^2,{G_{\epsilon}}^2) \to 0$$
is equivalent to 
$$ \Vert f^{\epsilon} \otimes_r g^{\epsilon}\Vert \to 0,$$
for $1\leq r \leq p \wedge q$.
\end{lemma}

It is also observed in  two  integrals  $I_p(f)$ and $I_q(g)$ are  independent, 
their Malliavin derivative begin orthogonal. This explain why Malliavin calculus  comes into prominent play,
which has been developed to its perfection in \cite[Lemma 3.2]{Nourdin-Nualart-Peccati}. 
The space of test functions is taken to be $C^\infty_q:=C^{\infty}\cap BC^{q-1}(\R^m)$, which is sufficient to approximiate
indicator functions of any measurable sets. We also set for $\phi \in C^\infty_q$,
$$ \Vert \phi \Vert_{q} = \Vert \phi \Vert_{\infty} + \sum_{\vert k \vert =1}^{q} \left\Vert \f { \partial^{k}} {\partial^k x} \right\Vert_{\infty},$$
where the sum runs over multi-indices $k=(k_1, \dots, k_m)$.
Let $L=-\delta D$.
\begin{lemma}\cite{Nourdin-Nualart-Peccati}
\label{a.i.-key-ineq}
Let $\theta \in C_q^\infty(\R^m)$, 
$G=I_q(g)$, $F=(F_1, \dots, F_m)$, where  $F_i= I_{p_i}(f_i)$ and $\E F_i^2 = 1$, where $p_i\ge q$.  Then 
$$ \E \vert \< (I-L)^{-1} \theta(F)DF_j,DG\>_{\H} \vert \leq c\,\Vert \theta \Vert_{q-1} \cov(F_j^2,G^2), $$
where $c$ is a constant depending on $\Vert F \Vert_{L^2}$ and $\Vert G \Vert_{L^2} $. 
\end{lemma} 
Throughout this section $f_i:\R^{p_i}\to \R$, $g:\R^q\to \R$  are symmetric functions.

The final piece of the puzzle is the observation that
 the defect in being independent is quantitatively controlled by the covariance of the squares of the relative components. The following is from
 \cite{Nourdin-Nualart-Peccati}, our only modification is to take $G$ to be vector valued.
Let $g_i:\R^{q_i}\to \R$  be symmetric functions.

\begin{lemma}\label{nnp-extension}
Let  $F=\left(I_{p_1}(f_1), \dots I_{p_m}(f_m)\right)$ and $G=\left(I_{q_1}(g_1), \dots, I_{q_n}(g_n)\right)$ such that $p_k \geq q_l$ for every pair of $k,l$.
Then for every $\phi \in \C^{\infty}_q(\R^m)$, $\psi \in \C^{\infty}_1(\R^n)$, the following holds for  some constant $c$, depending on $ \Vert F\Vert_{L^2}$, $\Vert G \Vert_{L^2}$:
$$  \E \left( \phi(F) \psi(G) \right) -\E \left( \phi(F) \right) \E \left( \psi(G) \right) \leq c  \Vert D\psi \Vert_{\infty} \Vert \phi \Vert_q \sum_{i=1}^{m} \sum_{j=1}^n \cov(F_i^2,G_j^2)$$
\end{lemma}
\begin{proof}   Define  $L^{-1} (\sum_{k=0}^\infty I_k(h_m))=\sum_{k=1}^\infty \f 1 k I_k(h_m) \in {\mathbb D}^{2,2}$.
The key equality is  $-DL^{-1}=(I-L)^{-1}D$. As in \cite{Nourdin-Nualart-Peccati},
\begin{align*}
\phi(F) - \E(\phi(F)) &= LL^{-1} \phi(F)= \sum_{j=1}^m \delta((I-L)^{-1} \partial_j \phi(F) DF_j)
\end{align*}
Multiply both sides by $\psi(G)$ and  use integration by parts we see
\begin{align*}
&\E \left( \phi(F) \psi(G) \right) -\E \left( \phi(F) \right) \E \left( \psi(G) \right) \\
&= \sum_{j=1}^m \sum_{i=1}^n \E \left( \< (I-L)^{-1} \partial_j \phi(F) DF_j,  D  G_i \>_{\H} \partial_i\psi(G) \right)\\
&\leq \Vert D\psi \Vert_{\infty} \sum_{j=1}^m \sum_{i=1}^n \left| \E \left( \< (I-L)^{-1} \partial_j \phi(F) DF_j, D G_i \>_{\H}  \right)\right|.
\end{align*}
To conclude, apply to each summand  Lemma \ref{a.i.-key-ineq} with $\theta = \partial_j \phi$ and $G=G_i$.
\end{proof}

\begin{lemma}\label{expectation-split}
Let $F_{\epsilon}=\left(I_{p_1}(f^{\epsilon}_1), \dots I_{p_m}(f^{\epsilon}_m)\right)$ and $G_{\epsilon}=\left(I_{q_1}(g^{\epsilon}_1), \dots, I_{q_n}(g^{\epsilon}_n)\right)$ with  $q_1 \leq q_2, \dots q_m \leq p_1 \leq p_2, \dots p_m$. Then for every $i \leq m ,j \leq n$,
$$ \Vert f^{\epsilon}_j \otimes_r g^{\epsilon}_i \Vert \to 0, \quad 1\le r \le p_j\wedge q_i$$
implies that  for any $\phi \in \C^{\infty}(R^m)_{p_m}$ $\psi \in \C^{\infty}_{q_n}$,
$$ \E \left( \psi(F_{\epsilon}) \psi(G_{\epsilon}) \right) - \E \left( \psi(F_{\epsilon}) \right) \E \left( \psi(G_{\epsilon})\right) \to 0.$$
\end{lemma}
\begin{proof}
Just combine Lemma \ref{nnp-extension} and Lemma \ref{contraction-covariance}.
\end{proof}
The following generalises results from \cite{Nourdin-Nualart-Peccati}.
\begin{proposition}\label{proposition-spit-independence}
Given $F_{\epsilon}=\left(I_{p_1}(f^{\epsilon}_1), \dots , I_{p_m}(f^{\epsilon}_m)\right)$ and  $G_{\epsilon}=\left(I_{q_1}(g^{\epsilon}_1), \dots, I_{q_n}(g^{\epsilon}_n)\right)$ such that $q_1 \leq q_2, \dots q_m, \leq p_1 \leq p_2, \dots \le p_m$ and such that
$$ \Vert f^{\epsilon}_j \otimes_r g^{\epsilon}_i \Vert \to 0.$$
If $F_\epsilon \to U$ and $G_{\epsilon} \to V$ weakly, then $(F_{\epsilon},G_{\epsilon}) \to (U,V)$ jointly where  $U,V$ are taken to be independent.
\end{proposition}
\begin{proof} 
Since $(F_{\epsilon},G_{\epsilon})$ is bounded in $L^2$ it is tight. Now choose a weakly converging subsequence $(F_n,G_n)$ with limit denoted by  $(X,Y)$. 
Let $\phi \in \C_{p_m}^{\infty}(R^m)$ $\psi \in \C_{q_n}^{\infty}(\R^n)$. Then
By  Lemma \ref{expectation-split}  and the bounds on $\phi, \psi$, we pass to the limit under the expectation sign and obtain 
$$\E\left( \phi(X) \psi(Y) \right)=\E\left( \phi(X)  \right) \E\left( \psi(Y) \right).$$
Thus every  limit measure is the product measure determined by $U,V$ and hence the $(F_{\epsilon},G_{\epsilon})$ converges as claimed.
\end{proof}

\noindent{ \bf {Acknowledgement:}} We would like to thank Martin Hairer for helpful discussions.

%\bibliographystyle{alpha}
%\bibliography{./fbm-refs}

\newcommand{\etalchar}[1]{$^{#1}$}

\end{document}